\documentclass[a4paper,10pt]{amsart}
\usepackage[bf,wide,cthms]{gewoehn} 

\newcommand{\E}{\mathbf{E}} \newcommand{\ty}[1]{\text{\small\ensuremath{#1}}}
 \renewcommand{\Hh}{\ensuremath{H}}

\begin{document}

\title[Homotopy type of spaces of curves on flat surfaces]{Homotopy type of
spaces of curves with constrained curvature on flat surfaces} \author{Nicolau
C.~Saldanha} \author{Pedro Z\"{u}hlke}
 \maketitle

\begin{abstract} Let $S$ be a complete flat surface, such as the Euclidean
	plane. We determine the homeomorphism class of the space of all curves on
	$S$ which start and end at given points in given directions and whose
	curvatures are constrained to lie in a given open interval, in terms of all
	parameters involved. Any connected component of such a space is either
	contractible or homotopy equivalent to an $n$-sphere, and every integer 
	$n\geq 1$ is realizable. Explicit homotopy equivalences between the
	components and the corresponding spheres are constructed.  
\end{abstract}



\setcounter{section}{-1} \section{Introduction}\label{S:introduction}

Let $-\infty\leq \ka_1<\ka_2\leq +\infty$ and $Q=(q,z)\in \R^2\times \Ss^1$.
Let $\sr C_{\kappa_1}^{\kappa_2}(Q)$ denote the set, furnished with the $C^r$
topology for some $r\geq 2$, of all regular curves $\ga\colon [0,1]\to \R^2 $ of
class $C^r$ such that: 
	\begin{enumerate} 
		\item [(i)] $\ga$ starts at $0\in \R^2$ in the
		direction of $1\in \Ss^1$ and ends at $q$ in the direction of $z$; \item
			[(ii)]  The curvature $\ka_\ga$ of $\ga$ satisfies
			$\ka_1<\ka_\ga(t)<\ka_2$ for all $t\in [0,1]$.  
	\end{enumerate}
	A more accurate reformulation of (i) is that $\ga(0)=0$, 
	$\ta_\ga(0)=1$, $\ga(1)=q$ and $\ta_\ga(1)=z$, where $\ta_\ga\colon [0,1]\to
	\Ss^1$ denotes the unit tangent to $\ga$. 

	There is a natural decomposition of $\sr C_{\kappa_1}^{\kappa_2}(Q)$  as the
	disjoint union of its subspaces $\sr C_{\kappa_1}^{\kappa_2}(Q;\theta_1)$,
	where the latter contains those curves which have total turning $\theta_1$,
	for $e^{i\theta_1}=z$. By Theorems 4.19 and 7.1 in \cite{SalZueh1}, each of
	these subspaces is either empty or a contractible connected component of
	$\sr C_{\kappa_1}^{\kappa_2}(Q)$, except when $\ka_1$, $\ka_2$ have opposite
	signs and
	$\abs{\theta_1}<\pi$. To study what happens in this case, it may be assumed
	without loss of generality that $\ka_1=-1$ and $\ka_2=+1$, by Theorem 2.4 in
	\cite{SalZueh1}. For a fixed $Q=(q,z)$ with $z\neq -1$, there exists exactly
	one subspace $ \sr C_{-1}^{+1}(Q;\theta_1) $ with $ \theta_1\in (-\pi,\pi)
	$; it contains the curves in $\sr C_{-1}^{+1}(Q)$ of minimal total turning
	in absolute value. Let it be denoted by $\sr M(Q)$.

The central result of this work states that $\sr M(Q)$ is homotopy equivalent to
$\Ss^n$ for some $n\in \se{0,1,\dots,\infty}$, and allows one to determine $n$
by means of a simple construction (recall that $\Ss^\infty$ is 
contractible). 
In particular, any of the indicated values is
possible.  

In the sequel $ \R^2 $ is identified with $ \C $ for convenience. Also, $\E$
denotes the separable Hilbert space, $C_r(a)$ denotes the circle of radius $r>0$
centered at $a\in \C$ and  $X\home Y$ (resp.~$X\iso Y$) means that $X$ is
homeomorphic (resp.~homotopy equivalent) to $Y$. 

\begin{figure} \begin{center} \includegraphics[scale=.316]{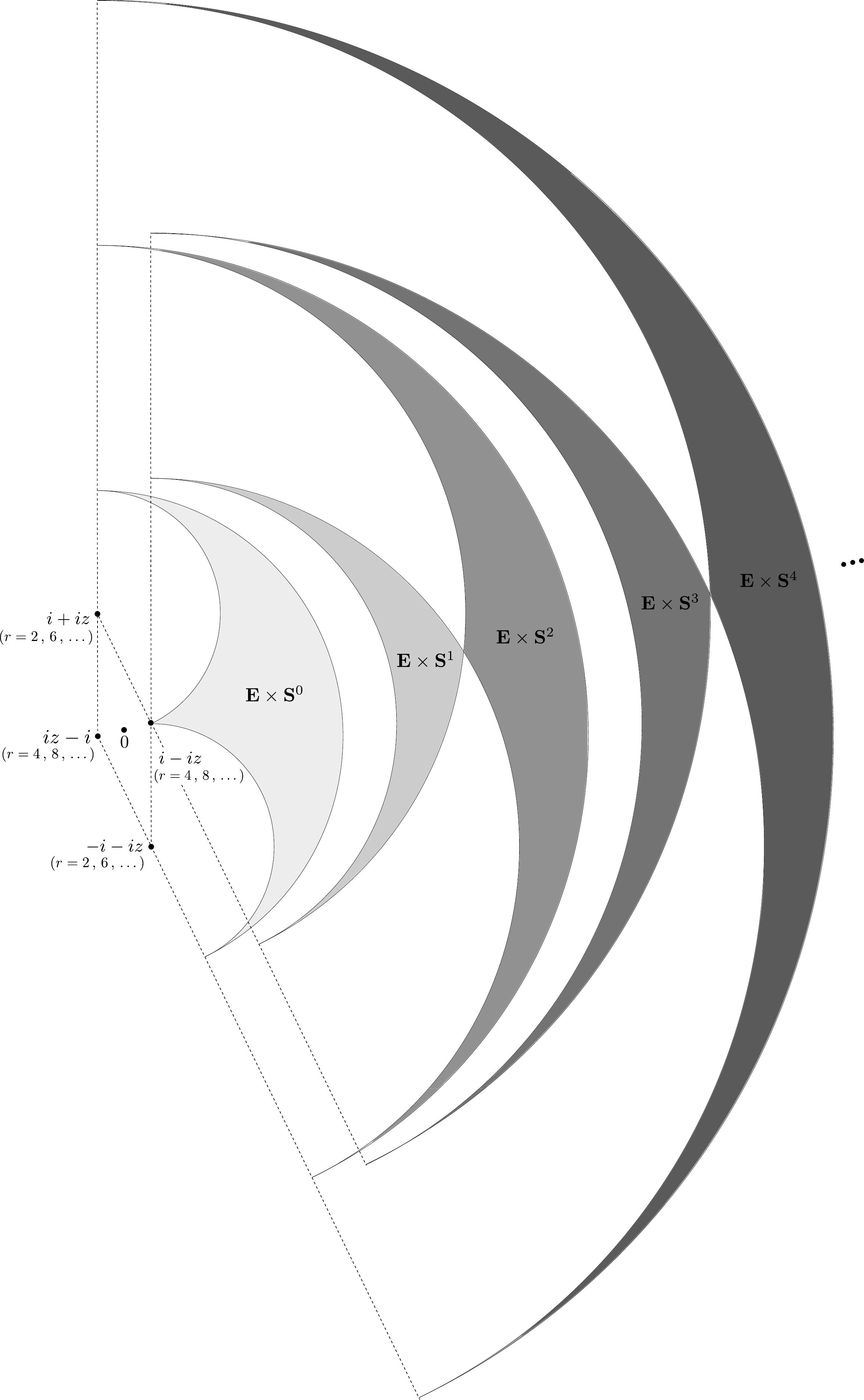}
		\caption{This drawing to scale indicates the homeomorphism class of $\sr
			M(Q)$ in terms of $q$, for $Q=(q,z)\in \C\times \Ss^1$ and a fixed
			$z\neq-1$ (here $z\approx \exp(\frac{i\pi}{7})$). If $q$ lies in the
			unshaded region, then $\sr M(Q)\home \E$, the separable Hilbert space.
			The line segments are only auxiliary elements and do not bound any
			regions. The line through 0 and $1+z$ (not drawn) contains $\pm
		(i-iz)$ and is an axis of symmetry of the figure. The radii of the
	circles are indicated inside parentheses near their centers.  }
	\label{F:shadesofgrey} \end{center} \end{figure}

\begin{uthm}
	Let $Q=(q,z)\in \C\times \Ss^1$, $z\neq -1$. Then
	$\sr M(Q)\home \E\times \Ss^{2k} \text{ or }\E\times \Ss^{2k+1}\ (k\geq 0)$
	for $q$ in the open region intersecting the ray from 0 through $1+z$ and
	bounded by the three circles \begin{equation*} \begin{cases}
			C_{4k+4}(iz-i)\text{ \,and \,}C_{4k+2}(\pm (i+iz)),\text{\ \ or}\\
			C_{4k+4}(i-iz)\text{ \,and \,}C_{4k+6}(\pm(i+iz)),\text{\ \
			respectively} \end{cases} (\text{see \fref{F:shadesofgrey}}).
	\end{equation*} If $q$ does not lie in
	the closure of any of these regions, then $\sr M(Q)\home \E$. If $q$ lies on
	the boundary of one of them, then $\sr M(Q)\home \sr
	M\big((q-\de(1+z),z)\big)$ for all sufficiently small $\de>0$.  \end{uthm}

\begin{urem}
	Let $S_n$ denote the set of all $(q,\theta_1)\in \C\times
	\R$ such that $ \sr C_{-1}^{+1}(Q;\theta_1) $ is
	homotopy equivalent to $\Ss^n$, where $Q=(q,e^{i\theta_1})$ and 
	$n\in \se{0,1,\dots,\infty}$. 
	Together with the aforementioned results of \cite{SalZueh1}, 
	the theorem implies that $S_n$ is a bounded subset of 
	$\C\times (-\pi,\pi)$, neither open nor closed but having nonempty interior,
	for any finite $n$.
	Moreover, if
	\begin{equation*}
		S_n(z)=\set{q\in \C}{\sr M(Q)\iso \Ss^n,~Q=(q,z)},
	\end{equation*}
	then $\lim_{n\to
	\infty}\tup{Area\?}(S_n(z))=+\infty=\tup{Area\?}(S_\infty(z))$ for any
	$z\in \Ss^1\ssm\se{-1}$, as suggested by Figure \ref{F:shadesofgrey}. The precise determination of $\tup{Area\?}(S_n(z))$
	in terms of $n$ and $z$ will be left as an exercise.  \end{urem}

\begin{uexm}\label{E:horizontal} Let $Q_x=(x,1)\in \R\times \Ss^1$. Then $\sr
M(Q_x)\home \E$ if $x\leq 0$ and \begin{equation*}
\sr M(Q_x)\home
\begin{cases} \E\times \Ss^{2k} & \text{\ if\ \ \ $\frac{x}{4}\in
\big(\sqrt{k^2+k}\,,\,k+1\big]$} \\ \E\times \Ss^{2k+1} & \text{\ if\ \ \
$\frac{x}{4}\in \big(k+1\,,\,\sqrt{k^2+3k+2}\?\big]$} \end{cases}\quad (k\in
\N).  \end{equation*} Note that the size of the interval where $\sr M(Q_x) \home
\E\times \Ss^n$ approaches $2$ as $n$ increases.
	\begin{figure}[ht] \begin{center} \includegraphics[scale=.61]{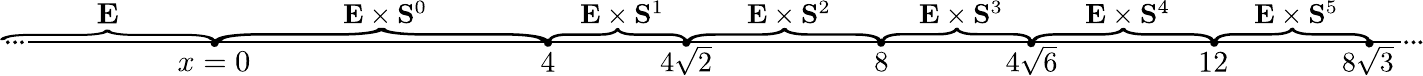}
			\caption{The homeomorphism class of $\sr M(Q_x)$ as a function of
			$x\in \R$.}
		\label{F:horizontal} \end{center} \end{figure} \end{uexm}

	The following concepts are essential to all that follows.

	\begin{defn}[condensed, critical and diffuse]
		\label{D:basic} Let $\ga\colon [0,1]\to \C$ be a  regular curve,
	$\ta_\ga\colon [0,1]\to \Ss^1$ its unit tangent vector, and
	$\theta_\ga\colon [0,1]\to \R$ a continuous \tdef{argument} function for
	$\ta_\ga$, that is, one satisfying $\ta_\ga=\exp(i\theta_\ga)$. We call $\ga$ 
	\tdef{condensed}, \tdef{critical} or \tdef{diffuse} according as its
	\tdef{amplitude} \begin{equation*}
		\om=\sup_{t\in
		[0,1]}\theta_\ga(t)-\inf_{t\in [0,1]}\theta_\ga(t) \end{equation*}
	satisfies $\om<\pi$, $\om=\pi$ or  $\om>\pi$, respectively. The open set of all
	condensed (resp.~diffuse) curves in $\sr M(Q)$ shall be denoted by $\sr U_c$
	(resp.~$\sr U_d$). A \tdef{sign string} $\sig$ is an alternating finite
	sequence of signs, such as \ty{+-+} or \ty{-+-+}. Its \tdef{length}
	$\abs{\sig}$, the number of terms in the string, is required to satisfy
	$\abs{\sig}\geq 2$, and $\sig(k)$ denotes its $k$-th term ($1\leq k\leq
	\abs{\sig}$). Its \tdef{opposite} $-\sig$ is the sign string satisfying
	$\abs{-\sig}=\abs{\sig}$ and $(-\sig)(k)=-\sig(k)$.  A critical curve $\ga$
	is \tdef{of type $\sig$} if there exist $0\leq
	t_1<t_2<\dots<t_{\abs{\sig}}\leq 1$ with $\theta_\ga(t_k)=\sup \theta_\ga$
	or $\inf \theta_\ga$ according as $\sig(k)=\ty{+}$ or $\ty{-}$, but it is
	impossible to find $0\leq s_1<\dots<s_{\abs{\sig}+1}\leq 1$ such that
	$\ta_\ga(s_{k+1})=-\ta_\ga(s_{k})$ for each $k=1,\dots,\abs{\sig}$.
\end{defn}

\begin{uexm} Suppose that $\sr M(Q)\iso \Ss^1$. Then a generator of $\pi_1\sr
M(Q)$ is represented by any family of curves $\ga_s\in \sr M(Q)$ $\tup(s\in
[0,1]$\tup) such that: \begin{enumerate} \item [(i)] $\ga_s$ is condensed for
			$s\in [0,\tfrac{1}{4})\cup (\tfrac{3}{4},1]$ and $\ga_0=\ga_1$;
		\item [(ii)]  $\ga_s$ is diffuse for $s\in (\tfrac{1}{4},\tfrac{3}{4})$;
		\item [(iii)] $\ga_s$ is critical of type $\ty{+-}$ when
			$s=\tfrac{1}{4}$ and critical of type $\ty{-+}$ when
			$s=\tfrac{3}{4}$.  \end{enumerate} As $\pi_k\sr M(Q)=0$ for $k>1$,
	the resulting map $\Ss^1\to \sr M(Q)$ is actually a weak homotopy
	equivalence, and hence a homotopy equivalence, since $\sr M(Q)$ is a Banach
	manifold (cf.~Theorem 15 of \cite{Palais}).
	
	In particular, suppose that $4<x\leq 4\sqrt{2}$ and let $Q_x=(x,1)\in
	\C\times \Ss^1$, as in the preceding example. A generator of $\pi_1\sr
	M(Q_x)$ can be visualized by completing Figure \ref{F:generator} to obtain a
	family $\ga_s\in \sr M(Q_x)$ as above. For $s=\frac{1}{2}$ one may take the
	concatenation of a figure eight curve (that is, a curve of total turning 0, 
	not drawn in the figure) with $\ga_0=\ga_1$, where the latter
	denotes the straight segment from $0$ to $x$. Of course, it needs to be
	checked that this homotopy can actually be carried out within $ \sr M(Q_x) $.
	This is the case if and only if  $x>4$; this was originally proved as
	Theorem 5.3 in \cite{Dubins1} and then generalized in Theorem 6.1 of 
	\cite{SalZueh1}. 
	
	A generator of $\pi_n\sr M(Q)$
	when $\sr M(Q)\iso \Ss^n$ is constructed in \S\ref{S:generator}, and an
	informal description is given at the end of this introduction.

\end{uexm}

\begin{figure}[ht] \begin{center} \includegraphics[scale=.84]{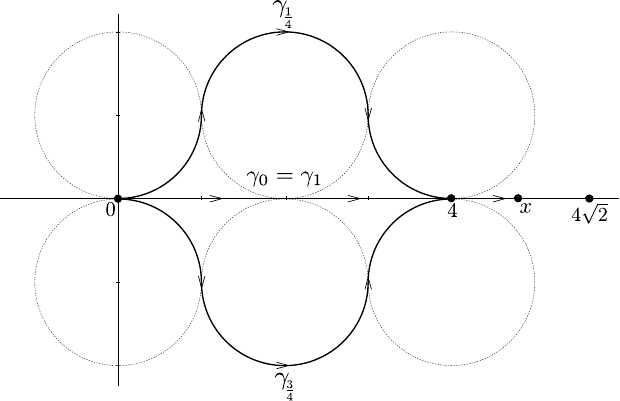}
		\caption{Constructing a generator of $\pi_1\sr M(Q)$ when $Q=(x,1)\in
			\R\times \Ss^1$ and $4<x\leq 4\sqrt{2}$.}
		\label{F:generator} \end{center} \end{figure}

Let $S$ be a complete flat surface, $\ka_1<\ka_2$ and $u,v\in UTS$, the unit
tangent bundle of $S$; throughout the article, $UT\C$ is identified with
$\C\times \Ss^1$. Let $\sr CS_{\ka_1}^{\ka_2}(u,v)$ denote the space 
of all curves on $S$ whose lift to $UTS$ joins $u$ to $v$ and whose geodesic 
curvature takes values in $(\ka_1,\ka_2)$, with the implicit convention that 
$\ka_2=-\ka_1>0$ if $S$ is nonorientable 
(for the formal definition, see \S8 of \cite{SalZueh1}).

When $\ka_1$ and $\ka_2 $ have opposite signs, the homeomorphism class of $\sr
CS_{\ka_1}^{\ka_2}(u,v)$ can be determined from the theorem as follows, provided
only that a description of $S$ as a
quotient of $\C$ by a group of isometries is known. 
If the coordinates of $\C$ are chosen (as they may be) so that the vector $1\in
\Ss^1$ at $0\in \C$ projects to $u$ under the induced map $\pr \colon UT\C\to
UTS$, then by Proposition 8.3 in \cite{SalZueh1}
\begin{equation}\label{E:decomp} \sr
	CS_{\kappa_1}^{\kappa_2}(u,v) \home \coprod_{Q\in \pr^{-1}(v)} \sr
	C_{\kappa_1}^{\kappa_2}(Q),  
\end{equation} 
where the inverted product denotes disjoint union. Moreover, there is a homeomorphism $h\colon \C\times
\Ss^1\to \C\times \Ss^1$ (depending on $\ka_1,\ka_2$) such that $\sr
C_{\ka_1}^{\ka_2}(Q)\home \sr C_{-1}^{+1}(h(Q))$ for all $Q\in UT\C$. A 
simple expression for $ h $ is 
available from Theorem 2.4 in \cite{SalZueh1}; hence, the homeomorphism 
class of $\sr CS_{\ka_1}^{\ka_2}(u,v)$ can actually be computed 
explicitly.

The homeomorphism in \eqref{E:decomp} always holds, but for $ \ka_1\ka_2\geq 0 $,
each of the spaces $\sr C_{\kappa_1}^{\kappa_2}(Q)$ appearing on the
right side decomposes as a union of infinitely many components homeomorphic to
$\E$, as shown in Theorem 7.1 of \cite{SalZueh1}. This case is thus not as
interesting as the one where $ \ka_1\ka_2< 0 $. (To determine the sign
of $ \ka_1\ka_2 $, we set $ 0\?(\pm\infty)
=0$ by convention.)

\begin{ucor}\label{C:complete} Let $S$ be a complete flat surface, $\ka_1<\ka_2$
	and $u,v\in UTS$. Then each component of $\sr CS_{\ka_1}^{\ka_2}(u,v)$ is
	homeomorphic to $\E\times \Ss^n$, for some $n\in \se{1,\dots,\infty}$
	depending upon the component. The number of components homeomorphic to
	$\E\times \Ss^n$ is finite for $n<\infty$, and infinite for $n=\infty$.
\end{ucor} 
\begin{proof} Only the last assertion for the case $\ka_1\ka_2<0$ still needs 
	to be justified. 
	For this, consider the decomposition \eqref{E:decomp}. The existence of the
	homeomorphism $h\colon UT\C\to UT\C$ such that $\sr
	C_{\ka_1}^{\ka_2}(Q)\home \sr C_{-1}^{+1}(h(Q))$ for all $Q\in UT\C$ shows 
	that it may be assumed that $\ka_1=-1$, $\ka_2=+1$.
	
	Now for any $Q\in \pr^{-1}(v)$, each of the infinitely many components of 
	$\sr C_{-1}^{+1}(Q)$ is contractible except possibly one, namely, $\sr M(Q)$.
	Write $S=\C/\Ga$,
	for some group $\Ga$ of isometries of $\C$. By proper discontinuity of the
	action of $\Ga$, for each $n<\infty$, the intersection of $\pr^{-1}(v)\subs
	UT\C$ with \begin{equation*}
	\set{Q\in UT\C}{
			\sr M(Q)\home \E\times \Ss^n} \end{equation*} is finite,
		since the latter is a bounded subset of $UT\C$ for all finite $n$.
	\end{proof}

\begin{uexm} For $a\in \C$, denote by $T_a\colon \C\to \C$ the translation
	$x\mapsto x+a$. Let $S=\C/\Ga$ be a flat torus, where $\Ga$ is the group
	$\gen{T_1,T_i}$, and let $u\in UTS$ be arbitrary. Then \tsl{for every $n\in
		\se{1,2,\dots,\infty}$}, there exists a connected component of $\sr
	CS_{-1}^{+1}(u,u)$ homeomorphic to $\E\times \Ss^n$. Since $S$ is
	isotropic, we may assume that $u=\pr(O)$, where 
	$O=(0,1)\in \C\times \Ss^1$. Then according to
	\eqref{E:decomp}, $\sr CS_{-1}^{+1}(u,u)$ contains homeomorphic copies of
	$\sr M(Q_k)$ for every $k\in \Z$, where $Q_k=(k,1)\in \C\times \Ss^1$ is as
	in the first example. The same conclusion holds even if $T_1$ and $T_i$ are
	replaced by  $T_{2}$ and $T_{2i}$, because the lattice $2(\Z\times \Z)$
	intersects the interior of each of the shaded regions in the analogue of
	Figure \ref{F:shadesofgrey} for $z=1$. In contrast, for
	$S=\C/\gen{T_{4},T_{4i}}$, no connected component of $\sr CS_{-1}^{+1}(u,u)$
	is homotopy equivalent to $\Ss^1$. This illustrates the general fact that
	the topology of $\sr CS_{\ka_1}^{\ka_2}(u,v)$ is closely linked to the
	global geometry of $S$.  \end{uexm}

\subsection*{Outline of the proof} Although the proof of the main theorem is somewhat
technical, the underlying idea is quite simple.  
For each sign string $\sig$, we define the concept of ``quasicritical
curves of type $\sig$''. These form an open set $\sr U_\sig \subs \sr M(Q)$
containing all critical curves of type $\sig$ in $\sr M(Q)$, with $\sr
U_\sig=\emptyset$  if there exists no curve of the latter type. The naive plan
is to prove that $\sr U_c$, $\sr U_d$ and the $\sr U_\sig$ (for $\sig$ ranging
over all possible sign strings) form a good cover of $\sr M(Q)$, meaning that
their $k$-fold intersections are either empty or contractible for any $k\geq 1$.
Since $\sr M(Q)$ is a Banach manifold, its homeomorphism class is completely
determined by the incidence data of this cover, which is equivalent either to
that of the good cover of $\Ss^{n-1}$ by the hemispheres
\begin{equation}\label{E:half-spaces} 
	U_{\pm k}=\set{(x_1,\dots,x_n)\in
	\Ss^{n-1}}{\pm x_k>0}	\qquad (k=1,\dots,n) \end{equation} 
or else to the cover of $\Ss^{n-1}\ssm \se{(0,\dots,0,-1)}$ obtained by omitting
$U_{-n}$.

More precisely, let $\tau$ be a \tdef{top} sign string for $\sr M(Q)$, i.e., one
having maximum length $\abs{\tau}$ among those strings $\sig$ such that $\sr
M(Q)$ contains critical curves of type $\sig$. The fact that $\sr U_{\tau}\neq
\emptyset$ will immediately imply that $\sr U_{\sig}\neq \emptyset$ whenever
$\abs{\sig}<\abs{\tau}$. The  integer $n$ appearing in \eqref{E:half-spaces}
equals $\abs{\tau}$, and the combinatorial equivalence between the cover of $\sr
M(Q)$ and that in \eqref{E:half-spaces} is given by
\begin{equation}\label{E:combinatorics} \sr U_c\dar U_{+1},\ \ \sr U_d\dar
	U_{-1}\ \ \text{and}\ \ \sr U_{\sig}\dar U_{\sig(1)\abs{\sig}}\quad 
	(\text{for }\sr
	U_{\sig}\neq \emptyset).  \end{equation} Thus, $\sr M(Q)$ is contractible if
$\sr U_{-\tau}=\emptyset$, and it has the homotopy type of $\Ss^{n-1}$ if $\sr
U_{-\tau}\neq \emptyset$. Note that $n=\abs{\tau}\geq 2$ by the definition of
sign string. If $\sr M(Q)$ does not admit a top sign string (or, equivalently,
if it contains no critical curves at all), then it has the homotopy type of a
point or of $\Ss^0$ according as $\sr U_c$ is empty or not; this situation was
already considered in Theorem 6.1 of \cite{SalZueh1}. 

Briefly stated, denoting by $\sr T$  the subset of $\sr M(Q)$ consisting of all
critical curves: \begin{alignat*}{99}
\sr U_c=\emptyset &\Rar \sr
	M(Q)\home \E; \\ \sr U_c\neq \emptyset \text{ and }\sr T=\emptyset &\Rar \sr
	M(Q)\home \E\times \Ss^0; \\ \sr U_c\neq \emptyset \text{ and }\sr T\neq
	\emptyset &\Rar \sr M(Q)\home \E \times \Ss^{n-1} \text{ or } \E\quad
	(n=\abs{\tau},~\tau \text{ a top sign string}), \end{alignat*} depending on
whether $\sr M(Q)$ contains critical curves of type $-\tau$ or not,
respectively. (It is shown in \cite{SalZueh1} that $\sr U_d$ is never empty, and
that $\sr U_c=\emptyset$ implies $\sr T=\emptyset$.) The determination of
whether $\sr M(Q)$ contains condensed or critical curves of any given type in
terms of $Q$ was already carried out in Propositions 3.17 and  5.3 of
\cite{SalZueh1}, and this is essentially what is depicted in Figure
\ref{F:shadesofgrey}.

Informally, $\ga\colon [0,1]\to \C$ is quasicritical of type $\sig$ if it is
possible to find $\vphi\in \R$ and $t_1<\dots<t_{\abs{\sig}}$ such that the unit
tangent vector $\ta_\ga$ to $\ga$ satisfies $\ta_\ga(t_k)\approx
\sig(k)ie^{i\vphi}$ for each $k=1,\dots,\abs{\sig}$ and
$\gen{\ta_\ga,e^{i\vphi}}>0$ away from these points. In words, $\ga$ is nearly
vertical with respect to the ``axis'' $e^{i\vphi}$ near the points $\ga(t_k)$,
with orientation prescribed by $\sig$, but elsewhere its image is the graph of a
function.  

Unfortunately, the set of all $\vphi\in \R$ with respect to which a curve is
quasicritical of type $\sig$ need not be an interval. Given a continuous family
$K\to \sr U_\sig$, $p\mapsto \ga^p$, this makes it difficult to choose $\vphi^p$
continuously so that each $\ga^p$ is quasicritical with respect to $\vphi^p$. To
circumvent this, we work instead with a certain space $\sr N(Q)\subs \sr
M(Q)\times \R$. The strategy to understand the topology of $\sr N(Q)$ is exactly
as described above: First an open cover $\fr V$ of $\sr N(Q)$ by subsets $\sr
V_c,\,\sr V_d$ and $\sr V_\sig$ is defined, where roughly $\sr V_c$ and $\sr
V_d$ are products of $\sr U_c$ and $\sr U_d$ with $\R$, and for each sign string
$\sig$, $\sr V_\sig$ consists of pairs $(\ga,\vphi)$ such that $\ga$ is
quasicritical of type $\sig$ with respect to $\vphi$. It is then proved that
these sets form a good cover of $\sr N(Q)$, whose combinatorics is determined by
\eqref{E:combinatorics} when $\sr U$ is replaced by $\sr V$. Finally, it is
established that the restriction to $\sr N(Q)$ of the natural projection $\sr M(Q)\times
\R\to \sr M(Q)$ is a homotopy equivalence.

\subsection*{Outline of the sections} Given a sign string $\sig_2$ and a
substring $\sig_1$ of $\sig_2$, there are in general many ways to embed
$\sig_1$ into $\sig_2$. For instance, if $\sig_1=\ty{-+}$ and
$\sig_2=\ty{-+-+}$, then there are three substrings of $\sig_2$ isomorphic to
$\sig_1$, namely, those determined by the pairs of coordinates $(1,2)$, $(1,4)$
and $(3,4)$.  In \S\ref{S:decomposition} we consider certain subspaces of $\R^n$ determined by
inequalities involving a set of strings $\sig_1,\dots,\sig_m$, each a substring
of the next, which encode the purely combinatorial difficulties that arise in
the study of the topology of $\sr V_{\sig_1}\cap \dots \cap \sr V_{\sig_m}$. The
main result of the section states that the former subspaces are in fact all
weakly contractible. In the case of two strings, we construct homeomorphisms
from the resulting spaces onto Euclidean spaces, and for larger sets of strings
we use induction and certain collapsing maps which are quasifibrations.

One of the tools in the proof that the cover $\fr V$ of $\sr N(Q)$ is good is a
procedure for ``stretching'' curves, illustrated in Figures
\ref{F:curvestretching} and \ref{F:surgery},  which generalizes the grafting
construction of \cite{SalZueh1}. This procedure is explained in
\S\ref{S:stretching}, along with some of its properties that are needed later.

The formal definitions of quasicritical curve, the space $\sr N(Q)$ and its
cover $\fr V$ are contained in \S 3. Most of the results in this section concern
basic properties of quasicritical curves, and how to continuously choose
``stretchable'' subarcs for a given family of such curves so that when the
stretching construction is actually carried out, the resulting homotopy will
preserve important properties of the original family, such as being condensed or
simultaneously quasicritical of several types. It is also shown there that the
projection $\sr N(Q)\to \sr M(Q)$ induces surjections on homotopy groups.

The combinatorics of the cover $\fr V$ of $\sr N(Q)$ is determined in \S4. It is
very easy to see that $\sr V_c\cap \sr V_d=\emptyset$ and $\sr V_\sig\cap \sr
V_{-\sig}=\emptyset$ for any sign string $\sig$. On the other hand, given sign
strings $\sig_1,\dots,\sig_m$ with $\abs{\sig_1}<\dots<\abs{\sig_m}$, with some
care one can deform a critical curve of type $\sig_m$ to make it simultaneously
quasicritical of type $\sig_j$ for each $j$. Thus an intersection of 
elements of $\fr V$ is empty if and only if it involves some ``opposite'' pair, just
as for the cover in \eqref{E:half-spaces}.

The objective of \S 5 is to prove that $\fr V$ is a \tsl{good} cover. Given a
continuous family $(\ga^p,\vphi^p)\in \sr V_{\sig_1}\cap \dots \cap \sr
V_{\sig_m}$, with $p$ ranging over a compact space, each $\ga^p$ can be
stretched to become nearly critical (as in Figure \ref{F:surgery}), and then
deformed to a concatenation of circles and line segments (as
in Figure \ref{F:pulleys}) which is essentially determined by the slopes of the
segments. The results of \S\ref{S:decomposition} can then be used to conclude
that the resulting family is nullhomotopic. 

The proof that $\sr N(Q)$ and $\sr M(Q)$ are homeomorphic is completed in \S6.
Moreover, when $\sr M(Q)\iso \Ss^{n-1}$, where $n=\abs{\tau}\geq 2$ is as above,
explicit homotopy inverses $f\colon \Ss^{n-1}\to \sr M(Q)$ and $g\colon \sr
M(Q)\to \sr \Ss^{n-1}$ are constructed. Let $\sr C_\tau$ denote the set of all
critical curves of type $\tau$ in $\sr M(Q)$. Intuitively, the map $g$ measures
the failure of curves in $\sr M(Q)$ to belong to $\sr C_\tau$. If $\al$ is a
generator of $H^*(\Ss^{n-1})$, then $g^{\ast}(\al)$ is the ``Poincar\'e dual''
of $\sr C_\tau$, except that the latter is not really a submanifold of $\sr
M(Q)$. The map $f$ represents a generator of $\pi_{n-1}\sr M(Q)$ and admits the following
description: Regard $\Ss^{n-1}$ as a CW complex with two $k$-cells $e^k_{\pm}$
for every $k=0,\dots,n-1$. Then \begin{equation*}
	f\big(e^{n-1}_+\big)\subs \sr U_d,\ \  f\big(e^{n-1}_-\big)\subs \sr U_c,
\end{equation*} and for each $k=0,\dots,n-2$, $f$ maps $e^{k}_{\pm}$ into the
set of critical curves of type $\pm \sig^{n-k}$ in $\sr M(Q)$, where
$\sig^{n-k}$ denotes any of the two sign strings of length $n-k$. The actual
construction of $f$ is a bit different, but more precise; in particular, it
shows that these inclusions can indeed be satisfied.

\subsection*{Related work}
As far as we know, the first person to systematically study planar curves with
constrained curvature was L.\,E.~Dubins. In the much-cited paper \cite{Dubins},
he investigated curves of minimal length in $ \sr C_{-\ka_0}^{+\ka_0}(P,Q)
$,\footnote{Actually, for the purposes of \cite{Dubins} it is more natural to
work with curves whose curvatures are allowed to be discontinuous and to take
values in the \tsl{closed} interval $ [-\ka_0,+\ka_0] $, otherwise the minimal
length may not be attained.} and in \cite{Dubins1} he attempted to determine the
connected components of this space, obtaining some partial results and
formulating several conjectures. Much later, in \cite{SalZueh1}, the components of $ \sr
C_{\ka_1}^{\ka_2}(P,Q) $ were characterized, and most of his conjectures were
proved. 

Of course, the definition of $ \sr CS_{\ka_1}^{\ka_2}(u,v) $ makes sense for any
Riemannian surface $ S $. One can even consider analogous spaces $\sr
CM_{\ka_1}^{\ka_2}(u,v)$ of curves on a Riemannian manifold $M$ of dimension
$n\geq 2$ by replacing the geodesic curvature of a curve by its $(n-1)$-th
curvature (also called its torsion, cf.~\cite{Kuehnel}, p.~18). The important
special case where $\ka_1=-\infty$ and $\ka_2=+\infty$ (that is, where the
curves are regular but no conditions are imposed on their torsion) was a
precursor to the Hirsch-Smale theory of immersions. Smale showed in
\cite{Smale}, Theorem C, that for any $ u\in UTM $, 
$ \sr CM_{-\infty}^{+\infty}(u,u) $ is weakly
homotopy equivalent to the loop space $ \Om UTM $. In one direction, the
homotopy equivalence comes simply from lifting a regular curve on $ M $ to $ UTM
$. The special case where $M=\R^2$ yields the classical Whitney-Graustein
theorem (\cite{WhiGra}, Theorem 1).

Later work on the subject was mostly concerned with characterizing the connected
components of spaces of closed \tit{nondegenerate} curves, i.e., those having
nonvanishing curvature (torsion). In the present notation, these correspond to
$\sr CM_{-\infty}^{0} (u,u)\du \sr CM_{0}^{+\infty}(u,u)$. Papers treating this
problem for the simplest manifolds, such as $\R^n$, $\Ss^n$ and $\RP^n$, include
\cite{Anisov}, \cite{Feldman}, \cite{Feldman1}, \cite{KheSha}, \cite{KheSha1},
\cite{Little}, \cite{Little1}, \cite{MosSad}, \cite{SalSha}, \cite{ShaSha} and
\cite{ShaTop}. In \cite{SalZueh} the connected components of $\sr
C(\Ss^2)_{\ka_1}^{\ka_2}(u,u)$ are characterized for all $\ka_1<\ka_2$, and in
\cite{Saldanha3} the homotopy type of spaces of (not necessarily closed)
nondegenerate curves on $\Ss^2$ is computed. 
The similar problem for nondegenerate curves on $\Ss^n$ for $n > 2$
appears to be harder and is addressed in
\cite{Alves-Saldanha}, \cite{Goulart-Saldanha1},
\cite{Goulart-Saldanha-cw} and \cite{Goulart-Saldanha0}.
A complete answer is obtained in \cite{Alves-Goulart-Saldanha}
for the case $n = 3$.

\smallskip

The present paper relies strongly on \cite{SalZueh1}. Some familiarity with the
contents of sections 0,\,1,\,3,\,4 and 5 therein will make a few of the proofs
here more easily understood.


\section{On certain subspaces of Euclidean space determined by sign
strings}\label{S:decomposition} \subsection*{A cell decomposition of $\R^n$}
Throughout the article, the set $\se{1,\dots,n}$ will be denoted by $[n]$. 
Let $2\leq n\in \N$, $m\in [n]$ and let $\emptyset\neq J_1,\dots,J_m\subs [n]$
satisfy $[n]=\Du_{j=1}^mJ_j$. Define \begin{equation*}
	W_{J_1,\dots,J_m}=\set{x\in \R^n}{x_k<x_{k'}\text{ if and only if }k\in
	J_j,~k'\in J_{j'} \text{ for some }j<j'\in [m]}.  \end{equation*} It is easy
to check that each \tdef{cell} $W_{J_1,\dots,J_m}$ is an $m$-dimensional convex
cone. Furthermore, $\R^n$ is the disjoint union of all such cells. There is only
one $1$-cell $W_{[n]}$, which consists of the multiples of $(1,1,\dots1)$ in
$\R^n$. At the other end, there are $n!$ cells of dimension $n$, each
$W_{J_1,\dots,J_{n}}$ being determined by the permutation $\pi\in S_{n}$ such
that $\pi(k)$ is the unique element of $J_{k}$. These $n$-cells are open in
$\R^n$, while the $k$-cells for $1<k<n$ are neither open nor closed. See Figure
\ref{F:cells}\?(a) for an illustration of the case $n=3$.

\begin{urem} The $k$-cells in this decomposition are dual to the $(n-k)$-faces
	of the $(n-1)$-dimensional permutohedron. The total number of cells (faces)
	is given by the $n$-th ordered Bell number.  \end{urem}

\begin{figure}[ht] \begin{center} \includegraphics[scale=.20]{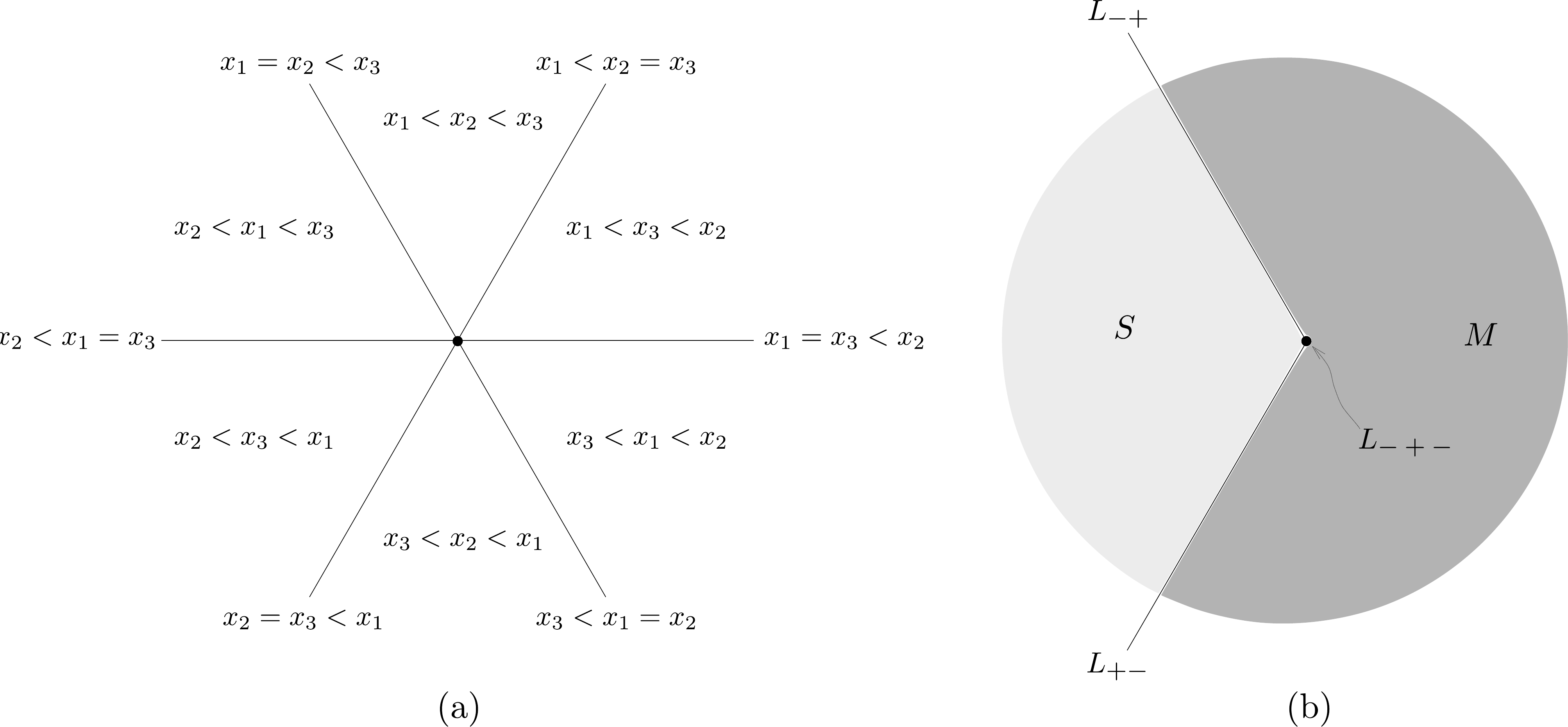}
		\caption{The decomposition of $\R^3$ into the 13 cells $W_{\ast}$ and
		into the sets $M$, $S$ and $L_\sig$, for $\abs{\sig}\geq 2$. More
		precisely, what is depicted here is the orthogonal projection of these
		sets onto the plane $\set{(x_1,x_2,x_3)\in \R^3}{x_1+x_2+x_3=0}$.}
		\label{F:cells} \end{center} \end{figure}

We are actually more interested in another decomposition of $\R^n$, obtained by
comparing even and odd cordinates. 

\begin{defn}[mixed, level, split]\label{D:mixed} Given  $x=(x_1,\dots,x_{n})\in \R^n$, let
	\begin{equation}\label{E:t(x)} t(x)=\min\set{x_k-x_{k'}}{\text{$k$ is odd
			and ${k'}$ is even, $k,k'\in [n]$}}.  \end{equation} We call $x$
		\tdef{mixed}, \tdef{level} or \tdef{split} according as $t(x)<0$,
		$t(x)=0$ or $t(x)>0$, respectively. Define \begin{equation*}
			M=\set{x\in \R^n}{x\text{ is mixed}},\ \ S=\set{x\in \R^n}{x\text{
			is split}},\ \ L=\set{x\in \R^n}{x\text{ is level}}.
		\end{equation*} \end{defn}

It is convenient to represent a point $x=(x_1,\dots,x_{n})\in \R^n$ as an
ordered set of $n$ beads, each of which is allowed to slide along a vertical
line. The height of the $k$-th bead (above a certain fixed ground height) gives
the value of $x_k$; see Figure \ref{F:splitlevelmixed}. 

\begin{figure}[ht] \begin{center} \includegraphics[scale=.18]{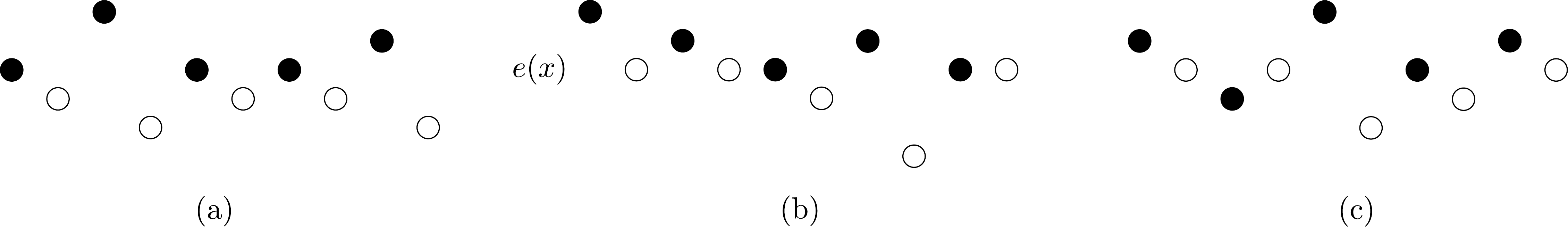}
		\caption{Split, level and mixed points in $\R^{10}$, respectively,
	represented by beads (black for odd-indexed coordinates and white for
even-indexed  coordinates).} \label{F:splitlevelmixed} \end{center} \end{figure}

An \tdef{interval} $J\subs [n]$ is a set of the form $(a,b)\cap [n]$ for some
$a<b\in \R$. Given two intervals $J_1,J_2$, we write $J_1<J_2$ if  $k_1<k_2$
whenever $k_1\in J_1,~k_2\in J_2$.

\begin{defn}[sign string, level type]\label{D:level} When $x\in \R^n$ is level, there exists a unique
	$e(x)\in \R$ satisfying $x_k=e(x)=x_{k'}$ for some odd $k$ and even ${k'}$
	(see Fig.~\ref{F:splitlevelmixed}\?(b); `e' stands for ``elevation''). For
	each integer $m\geq 2$, define \begin{equation}\label{E:standard}
		\sig^m, -\sig^m\colon [m]\to \se{\pm 1}\equiv \se{\pm},~\text{\quad
		by\quad}\sig^m(j)=(-1)^j\text{ and }-\sig^m(j)=(-1)^{j+1}.
	\end{equation} For example, $\sig^3$ is represented by $\ty{-+-}$, and
	$-\sig^4$ by $+-+-$. By definition, a sign string $\sig$ is of the form
	$\pm \sig^m$ for some $m\geq 2$, and $\abs{\sig}$ denotes its length $m$. A
	level point $x=(x_1,\dots,x_{n})\in \R^n$ is \tdef{of type} $\sig$ if we can
	find nonempty intervals $J_1,\dots,J_{\abs{\sig}}$, such that:
\begin{enumerate} \item [(i)] $J_1<J_2<\dots<J_{\abs{\sig}}$ and
		$[n]=\bcup_{j=1}^{\abs{\sig}} J_j$.  \item [(ii)]  For each $j$,
			there exists at least one $k\in J_j$ with 
			$x_k=e(x)$, and $(-1)^k=\sig(j)$ for all such $k$.  \end{enumerate} The set of all
	level points of type $\sig$ in $\R^n$ will be denoted by $L^n_\sig$ or
	simply $L_\sig$.  \end{defn} 

In other words, to determine the type of a level point $x\in \R^n$, we assign a
tag $\ty{-}$ (resp.~$\ty{+}$) to each odd (resp.~even) bead lying at height
$e(x)$, and read off the corresponding signs, omitting any repetitions;
see Figure \ref{F:level}. 

\begin{figure}[ht] \begin{center} \includegraphics[scale=.18]{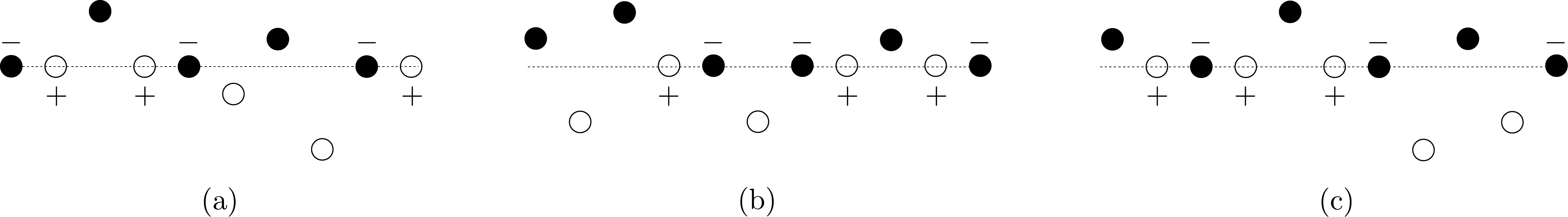} \caption{An
			element of $L^{10}_{\sig^4}$  and two elements of
			$L^{11}_{-\sig^4}$, respectively. According to
			\tref{P:decomposition}, the latter space is homeomorphic to $\R^8$.
			In particular, the two points represented in (b) and (c) can be
			joined without leaving $L^{11}_{-\sig^4}$.} \label{F:level}
		\end{center} \end{figure}

Observe that the sets $M$, $S$ and $L_\sig$ are pairwise disjoint cones.
Moreover, $L_{\sig^{n}}=W_{[n]}$ and $L_{\tau}=\emptyset$ if $\tau=-\sig^{n}$ or
$\abs{\tau}>n$. The sets $M$, $S$ are open in $\R^n$, while the $L_\sig$ are
neither open nor closed for $\abs{\sig}<n$. Each of the sets $M$, $S$ and 
$L_\sig$, for any sign string $\sig$, is a union of cells $W_{\ast}$ of $\R^n$. 
Equivalently, each cell of $\R^n$ is contained in one of these sets. The proofs
of these assertions are all straightforward. See Figure \ref{F:cells}\?(b) for the case 
$n=3$. 

\begin{defn}\label{D:boundary} For an integer $m\geq 1$, let
	\begin{equation*}
	\Hh^m=\set{(x_1,\dots,x_m)\in \R^m}{x_m\geq
		0}\text{\ \ and\ \ }-\Hh^m=\set{(x_1,\dots,x_m)\in \R^m}{x_m\leq 0}.
	\end{equation*} For a space $Y\home \Hh^m$, define $\bd Y$ to consist of all
	$y\in Y$ such that the local homology  $H_*(Y,Y\ssm \se{y})$ at $y$ is
	trivial. Note that $\bd Y$ is exactly the image of $\R^{m-1}\times \se{0}$
	under any homeomorphism $\Hh^m\to Y$.  \end{defn}

Our first goal is to prove the following result. 

\begin{prop}\label{P:decomposition} For $\sig$ a sign string with $2\leq
	\abs{\sig}\leq n-1$, let \begin{alignat*}{9}
		Y_{\pm}=&\set{(y_1,\dots,y_{n})\in \R^{n}}{\pm y_{1}>0}, \\
		Y_\sig=&\set{(y_1,\dots,y_{n})\in \R^{n}}{y_k=0\text{ for all
			$k<\abs{\sig}$ and $\sig(1)y_{\abs{\sig}}>0$}}, \\
			Y_{\sig^n}=&\set{(y_1,\dots,y_{n})\in \R^{n}}{y_k=0\text{ for all
			$k<n$}}.  \end{alignat*} Then there exists a homeomorphism $f\colon
		\R^n\to \R^n$ such that $f(M)=Y_{-}$, $f(S)=Y_{+}$ and
		$f(L_\sig)=Y_{\sig}$ for all sign strings $\sig$ with $\abs{\sig}\leq
		n$, $\sig\neq -\sig^{n}$.  \end{prop} \begin{cor}\label{C:decomposition}
		Let $M,\,S,\,L^n_\sig\subs \R^n$. Then $M\home S\home \R^n$,
		$\ol{M}\home \ol{S}\home \Hh^n$, $L^n_\sig\home \R^{n+1-\abs{\sig}}$ and
		$\ol{L}^n_\sig\home \Hh^{n+1-\abs{\sig}}$ \tup($\abs{\sig}<n$\tup).
		Also, $L^n_{\sig^{n}}=\ol{L}^n_{\sig^{n}}\home \R$ and
		$L^n_{-\sig^{n}}=\emptyset$.\qed \end{cor} In particular, each of the
	sets $L_\sig$ and $\ol{L}_\sig$ is contractible. It is a good exercise to try to
	visualize a contraction using the representation by beads, as in Figure
	\ref{F:level}. 

\begin{rem}\label{R:extension} Any homeomorphism $\bd \Hh^{k}\to \bd \Hh^{k}$
	may be extended to a homeomorphism of $\Hh^{k}$ onto itself.  \end{rem}

\begin{lem}\label{L:glueing} Let $H_1\cup H_2$ be a topological space with 
	$H_1\home H_2\home \Hh^k$ and $\bd H_1=\bd H_2=H_1\cap H_2$.
	Then there exists a homeomorphism $f\colon H_1\cup H_2\to \R^k$ such that
	$f(H_1)=\Hh^k$ and $f(H_2)=-\Hh^k$.  \end{lem}

\begin{proof} Let $g_1\colon H_1\to \Hh^k$ and $g_2\colon H_2\to -\Hh^k$ be
	homeomorphisms. Then the restriction of $g_2\circ (g_1)^{-1}$ to $\bd \Hh^k$
	is a homeomorphism $\bd \Hh^k\to \bd \Hh^k$. Using \rref{R:extension},
	extend this to a homeomorphism $g\colon \Hh^k\to \Hh^k$. Now glue $g\circ
	g_1$ and $g_2$ along $\bd H_1=\bd H_2$.
\end{proof}

\begin{lem}\label{L:glueing1} Let $H_1\cup H_2$ be a topological space with 
	$H_1\home H_2\home \Hh^k$ and $\bd H_i=C\cup D_i$, where $C\home
	D_i\home \Hh^{k-1}$, $C\cap D_i=\bd C=\bd D_i$ and $H_1\cap H_2=C$
	$(i=1,2)$. Then $H_1\cup H_2\home \Hh^k$.  \end{lem}

\begin{proof} Let $f_0\colon C\to \Hh^{k-1}$ be a homeomorphism. Using
	\lref{L:glueing}, $f_0$ may be extended to a homeomorphism $f_1\colon C\cup
	D_1\to \R^{k-1}$, and then since $\bd H_1=C\cup D_1$, $f_1$ has an extension
	to a homeomorphism $g_1\colon H_1\to \Hh^k$, by \rref{R:extension}. Finally,
	we  compose $g_1$ with the homeomorphism \begin{equation*}
		\Hh^k\to Q_1=\set{(x_1,\dots,x_k)\in \R^k}{x_{k-1}\geq 0\text{ and
		}x_{k}\geq 0}, \end{equation*} obtained by taking the square root (in
	$\C$) of the last two coordinates $(x_{k-1},x_{k})$ of points $x\in \Hh^k$.
	The result is a homeomorphism $h_1\colon H_1\to Q_1$ such that
	$h_1|_{C}=f_0$. 
	
	Repeating the argument for $H_2$, starting from $f_0$ again, we obtain a
	homeomorphism \begin{equation*}
	h_2\colon H_2\to
		Q_2=\set{(x_1,\dots,x_k)\in \R^k}{x_{k-1}\geq 0\text{ and }x_{k}\leq 0},
	\end{equation*} with $h_2|_{C}=f_0$. Glueing $h_1$ and $h_2$ along $C$, we
	finally obtain the desired homeomorphism \begin{equation*}
		h\colon H_1\cup H_2\to Q_1\cup Q_2=\set{(x_1,\dots,x_k)\in
		\R^k}{x_{k-1}\geq 0}\home \Hh^k.\qedhere \end{equation*} \end{proof}

\begin{lem}\label{L:MandS} Let $M,\,S,\,L\subs \R^n$ be as in \dref{D:mixed}.
	Then there exists a homeomorphism $g\colon \R^n\to  L\times \R$ with
	$g(M)=L\times (-\infty,0)$ and $g(S)=L\times (0,+\infty)$. In particular,
	$\ol{M}\cap \ol{S}=L$.  \end{lem} \begin{proof} Define a map $h\colon
	L\times \R\to \R^n$ by \begin{equation*}
		h(x,t)=(x_1+t,x_2-t,\dots,x_{n}+(-1)^{n-1}t)\qquad
		(x=(x_1,\dots,x_{n})\in L,~t\in \R).  \end{equation*} Given $x\in \R^n$,
	let $t(x)$ be as in eq.~\eqref{E:t(x)} and $\bar t(x)=\frac{1}{2} t(x)$. Let
	\begin{equation*}
	g\colon \R^n\to L\times \R,\quad
		g(x)=\big((x_1-\bar t(x),x_2+\bar t(x),\dots,x_n+(-1)^n\bar
		t(x))\,,\,\bar t(x)\big).  \end{equation*} Then $g$ and $h$ are inverse
	maps. Moreover, it is an immediate consequence of \dref{D:mixed} that
	$g(M)=L\times (-\infty,0)$ and $g(S)=L\times (0,+\infty)$, as claimed.
\end{proof}

\begin{lem}\label{L:boundaries} For any sign string $\sig$, the closure
	$\ol{L}_\sig$ of $L_\sig$ in $\R^n$ satisfies $\ol{L}_\sig=L_\sig\cup
	\bcup_{\abs{\tau}>\abs{\sig}}L_{\tau}$. In particular, $\ol{L}_{\sig^m}\cap
	\ol{L}_{-\sig^m}=\bcup_{\abs{\tau}>m}L_\tau$.  \end{lem} \begin{proof}
	Let $\tau$ be a sign string and suppose that $ x\in L_{\tau} $. Define
	\begin{equation*}
		\mu=\frac{1}{2}\min\set{\abs{x_k-e(x)}}{x_k\neq e(x),~k\in [n]}.
	\end{equation*} Then the set \begin{equation*}
		U=\set{(y_1,\dots,y_{n})\in \R^n}{\abs{y_k-x_k}<\mu\text{ for each }k\in
	[n]} \end{equation*} is a neighborhood of $x$ with the property that
$U\cap L_{\tau'}=\emptyset$ if $\tau'$ is not a substring of $\tau$ (see
\dref{D:extended} for the formal definition of ``substring''). It follows that,
$\ol{L}_\sig\subs L_\sig\cup\bcup_{\abs{\tau}>\abs{\sig}}L_{\tau}$. 
	
	Conversely, if $\abs{\tau}>\abs{\sig}$ and $x\in L_\tau$, choose indices
$k_1<\dots<k_l$ such that: \begin{enumerate} \item [(i)] $x_{k_i}=e(x)$ for each
		$i\in [l]$; \item [(ii)] If  $k'_1<\dots<k'_r$ are all the remaining
			indices such that $x_{k'}=e(x)$, then $r=\abs{\sig}$ and
			$(-1)^{k'_j}=\sig(j)$ for each $j\in [r]$.  \end{enumerate} This is
	possible since $\sig$ is a substring of $\tau$. Points in $L_\sig$
	arbitrarily close to $x$ can be obtained by moving the coordinates $x_{k_i}$
	away from $e(x)$. More precisely, for $s\in [0,1]$, let
	$x(s)=(x_1(s),\dots,x_{n}(s))\in \R^n$ be defined by:
	\begin{equation*}
		x_k(s)=
		\begin{cases} 
			x_k+(-1)^{k-1}s & \text{ if $k\in \{k_1,\dots,k_l\}$;} \\
			x_k & \text{ otherwise}  
		\end{cases}\qquad (k\in [n]).  
	\end{equation*} Then $x(0)=x$
	and $x(s)\in L_\sig$ for all $s>0$ by construction. Hence $x\in
	\ol{L}_\sig$.  \end{proof}

\begin{proof}[Proof of \tref{P:decomposition}] By induction on $n$. If $n=2$,
	then $L=L_{-+}=\set{(x_1,x_2)\in \R^2}{x_1=x_2}$, while $M$ (resp.~$S$)
	consists of those points above (resp.~below) this line. Thus, rotation by
	$\frac{\pi}{4}$ about the origin is the desired homeomorphism. 	

	Let $n\geq 3$ and assume that the assertion has been proved for all
	dimensions smaller than $n$. (The case
	$n=3$ also follows from Figure \ref{F:cells}\?(b).)
	The homeomorphism $\R^n\to \R^n$ will be
	constructed stepwise. We start with a homeomorphism $f\colon
	L^n_{\sig^{n}}\to Y_{\sig^{n}}$, which exists since both of these sets are
	lines in $\R^n$. Suppose that $f$ has already been extended to a
	homeomorphism $f\colon \bcup_{\abs{\sig}\geq m+1}L^n_\sig\to
	\bcup_{\abs{\sig}\geq m+1}Y_\sig$ for some $m$ satisfying $2\leq m\leq n-1$. 
	
	Let $\phi\colon L\to \R^{n-1}$ and $\la\colon L\to [0,+\infty)$ be the
		maps which forget and recover the last coordinate:
		\begin{equation*}
		\phi(x)=( x_1,\dots,x_{n-1}),\quad
			\la(x)=\abs{x_{n}-e(x)}\qquad (x=(x_1,\dots,x_{n})\in L),
		\end{equation*} where $e(x)$ is as in \dref{D:level}. Let us suppose for
		concreteness that $m\equiv n\pmod 2$; the only difference in the other
		case is that the roles of $L^n_{\sig^m}$ and $L^n_{-\sig^m}$ are
		switched. A straightforward verification shows that
		\begin{equation*}
		\phi\times \la\colon \ol{L}^n_{-\sig^m}\to
			\ol{L}^{n-1}_{-\sig^m}\times [0,+\infty) \end{equation*} is a
			homeomorphism, hence $\ol{L}_{-\sig^m}^n\home \Hh^{n-m}\times
			[0,+\infty) \home \Hh^{n+1-m}$ by the induction hypothesis on $n$. 
	
	To understand $\ol{L}^n_{\sig^m}$, we consider its decomposition into
	$H_1\cup H_2$, where 
	\begin{alignat*}{10}
		H_1:&=\set{x\in \ol{L}^n_{\sig^m}}{\la(x)=0,~\phi(x)\in
			\ol{L}_{\sig^{m-1}}^{n-1}}\home \Hh^{n-m+1}\text{ via }\phi,\\
			H_2:&=\set{x\in \ol{L}^n_{\sig^m}}{\la(x)\geq 0,~\phi(x)\in
				\ol{L}_{\sig^m}^{n-1}}\home \Hh^{n-m}\times [0,+\infty)\home
					\Hh^{n-m+1}\text{ via }\phi\times \la,\\ C:=H_1\cap
					H_2&=\set{x\in \ol{L}^n_{\sig^m}}{\la(x)=0,~\phi(x)\in
						\ol{L}^{n-1}_{\sig^{m}}}\home \Hh^{n-m}\text{ via }\phi, 
	\end{alignat*} by the induction hypothesis on
						$n$. Moreover, $\bd H_1=C\cup D_1$ and $\bd H_2=C\cup
						D_2$, where \begin{alignat*}{10}
							D_1&=\set{x\in
								\ol{L}^n_{\sig^m}}{\la(x)=0,~\phi(x)\in
									\ol{L}^{n-1}_{-\sig^m}}\home \Hh^{n-m}\text{
									via }\phi,\\ D_2&=\set{x\in
										\ol{L}^n_{\sig^m}}{\la(x)\geq
											0,~\phi(x)\in
											\ol{L}^{n-1}_{-\sig^{m+1}}\cup
											\ol{L}^{n-1}_{\sig^{m+1}}} \home
											\R^{n-m-1}\times [0,+\infty)\home
												\Hh^{n-m}\text{ via }\phi\times
												\la, \end{alignat*} again by the
											induction hypothesis on $n$. Thus we
											are in the setting of
											\lref{L:glueing1}, and the
											conclusion is that
											$\ol{L}_{\sig^m}^n=H_1\cup H_2\home
											\Hh^{n-m+1}$. 
	
	Now by \lref{L:boundaries}, \begin{equation*}
		\ol{L}^n_{\sig^m}\cap \ol{L}^n_{-\sig^m}=\bcup_{\abs{\tau}\geq
	m+1}L_{\tau}^{n}.  \end{equation*} Since by assumption we already have
a homeomorphism from the latter set to $\bcup_{\abs{\tau}\geq
m+1}Y_{\tau}\home \R^{n-m}$, \lref{L:glueing} guarantees the existence of a
homeomorphism \begin{equation*}
f\colon \bcup_{\abs{\tau}\geq
m}L_{\tau}^{n}\to \bcup_{\abs{\tau}\geq m}Y_{\tau}\home \R^{n+1-m}.
\end{equation*} Continuing this down to $m=2$, a homeomorphism $f\colon L\to
\bcup_{\abs{\tau}\geq 2}Y_{\tau}\home \R^{n-1}$ taking each $L_{\sig}$ onto
$Y_{\sig}$ is obtained. Finally, an application of \lref{L:glueing} using
\lref{L:MandS} shows that this can be extended to a homeomorphism $f\colon
\R^n\to \R^n$ having the required properties.  \end{proof}

\subsection*{Subspaces determined by nested strings}

	Let $E,\,Y$ be topological spaces, $q\colon E\to Y$ be a (continuous)
	surjective map and for each $y\in Y$, let $F_y=q^{-1}(y)$ denote the fiber
	of $y$. Then $q$ is a \tdef{quasifibration} if for any $k\geq 0$, $y\in Y$
	and $e\in F_y$, the induced map $q_*\colon \pi_k(E,F_y,e)\to \pi_k(Y,y)$ on
	homotopy groups is an isomorphism.\footnote{See \cite{DolTho}, Bemerkung 1.2
for the definition of $\pi_0(E,F_y,e)$. For $k=0,\,1$, when the set on the left
side has no natural group structure, it should be understood that $q_*\colon
\pi_k(E,F_y,e)\to \pi_k(Y,y)$ is a bijection.}

Thus, if $q\colon E\to Y$ is a quasifibration, then for any $y\in Y$ and $e\in
F_y$, there is a long exact sequence \begin{equation}\label{E:longexact} \dots
	\to{} \pi_k(F_y,e) \lto{j_*}  \pi_k(E,e) \lto{q_*} \pi_k(Y,y) \lto{\bd}
	\pi_{k-1}(F_y,e)\to{}\dots   	\to{}\pi_0(E,e)\to{} 0 \end{equation} which
is obtained from the long exact sequence of the triple $(E,F_y,e)$ by
identifying $\pi_k(E,F_y,e)$ with $\pi_k(Y,y)$; here $j$ is the inclusion
$F_y\inc E$. Just as for a Serre fibration, it can be shown that if $Y$ is
path-connected, then all fibers $F_y$ have the same weak homotopy type.	

\begin{prop}[\cite{DolTho}, Satz 2.2]\label{P:DolTho} Let $q\colon E\to Y$ be a
surjective map and suppose that $\fr U=(U_\nu)_{\nu\in I}$ is an open cover of
$Y$ satisfying: \begin{enumerate} \item [(i)] For each $\nu\in I$,
			$q|_{q^{-1}(U_\nu)}\colon q^{-1}(U_\nu)\to U_\nu$ is a
			quasifibration; \item [(ii)]  If $y\in U_{\nu_1}\cap U_{\nu_2}$,
				then there exists $\nu$ such that $y\in U_\nu\subs U_{\nu_1}\cap
				U_{\nu_2}$ (for $\nu_1,\,\nu_2,\,\nu\in I$).  \end{enumerate}
		Then $q$ is a quasifibration.\qed \end{prop}

\begin{defn}\label{D:extended} An \tdef{extended string} $\tau$
	is a function $\tau\colon [l]\to \se{\pm}$ ($l\geq 2$). Thus, in contrast to sign
	strings, in an extended string some signs may be repeated. Given 
	extended strings $\tau_i\colon [l_i]\to \se{\pm}$ $(i=1,2)$,
	$\tau_1$ is a \tdef{substring} of $\tau_2$, denoted $\tau_1\po
	\tau_2$ (or $\tau_1\prec \tau_2$ if in addition $\tau_1\neq \tau_2$), if
	there is a strictly increasing $f\colon [l_1]\to [l_2]$ such that
	$\tau_1=\tau_2\circ f$. For $\tau$ an extended string, its \tdef{reduced
	string} is the unique sign string $\vrho$ of maximal length such that
	$\vrho\po \tau$. It is obtained by omitting all repetitions in $\tau$; e.g.,  the
	reduced sign string of $\ty{+--++}$ is  $\ty{+-+}$.  \end{defn}

\begin{defn}\label{D:nested} Let $\sig_1\po \dots \po
	\sig_{m}$ $(m\geq 1)$, where $\sig_m$ is an extended string and the remaining $\sig_j$
	are sign strings. Let $n=\abs{\sig_m}$. Define
	$X_{(\sig_1,\dots,\sig_m)}\subs \R^{n}$ to be the subspace consisting of all
	$x=(x_1,\dots,x_n)$ satisfying the following conditions: \begin{enumerate}
		\item [(i)] $\sig_m(k)x_k\leq 0$ for all $k\in [n]$; \item [(ii)]
			$\abs{x_k}\leq m$ for all $k\in [n]$ and for each $j\in [m-1]$, if
			$k_1<\dots<k_l$ are all the indices in $[n]$ such that
			$\abs{x_k}\leq j$, then $\sig_j$ is the reduced string of
			$\tau\colon [l]\to \se{\pm}$, $\tau(i)=\sig_m(k_i)$.
	\end{enumerate} \end{defn} Representing a point in $\R^n$ by beads, to
	determine whether it satisfies (ii) we assign a tag $\sig_m(k)$ to its
	$k$-th bead for each $k\in [n]$ and read off the tags of those coordinates
	that lie at or below height $j$ and at or above height $-j$; the
	corresponding reduced string should coincide with $\sig_j$ for each $j\in
	[m-1]$, and $\abs{x_k}\leq m$ should hold for all $k\in [n]$. See Figure
	\ref{F:nested}. 
	\begin{rem}\label{R:cube1}
		Note that if $m=1$, then the resulting space is just an
		$n$-dimensional cube.
	\end{rem}
\begin{figure}[ht] \begin{center} \includegraphics[scale=.24]{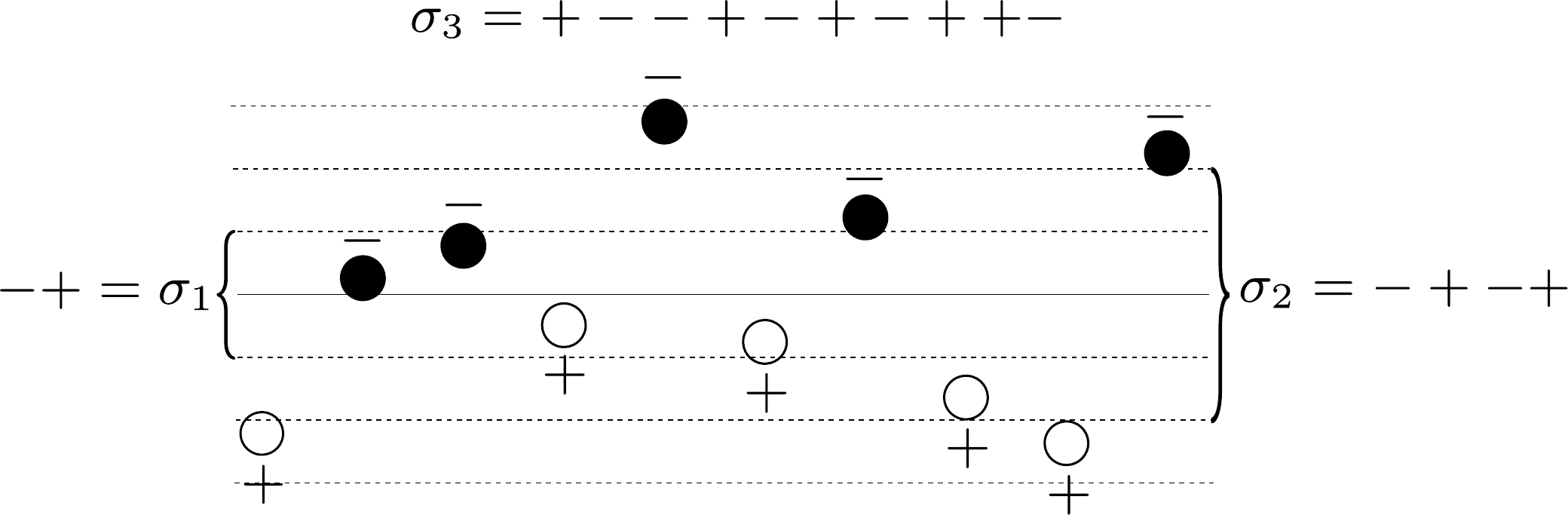}
		\caption{An element of $X_{(\sig_1,\sig_2,\sig_3)}$ for $\sig_j$ as
	indicated in the figure.}
\label{F:nested} \end{center}
\end{figure}

\begin{prop}\label{P:nested} Let $\sig_1\po \dots \po
	\sig_{m}$ \tup($m\geq 1$\tup), where $\sig_m$ is an extended string and the remaining $\sig_j$
	are sign strings. Then $X_{(\sig_1,\dots,\sig_{m})}$ is weakly contractible.
\end{prop}

We are only interested in the case where $\sig_m$ is a sign string, but for
the proof given below to work, this more general version is needed, as well as
another definition: Let $\sig_2$ be an extended string and $\sig_1\po \sig_2$
be a sign string, $\abs{\sig_2}=n$. Define $L_{\sig_1}^{\sig_2}\subs
\R^n$ by declaring that $x=(x_1,\dots,x_n)\in L_{\sig_1}^{\sig_2}$ if and only
if it satisfies condition (i) above (with $m=2$) together with:
\begin{enumerate} \item [(iii)] $\abs{x_k}\leq 1$ for all $k\in [n]$ and if
			$k_1<\dots<k_l$ are all the indices in $[n]$ such that $x_k=0$, then
			$\sig_1$ is the reduced string of $\tau\colon [l]\to \se{\pm }$,
			$\tau(i)=\sig_2(k_i)$.  \end{enumerate}

\begin{lem}\label{L:basestep} Let $\sig_1\po \sig_2$ be a sign and an extended
	string, respectively. Then $L_{\sig_1}^{\sig_2}$ is contractible.  \end{lem}
\begin{proof} Let $n=\abs{\sig_2}$ and \begin{equation*}
		L_0=\set{(x_1,\dots,x_n)\in L_{\sig_1}^{\sig_2}}{x_k=x_{k+1} \text{ if
		}\sig_2(k)=\sig_2(k+1),\text{ for each }k\in [n-1]}.  \end{equation*}
	Let $\vrho$ be the reduced string of $\sig_2$, $r=\abs{\vrho}$ and
	$J_1<\dots<J_{r}$ be the maximal intervals in $[n]$ such that
	$\sig_2(J_i)=\se{\vrho(i)}$. Define a deformation retraction $f\colon
	[0,1]\times L_{\sig_1}^{\sig_2}\to L_{\sig_1}^{\sig_2}$ onto $L_0$ by:
	\begin{alignat*}{9}
	f_k(s,x)&=(1-s)x_k+s\mu_i(x)\text{ if $k\in
		J_i$\ $(k\in [n])$,\ where}\\ \mu_i(x)&=\begin{cases} \min\set{x_j}{j\in
			J_i} & \text{\ \ if $\vrho(i)=\ty{-}$}; \\ \max\set{x_j}{j\in J_i} &
			\text{\ \ if $\vrho(i)=\ty{+}$}. \\ \end{cases} \end{alignat*} No
	generality is lost in assuming that $\vrho=\sig^{r}$ instead of
	$-\sig^{r}$ (as defined in \eqref{E:standard}). Then $L_0$ is homeomorphic
	to the subspace of $L_{\sig_1}^{r}$ consisting of those $y$ for which
	$\abs{y_i}\leq 1$ for all $i\in [r]$. But clearly, this subspace is a
	deformation retract of $L_{\sig_1}^{r}$, hence $L_{\sig_1}^{\sig_2}\iso
	L_{\sig_1}^{r}$ is contractible by \pref{C:decomposition}.  \end{proof}

\begin{proof}[Proof of \pref{P:nested}] By induction on $m$. The case $m=1$
	follows from \rref{R:cube1}. 
	Suppose that $m\geq 2$ and that the assertion has been established for
	$m-1$.  Set \begin{equation*}
		E=X_{(\sig_1,\dots,\sig_{m})},\ \  Y=L_{\sig_{m-1}}^{\sig_m}\text{\ \
		and \ \ }n=\abs{\sig_m}.  \end{equation*} Let $q\colon E\to Y$ be the
	map which collapses everything at height between $-(m-1)$ and $(m-1)$. To be
	precise, if $x=(x_1,\dots,x_n)\in E$, then its image $y=q(x)$ has
	coordinates \begin{equation*}
		y_k=-\sig_m(k)\max\{\abs{x_k}-(m-1),0\}\quad (k\in [n]).
	\end{equation*} Although $q$ is generally not a Serre nor a Dold fibration,
	we claim that it is a quasifibration.
	
	Given $y\in Y$, let $F_y=q^{-1}(y)$ and let $\tau$ be the substring of
	$\sig_m$ determined by all indices $k$ such that $y_k=0$. Note that
	$\sig_{m-1}$ is the reduced string of $\tau$. The map $F_y\to
	X_{(\sig_1,\dots,\sig_{m-2},\tau)}$ which sends $x\in F_y$ to the point in
	$\R^{\abs{\tau}}$ obtained by deleting its coordinates $x_k$ such that
	$\abs{x_k}>m-1$ is a homeomorphism. Hence $F_y$ is weakly contractible by
	the induction hypothesis. 
	
	For $y\in Y$, let $\de(y)=\min\set{\abs{y_k}}{y_k\neq 0,~k\in [n]}$. Then
	the sets \begin{equation*}
	U_{y,\de}=\set{(z_1,\dots,z_n)\in
		Y}{\abs{ z_k-y_k}<\de\text{ for each $k\in [n]$}}\qquad \big(y\in
		Y,~0<\de<\de(y)\big) \end{equation*} form an open cover $\fr U$ of $Y$.
	Condition (ii) in \pref{P:DolTho} is obviously satisfied by $\fr U$.
	Moreover, each $U_{y,\de}\in \fr U$ is star-shaped with respect to $y$, 
	hence contractible.
	A deformation retraction 
	\begin{equation*}
		g\colon [0,1]\times q^{-1}(U_{y,\de})\to q^{-1}(U_{y,\de}) 
	\end{equation*}
	onto $F_y$ can be defined through 
	\begin{equation*}
	g_k(s,x)=\begin{cases} x_k & \text{ if\ \
				$\abs{x_k}\leq m-1$} \\ (1-s)x_k+s\big[y_k-\sig_m(k)(m-1)\big] &
				\text{ if\ \  $\abs{x_k}\geq (m-1)$} \end{cases}\quad (k\in
			[n]).  \end{equation*} Therefore, condition (i) in \pref{P:DolTho}
		is trivially satisfied: From the long exact sequence of homotopy groups
		of the pair $\big(q^{-1}(U_{y,\de})\,,\,F_y\big)$, it follows that
		$\pi_i(q^{-1}(U_{y,\de})\?,\?F_y\?,\?e)$ is trivial for all $i\geq 0$
		and $e\in F_y$, and so is $\pi_i(U_{y,\de}\?,\?y)$. Hence $q$ is a
		quasifibration. By \lref{L:basestep}, $Y$ is weakly contractible. Using
		exactness of \eqref{E:longexact} we conclude that $E$ is weakly
		contractible.  \end{proof}

\begin{defn}\label{D:dnested} Define  $X_{(d,\sig_1,\dots,\sig_m)}\subs \R^{n}$
as in \dref{D:nested}, but replacing (i) by: \begin{enumerate} \item [(i$_d$)]
			There exist $k_1,k_2\in [n]$ with $\sig_m(k_2)=-\sig_m(k_1)$ and
			$\sig_m(k_i)x_{k_i}> 0$.  \end{enumerate} \end{defn}

The `d' here refers to the relation of this condition to diffuse curves, as will
become clear later.
\begin{prop}\label{P:dnested} The space $X_{(d,\sig_1,\dots,\sig_{m})}$ is
	weakly contractible.  \end{prop} \begin{proof} Analogous to the proof of
	\pref{P:nested}: Use induction on $m$ and the same collapsing map $q$ as
	before to reduce to the case where $m=1$. Then consider the map
	\begin{equation*}
	p\colon X_{(d,\sig_1)}\to L=\set{x\in
		\R^n}{x\text{ is level}},\ \ p_k(x)=\begin{cases} x_k & \text{ if
			$\sig_m(k)x_k\leq 0$;} \\ 0& \text{ if $\sig_m(k)x_k\geq 0$.}
		\end{cases}\quad (k\in [n]).  \end{equation*} This is a quasifibration
	with convex fibers, and $L\home \R^{n-1}$ by
	\pref{P:decomposition}.  \end{proof}

\begin{rem}\label{R:strict} For the sake of simplicity, in condition (ii) of
	\dref{D:nested} the ``heights'' appearing in the inequalities were
	chosen to be elements of $ [m] $. However, we clearly could have replaced
	$j$ by $ \eps_j $ without affecting the subsequent results, for any
	$0<\eps_1<\dots<\eps_m$. Furthermore, 
	since only \tsl{weak} contractibility is asserted, it follows from this more
	general version of \pref{P:nested} and \pref{P:dnested} that if some of the
	inequalities in (i) and (ii) are replaced by strict inequalities, then the
	resulting space is again weakly contractible.  \end{rem}


\section{Stretching}\label{S:stretching}

\subsection*{Stretching of functions}In this section we shall describe a
procedure for ``stretching'' curves (as illustrated in Figure
\ref{F:curvestretching}), generalizing the 
grafting construction of \cite{SalZueh1}. We rely
heavily on the results of \S3 of \cite{SalZueh1} and retain the notation
introduced there. The procedure is more clearly formulated in terms of real
functions. For a (Lebesgue integrable) function $g\colon J\to \R$, $\int g$
denotes $\int_Jg$. 
\begin{notn}\label{N:stretchable}
	Let $b>0$, $\ka_0\in (0,1)$, $r_0,\,r_b,\,A\in \R$ be fixed but
	otherwise arbitrary and $f\colon [0,b]\to \R$ be an absolutely continuous
	function whose derivative $f'$ lies in $L^2[0,b]$. Assume that:
	\begin{enumerate}
		\item [(i)] $\abs{f'(x)}\leq \ka_0\big[1+f(x)^2\big]^{\frac{3}{2}}$ for almost every
			$x$ in the domain of $f$; 
		\item [(ii)] $f(0)=r_0$, $f(b)=r_b$ and $\int
				f=A$.	
	\end{enumerate} 
\end{notn}
At this point the reader is referred to \dref{D:stretchablecurve} and 
\rref{R:tangent} for the motivation for these conditions and the
following results, which will otherwise be lacking.

Let $g_{\pm}\colon \R\to \bar{\R}$ and $h_{\pm}^b=h_{\pm}\colon \R\to \bar{\R}$
be as in (24) and (25) of \cite{SalZueh1}, where $\bar{\R}=\R\cup
\se{\pm \infty}$. The
functions $g_\pm$ are the solutions of the differential equations $g'=\pm
\ka_0(1+g^2)^{\frac{3}{2}}$ with $g(0)=r_0$. Similarly, $h_{\pm}^b$ are the
solutions of $h'=\mp \ka_0(1+h^2)^{\frac{3}{2}}$
satisfying $h(b)=r_b$. Since $g_+$ is strictly increasing and $h_+^b$ is strictly
decreasing, the graphs of these functions either do not intersect, or do so at a
single point. In the latter case let $\la_+(b)$ denote their common value at
this point, and in the former set $\la_+(b)=+\infty$. Let $\la_-(b)$ be defined
analogously.\footnote{Notice that $g_{\pm}$ depend upon the values of $\ka_0,r_0$ and
$h_{\pm}^b$ depend upon the values of $\ka_0,r_b$, even though this is not
indicated explicitly in the notation. The same comment applies to the numbers
$\la_{\pm}(b)$.}

\begin{rem}\label{R:shift} Let $c>b>0$. Then $h_{\pm}^c$ is obtained from
	$h_{\pm}^b$ by a shift of the parameter through $c-b$, that is,
	$h^c_{\pm}(x)=h^b_{\pm}\big(x-(c-b)\big)$ for all $x\in \R$. The
	monotonicity of $h^b_{\pm}$ implies that $h_{+}^b\leq h_{+}^c$ and
	$h_{-}^c\leq h_{-}^b$ throughout $\R$, whence 
	\begin{equation}\label{E:sandwich}
		\la_-(c)\leq \la_-(b)\leq \la_+(b)\leq \la_+(c).
	\end{equation}
\end{rem}

\begin{lem}\label{L:1andahalf} Let $f\colon [0,b]\to \R$ be as in
	\nref{N:stretchable}. Then
	\begin{equation}\label{E:1andahalf} \la_-(b)\leq
		\max\se{g_-(x),h_-^b(x)}\leq f(x)\leq \min\se{g_+(x),h_+^b(x)}\leq
		\la_+(b)\quad \text{for all $x\in [0,b]$}.  
	\end{equation} 
	\end{lem} 
	\begin{proof} The
	innermost inequalities were already established in (26) of \cite{SalZueh1}.
	The other two are immediate from the definition of  $\la_{\pm}(b)$ and the
	monotonicity of $g_{\pm}$, $h_{\pm}^b$.  \end{proof}

%

\begin{defn}[$ \ze_{(\mu,b)} $] \label{D:ze} For $b>0$ and $\mu\in [\la_-(b),\la_+(b)]\cap \R$,
	define $\ze_{(\mu,b)}\colon [0,b]\to \R$ by \begin{equation}\label{E:ze}
		\ze_{(\mu,b)}(x)=\midd\big(h^b_-(x)\,,\,g_-(x)\,,\,\mu\,,\,g_+(x)\,,\,h^b_+(x)\big)\quad
		(x\in [0,b]).  \end{equation} \end{defn}

Notice that by monotonicity of $g_{\pm}$, $h_{\pm}^b$,
\begin{equation}\label{E:infsup} \inf_{x\in [0,b]}
	\ze_{(\mu,b)}(x)=\min\se{r_0,\,r_b,\,\mu}\text{\ \ and\ \ }\sup_{x\in [0,b]}
	\ze_{(\mu,b)}(x)=\max\se{r_0,\,r_b,\,\mu}.  \end{equation}

\begin{figure}[ht] \begin{center} \includegraphics[scale=.11]{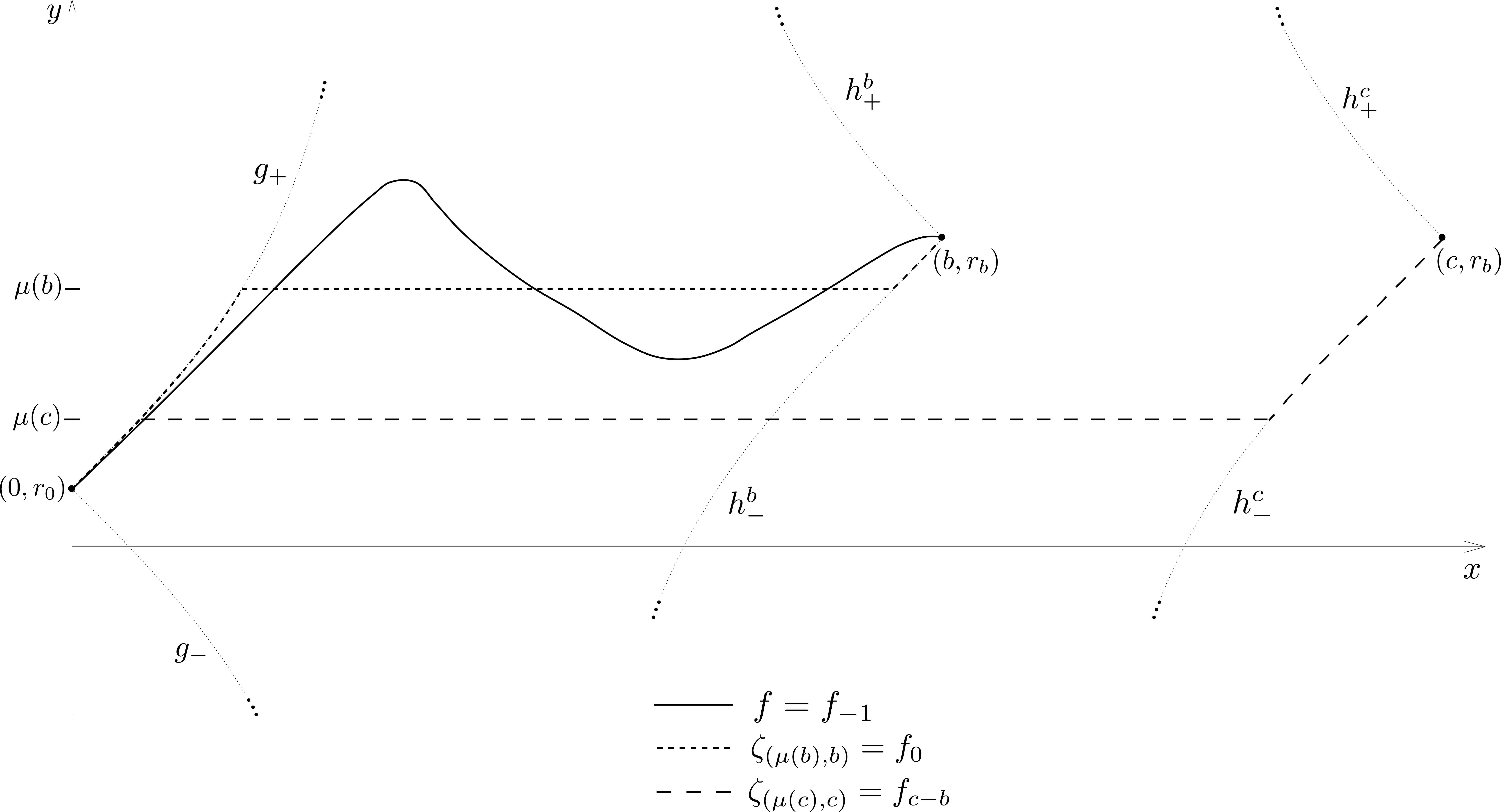} \caption{} \label{F:stretching} \end{center} \end{figure}

\begin{lem}\label{L:six} Let $\mu_1<\mu_2\in
	[\la_-(b),\la_+(b)]\cap \R$. Then $\ze_{(\mu_1,b)}(x)\leq
	\ze_{(\mu_2,b)}(x)$ for all $x\in [0,b]$ and strict inequality holds for at
	least one $x$. In particular, $\int \ze_{(\mu,b)}$ is a strictly increasing
	function of $\mu\in [\la_-(b),\la_+(b)]$.  
\end{lem} 
\begin{proof} Left to the reader.
\end{proof}

\begin{lem}\label{L:seven} Let $c\geq b$ and $\mu\in [\la_-(b),\la_+(b)]\cap \R$.
	Then $\int \ze_{(\mu,c)}-\int \ze_{(\mu,b)}=\mu(c-b)$ and
	$\ze_{(\mu,c)}^{-1}(\se{\mu})$ is a closed interval of length at least
	$c-b$.  
\end{lem} 
\begin{proof} Notice that $\ze_{(\mu,c)}$ is defined (that is, $\mu \in
	[\la_-(c),\la_+(c)]$) by \eqref{E:sandwich}. The inverse image of $\mu$ under
	$\ze_{(\mu,b)}$ is a (possibly degenerate) closed interval
	$[x_0,x_1]$. By (\ref{R:shift}), \begin{equation}\label{E:explicit}
		\ze_{(\mu,c)}(x)=\begin{cases} \ze_{(\mu,b)}(x) & \text{ if $x\in
			[0,x_0]$}; \\ \mu & \text{ if $x\in [x_0,x_1+(c-b)]$};\\
			\ze_{(\mu,b)}(x-(c-b)) & \text{ if $x\in [x_1+(c-b),c]$}. \\
		\end{cases} \end{equation} The assertions of the lemma are consequences
	of this expression.  \end{proof}

\begin{cor}\label{C:eight} Let $b>0$, $\mu(b)\in [\la_-(b),\la_+(b)]\cap \R$ be
	fixed and $A=\int \ze_{(\mu(b),b)}$. Suppose that $0\in
	[\la_-(b),\la_+(b)]$. Then for each $c\geq b$, there exists a unique
	$\mu(c)\in \R$ such that $\int \ze_{(\mu(c),c)}=A$. The resulting function
	$\mu\colon [b,+\infty)\to \R$, $c\mapsto \mu(c)$, is continuous and
		$\abs{\mu(c)}\decr 0$ as $c\to +\infty$. Moreover, $ \abs{\mu(c)} $ is
		strictly decreasing if $ \mu(b)\neq 0 $. 
\end{cor} 
\begin{proof} 
	From $0\in [\la_-(b),\la_+(b)]$ and  (\ref{R:shift}), it follows that $0\in
	[\la_-(c),\la_+(c)]$ for all $c\geq b$. No generality is lost in assuming
	that $\mu(b)\geq 0$. In this case, (\ref{L:six}) and (\ref{L:seven}) yield:
	\begin{equation*}\label{E:intintint} \int \ze_{(0,c)}=\int \ze_{(0,b)}\leq
		A=\int \ze_{(\mu(b),b)} \leq \int \ze_{(\mu(b),b)}+(c-b)\mu(b)=\int
		\ze_{(\mu(b),c)}.  
	\end{equation*} 
	Hence, by (\ref{L:six}), there exists
	a unique $\mu(c)\in [0,\mu(b)]$ such that $\int \ze_{(\mu(c),c)}=A$.
	Moreover, $\mu(c)\in (0,\mu(b))$ in case $\mu(b)>0$, because then the first
	two inequalities above are strict. The same argument also
	shows that $\mu(d)\in [0,\mu(c)]$ whenever $d\geq c$ (and $ \mu(d)\in
	(0,\mu(c)) $ in case $ \mu(b)>0 $).  Thus, $\mu(c)$ is a decreasing function
	of $c$ (strictly decreasing if $ \mu(b)>0$), $\la=\lim_{c\to +\infty}\mu(c)$
	exists and is nonnegative. The continuity of $c\mapsto \mu(c)$ follows from
	the fact that $ \int \ze_{(\mu(c),c)} = A$ is constant as a function of $ c
	$. Finally, $ \la = 0 $ because 
	\begin{equation*}
		A=\int \ze_{(\mu(c),c)}\geq \int
			\ze_{(\la,c)}=\la(c-b)+\int \ze_{(\la,b)}\text{\quad for all $c\geq
			b$}.  \qedhere\end{equation*} 
\end{proof}

\begin{defn}[flattening and stretching functions]\label{D:functionstretching}
	Let $f=f_{-1}\colon [0,b]\to \R$ be as in \nref{N:stretchable}. The
	\tdef{flattening} of $f$ is the family $f_s\colon [0,b]\to \R$ ($s\in
	[-1,0]$) obtained by
	applying Construction 3.8 in \cite{SalZueh1} to $f$ (note that there $s$
	goes from 1 to 0, instead of from $-1$  to 0 as here). 
	
	We say that $f$ is \tdef{$\ka_0$-stretchable} if $0\in [\la_-(b),\la_+(b)]$.
	In this case, the \tdef{stretching} of $f$ is the extension of the above
	family to $s\in [-1,+\infty)$ obtained by setting
	$f_s=\ze_{(\mu(b+s),b+s)}$ for $ s\geq 0 $, where $\mu\colon [0,+\infty)\to
	\R$ is as in \cref{C:eight}. See \fref{F:stretching}.
\end{defn}

\begin{lem}\label{L:regularity} Let $f_{-1}=f\colon [0,b]\to \R$ be a 
$\ka_0$-stretchable function and $(f_s)_{(s\in [-1,+\infty))}$ be the stretching
of $f$. Then: 
\begin{enumerate} 
	\item [(a)] If $f$ is piecewise smoooth, then so is $f_s$ for all $s\in
	[-1,+\infty)$.  
	\item [(b)] $\sup \abs{f_s}$ is a decreasing function of $s$.  
	\item [(c)] If $f$ does not
			change sign inside its domain, then none of the $f_s$ do.  
	\item [(d)] For $s\in [-1,0]$, $\sup_{[0,b]} f_s$ \tup(resp.~$\inf_{[0,b]}
		f_s$\tup) is a decreasing \tup(resp.~increasing\tup) function of $s$.
	\item [(e)] Let $f_s=\ze_{(\mu(b+s),b+s)}$ $(s\geq 0)$ and let $L_s$ denote
		the length of the interval \begin{equation*}
			\set{x\in [0,b+s]}{f_s(x)=\mu(b+s)}.  \end{equation*} Then $L_s\sim
		s$ \tup(that is, $\lim_{s\to +\infty}\frac{L_s}{s}=1$\tup).  
	\item [(f)] There exists $\vka_2>0$ such that 
		\begin{equation*}
			\abs{\mu(b+s)}\leq \frac{\vka_2}{s+1}\quad \text{for all $s\geq 0$}.
		\end{equation*} 
		Moreover, if $f>0$ over $[0,b]$, then there also exists
		$\vka_1>0$ such that 
		\begin{equation*}
			\frac{\vka_1}{s+1}\leq \mu(b+s)\quad \text{for all $s\geq 0$}.  
	\end{equation*} 
	\item [(g)] $f_s$ is
		$\ka_0$-stretchable for all $s\in [-1,+\infty)$.  \end{enumerate}
\end{lem} \begin{proof}The proof will be split into the corresponding parts. 

	(a): By definition, $f_s$ is the median of a finite collection of piecewise 
	smooth functions for all $s$.
	
	(b): For $s\geq 0$, this follows from \cref{C:eight}. For $s\in [-1,0]$,
	this follows from Corollary 3.12 of \cite{SalZueh1}.
	

	(c):  No generality is lost in assuming that $f=f_{-1}\geq 0$ over $[0,b]$.
	Then, by Corollary 3.12 of \cite{SalZueh1}, \begin{equation*}
	0\leq
	\inf_{x\in [0,b]}f_{-1}(x)\leq \inf_{x\in [0,b]}f_s(x)\text{\ \ for all
	$s\in [-1,0]$}.  \end{equation*} Let $r_0=f(0)$, $r_b=f(b)$; both are
	nonnegative by hypothesis.  By \eqref{E:infsup}, \begin{equation*}
	\inf_{x\in [0,b+s]}f_s(x)=\inf_{x\in [0,b+s]}
	\ze_{(\mu(b+s),b+s)}(x)=\min\se{r_0,\,r_b,\,\mu(b+s)} \text{\ \ for all
	$s\geq 0$}.  \end{equation*} By \cref{C:eight}, 
	$\mu(b+s)\geq 0$ for all $s\geq 0$, hence $ f_s\geq 0 $ for all $ s\geq 0 $. 

	(d): This was proved in Corollary 3.12 of \cite{SalZueh1}.

	(e): The functions $g_{\pm}$, $h_{\pm}^b$ all blow up to $\pm \infty$ in finite
	time (compare eqs.~(24) and (25) of \cite{SalZueh1}). The length of $[0,b+s]$ is
	asymptotically equal to $s$, hence  $L_s\sim s$ as well. 
	
	(f): Let $L_s$ be as in part (e). We can write 
	\begin{equation*}
		A=\int f_s=\int_{\se{f_s=\mu(b+s)}}f_s + \int_{\se{f_s\neq
		\mu(b+s)}}f_s = L_s\?\mu(b+s) + \int_{\se{f_s\neq \mu(b+s)}}f_s.
	\end{equation*}
	A straightforward calculation shows that the improper integrals of 
	$g_{\pm}$ and $ h_{\pm}^{b} $ over the respective intervals where these
	functions
	assume real (finite) values are all finite. Hence the last term in the
	preceding equation admits a bound independent of $s$. The first assertion
	thus follows from (f). 

	The proof of the second assertion is similar. If $f>0$, then $\mu(b)>0$ and
	\begin{equation*}
		\int \ze_{(0,b)}<A=\int f_0=\int f_s\leq L_s\?\mu(b+s) +
		\int \ze_{(0,b)},
	\end{equation*}
	so again the assertion follows from (f).
	

(g): This is immediate from \eqref{E:sandwich}.
\end{proof}

\subsection*{Stretching of curves} We shall now reinterpret the preceding
definitions and results in terms of planar curves. Let $P=(p,w),\,Q=(q,z)\in \C\times \Ss^1$ and
	$\ga\in \sr L_{-1}^{+1}(P,Q)$ (see \S1 of \cite{SalZueh1} for the definition of
	this space). Recall that $\ta_\ga\colon [0,1]\to \Ss^1$ denotes the unit
	tangent to $\ga$.
\begin{defn}[stretchable curve]\label{D:stretchablecurve}
	Suppose that $\lan\ta_\ga,e^{i\psi}\ran>0$ throughout $[0,1]$ ($
	\psi\in \R $). After translating $ p $ to the origin, rotating $ \C $
	about the latter through $ \psi $, and relabeling the $x$- and $y$-axes
	accordingly, $\ga$ may be reparametrized as $\ga(x)=(x,y(x))$ for $x$ in
	some interval $[0,b]$. Let
	$f=y'\colon [0,b]\to \R$. We call $\ga$ \tdef{$\ka_0$-stretchable \tup(with
	respect to $e^{i\psi}$\tup)} if $f$ is $\ka_0$-stretchable in the sense
	of \dref{D:functionstretching}. Also, $\ga$ will be called
	\tdef{stretchable \tup(with respect to $e^{i\psi}$\tup)} if it is
	$\ka_0$-stretchable for some $\ka_0\in (0,1)$. 
\end{defn}

\begin{rem}\label{R:tangent} In this context, $f(x)=\tan(\theta_\ga(x))$
	for all $x\in [0,b]$, where $ \theta_\ga $ measures the angle from $
	e^{i\psi}$ to $ \ta_\ga $. Condition (i) in \nref{N:stretchable} means that the
	curvature $\ka_\ga$ of $\ga$ satisfies $\abs{\ka_\ga}\leq \ka_0$ almost everywhere.
	The numbers $r_0$, $r_b$ in (ii) represent the slopes of $w$, $z$,
	respectively, $A=\Im(q-p)$ and $b=\Re(q-p)$ (all of these with respect to
	the new coordinate axes determined by $e^{i\psi}$ and $ie^{i\psi}$). The reader
	may have noticed that the condition of being stretchable does not really
	concern $\ga$, but rather the pair $(P,Q)$. The geometric interpretation is
	that curves in $ \sr L_{-1}^{+1}(P,Q) $ are $ \ka_0 $-stretchable if and 
	only if there exists a curve $ \eta $ in this space such that $
	\abs{\ka_{\eta}}\leq \ka_0 $ a.e., $ \lan \ta_\eta, e^{i\psi}\ran>0 $
	everywhere and $ \ta_\eta(t_0)=e^{i\psi} $  for some
	$ t_0\in [0,1] $. 
	\end{rem}

\begin{defn}[flattening and stretching curves]\label{D:stretching} 
	Let $\ga$ be a $\ka_0$-stretchable curve as in \dref{D:stretchablecurve}
	and let $(f_s)_{s\in [-1,+\infty)}$ be 
	the corresponding stretching of $f=f_{-1}$, as in \dref{D:functionstretching}.
	Let $\ga_s\colon J_s\to \C$ be defined
	by 
	\begin{equation}\label{E:dictionary}
		\ga_s(x)=\Big(x\,,\,y(0)+\int_{J_s}f_s(u)du\Big),~\text{where }J_s=\begin{cases} [0,b] & \text{ if $s\in [-1,0]$;}
		\\ [0,b+s] & \text{ if $s\geq 0$.} \end{cases} 
	\end{equation} 
	The family $(\ga_s)_{(s\in [-1,+\infty))}$ will be called the 
	\tdef{stretching} of $\ga$ with respect to $e^{i\psi}$, and the 
	family $(\ga_s)_{(s\in [-1,0])}$ the \tdef{flattening} of $\ga$ with 
	respect to $e^{i\psi}$. The \tdef{stretching of $ \ga $ by $M$} is the family $(\ga_s)_{(s\in [-1,M])}$, 
	$M>0$; see Figure \ref{F:curvestretching}.
\end{defn} 

\begin{figure}[ht] \begin{center} \includegraphics[scale=.13]{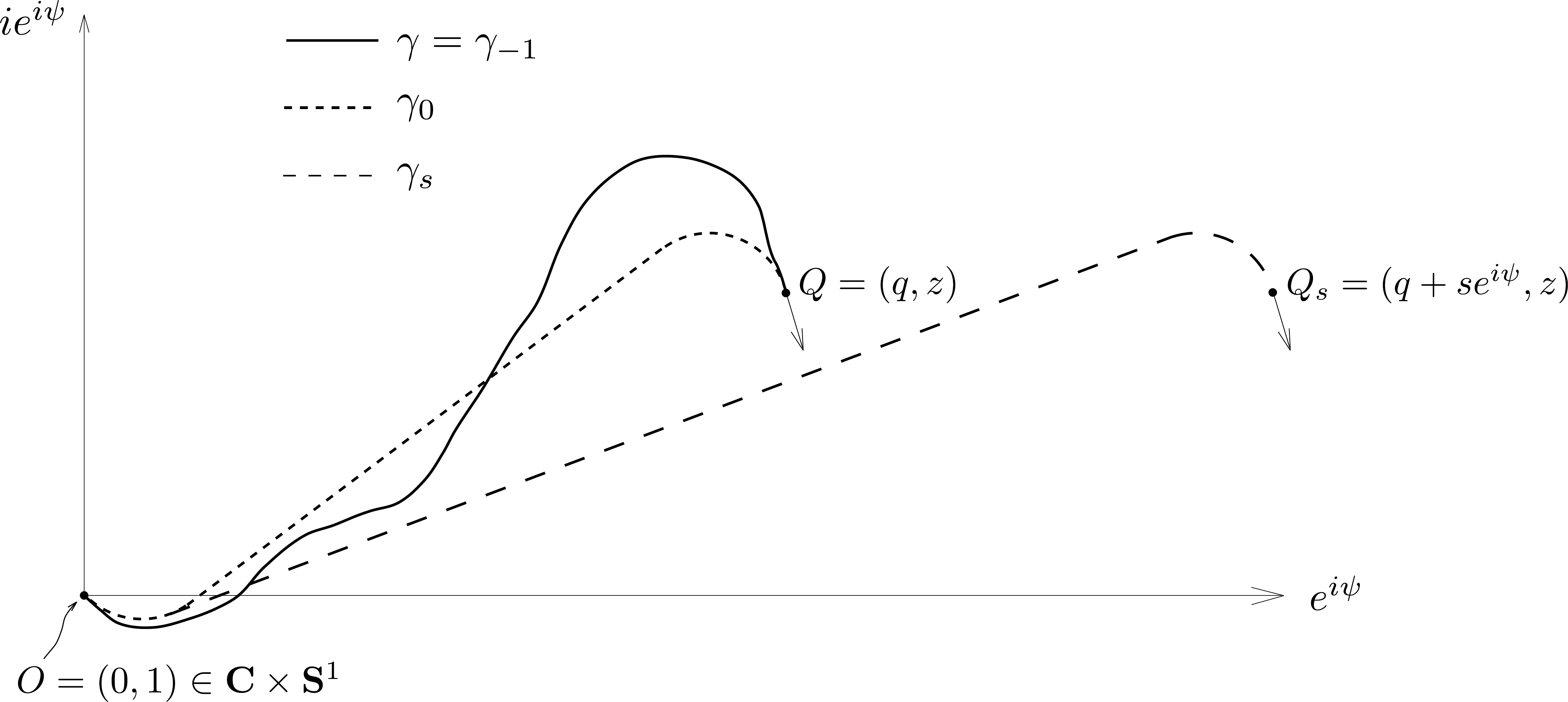}
		\caption{Flattening and stretching a curve $\ga$ in the direction of
			$e^{i\psi}$.}
	\label{F:curvestretching} \end{center} \end{figure}

Notice that $\ga_s\in \sr L_{-1}^{+1}(P,Q_s)$ for $Q_s=(q_s,z)\in \C\times \Ss^1$,
where $q_s=q$ for all $s\in [-1,0]$ and $q_s=q+se^{i\psi}$ for $s\geq 0$. The
curves $\ga_s$ are independent of the starting curve $\ga=\ga_{-1}$ for $ s\geq
0 $ (and fixed $\psi$ and $\ka_0$); they are each a concatenation of an arc of
circle of curvature $\pm \ka_0$, a straight line segment, and another such arc,
where both arcs have amplitude at most $\pi$ (cf.~Figure
\ref{F:curvestretching}). The functions $ g_\pm $, $ h^b_\pm $ appearing above
correspond to the arcs of circles of curvature $ \pm \ka_0 $ 
starting (resp.~ending) at $ P $ (resp.~$ Q
$).  In vague but suggestive language, the family $(\ga_s)$ is obtained from
$\ga$ by ``stretching'' it in the direction of $e^{i\psi}$. Clearly, the
stretching and flattening of a curve $\ga$ depend upon the chosen axis $\psi$.
Nonetheless, the curve $\ga_0$ is independent of both $ \ga $ and $ \psi $; see
Remark 3.9 of \cite{SalZueh1}.

\begin{uexr}
	Translate the assertions of \lref{L:regularity} into
	statements about the curves $\ga_s$, using that $
	f_s(x)=\tan(\theta_{\ga_s(x)}) $. (For instance, part (b) states that 
	$\sup_{x}\abs{\theta_{\ga_s}(x)}$ is a decreasing function of $s$.)
\end{uexr}

\begin{lem}\label{L:lowering}
	Let $\ga\in \sr L_{-1}^{+1}(P,Q)$. Suppose that $\abs{\ka_\ga}\leq
	\ka_0$ a.e.~and $\lan\ta_\ga,e^{i\psi}\ran>0$ over $ [0,1] $.
	\begin{enumerate}
		\item [(a)] If $\ta_\ga(t_0)=e^{i\psi}$ for some $t_0\in [0,1]$, then
			$\ga$ is $\ka_0$-stretchable with respect to $e^{i\psi}$.
		\item[(b)] Suppose that $
			e^{i\psi}\nin \ta_\ga([0,1]) $ and that $ \ga $ is $ \ka_0 $-stretchable with
			respect to $ e^{i\psi} $. Then $ \ga $ is $ \ka_{0} $-stretchable
			with respect to any  $ z $ lying in the shortest arc \tup(in $ \Ss^1 $\tup) 
			joining $ e^{i\psi}
			$ to $ \ta_\ga([0,1])$. 
		\item [(c)] If $ \lan q-p,e^{i\psi}\ran $ is sufficiently large, then
			any $ \ga $ as above is $ \ka_0 $-stretchable with respect to $
			e^{i\psi} $. 
		\item [(d)] If $I\subs [0,1]$ is an interval and $\ga|_I$ is
			$\ka_0$-stretchable with respect to $ e^{i\psi} $, then so is $\ga$.
		\item [(e)]	If $ \ga_s $ is the flattening of $ \ga $ and 
			$\ga|_I$ is a line segment of length $L>2\pi$, then there exists a
			subinterval $I'$ such that $\ga_s|_{I'}$ is a line segment of length
			$>L-2\pi$ for all $s\in [-1,0]$. 
		\item [(f)] If $\ga$ is a line segment of length greater than
			$\frac{4}{\ka_0}$, then $\ga$ is $\ka_0$-stretchable with respect to
			$ e^{i\psi} $.
	\end{enumerate}
\end{lem}

\begin{proof}The proof of each part is given separately; in all of them, $ f=y'$
	is as in \dref{D:stretchablecurve}.

	(a): This is clear from the geometric interpretation described in
	\rref{R:tangent}. Alternatively, in terms of $ f $, the
	hypothesis means that there exists some $ x_0\in [0,b] $ satisfying $
	f(x_0)=0 $. Hence $ 0\in [\la_-(b),\la_+(b)] $ by \lref{L:1andahalf}, so
	that $ f $ is $ \ka_0 $-stretchable.

	(b): In terms of \rref{R:tangent}, the hypothesis means that there exists some $ \eta\in \sr
	L_{-1}^{+1}(P,Q) $ such that $ \abs{\ka_\eta}\leq \ka_0 $ a.e., $ \lan
	\ta_\eta,e^{i\psi} \ran>0 $ throughout, and the image of $ \ta_\eta $
	includes $ e^{i\psi} $. As proved in Remark 3.9 of
	\cite{SalZueh1}, the flattenings $\ga_0$ and $\eta_0$ of $ \ga $, $ \eta $ 
	with respect to $e^{i\psi}$ are the same curve. Therefore, there exists a
	homotopy $ s\mapsto \al_s\in \sr L_{-1}^{+1}(P,Q) $ such that $ \al_0=\ga
	$, $ \al_1=\eta $, $ \abs{\ka_{\al_s}}\leq \ka_0 $ and $ \lan
	\ta_{\al_s},e^{i\psi}\ran>0 $ for each $ s\in [0,1] $. Let $ s_0 $ be the
	smallest $ s\in [0,1] $ for which $ z\in \ta_{\al_s}([0,1]) $. Then $ \lan
	\ta_{\al_{s_0}}(t)	,z\ran>0$ for all $ t\in [0,1] $, as is readily
	verified, hence any curve in $ \sr L_{-1}^{+1}(P,Q) $ is $ \ka_{0}
	$-stretchable with respect to $ z $.
	
	(c): Since the functions $ g_{\pm},h_{\pm}^b $ go to $ \pm\infty $ in finite
	time which is independent of $ b =\lan q-p,e^{i\psi}\ran $, if the latter is 
	large enough, then we shall have $ \la_-(b)=-\infty $ and $ \la_+(b)=+\infty
	$, so that certainly $ 0\in [\la_-(b),\la_+(b)] $.

	(d): Let $ I=[c,d] $ and denote by $ \bar{g}_\pm $ and $ \bar{h}_\pm $ the
	solutions to the differential equations  $g'=\pm
	\ka_0(1+g^2)^{\frac{3}{2}}$ and $ h' =\mp
	\ka_0(1+h^2)^{\frac{3}{2}}$ respectively, with $g_\pm(c)=f(c)$ and 
	$h_\pm(d)=f(d)$. If $ \bar\la_{\pm} $ denote the common values of 
	$ \bar{g}_+,\bar{h}_+ $ and $ \bar{g}_-,\bar{h}_- $, respectively, at the
	points where their graphs intersect, then the
	hypothesis means that $ 0\in [\bar\la_-,\bar\la_+] $. This implies that $ 0\in
	[\la_-(b),\la_+(b)] $ because $ \bar{g}_+\leq g_+ $, $ \bar{g}_-\geq g_- $,
	$ \bar{h}_+\leq h_+ $ and $ \bar{h}_-\geq h_- $. These inequalities 
	follow from the fact that the graph of $ f $ stays within the region bounded
	by the graphs of $ g_{\pm} $ and $ h_\pm $, since $\abs{f'(x)}\leq
	\ka_0\big[1+f(x)^2\big]^{\frac{3}{2}}$ for almost every $x\in [0,b]$.

	(e): This is immediate from \cite{SalZueh1}, Construction 3.8.  

	(f): If $ \ka_0=1 $ this follows from  \fref{F:generator}, which shows that
	there exists $ \eta \in \sr L_{-1}^{+1}(P,Q)$ such that 
	$ \abs{\ka_\eta}\leq \ka_0$ a.e.~, $\lan \ta_\eta,e^{i\psi}\ran>0$
	over $ [0,1] $ and  $e^{i\psi}\in \ta_\eta([0,1]) $, provided
	that $ \abs{q-p}>4$; the latter inequality holds by hypothesis. 
	  For other values of $ \ka_0 $, just apply a dilation.
\end{proof}

The details of the
construction of the family $(\ga_s)$ may be now be safely forgotten. Only the properties
listed in \lref{L:regularity} and \lref{L:lowering} will be used.

\section{Quasicritical curves}\label{S:quasicritical} 
\begin{unotn} 
	Throughout
	the rest of the paper, $Q=(q,z)$ denotes a fixed element of
	$\C\times \Ss^1\equiv UT\C$ with $z\neq -1$. For our purposes, it is more convenient to work with
	the space $\sr L_{-1}^{+1}(Q)$ (see \S1 of \cite{SalZueh1})  instead of the
	space $\sr C_{-1}^{+1}(Q)$ defined in the introduction; these are
	homeomorphic by Lemma 1.12 in \cite{SalZueh1}. Accordingly,  $\sr M(Q)$
	shall denote the subspace $\sr L_{-1}^{+1}(Q;\theta_1)\subs \sr
	L_{-1}^{+1}(Q)$ with $\abs{\theta_1}<\pi$. 
	 
	 Let $\ga\in \sr M(Q)$. We denote by $\theta_\ga\colon [0,1]\to
	 \R$ the unique continuous function satisfying
 $\exp(i\theta_\ga)=\ta_\ga$ and $\theta_\ga(0)=0$. Also, 
	 \begin{equation}\label{E:average}
		 \bar{\vphi}^\ga:=\frac{1}{2}\Big(\max_{t\in
		 [0,1]}\theta_\ga(t)+\min_{t\in [0,1]}\theta_\ga(t)\Big).
	 \end{equation} Finally, given $\vphi\in \R$, it will be very convenient to
	 use the abbreviations $\vphi_{\pm}:=\vphi\pm
	 \frac{\pi}{2}$.  \end{unotn}

\subsection*{Quasicritical curves} The central definition of this paper is the
following generalization of the concept of critical curves.

\begin{defn}[quasicritical curve]\label{D:quasicritical} Let $\sig$ be a sign string, $n=\abs{\sig}$,
	$\ga\in \sr M(Q)$, $\vphi\in \R$ and $\eps\in (0,\frac{\pi}{4})$. Then $\ga$
	is $(\vphi,\eps)$-\tdef{quasicritical of type} $\sig$ if there exist closed
	intervals $J_1<\dots<J_n$ such that for each $k\in [n]$: \begin{enumerate}
		\item [(i)] $\theta_\ga(J_k)\subs (\vphi_-+2\eps,\vphi_++\eps)$ if
		$\sig(k)=\ty{+}$ and $\theta_\ga(J_k)\subs (\vphi_--\eps,\vphi_+-2\eps)$
	if $\sig(k)=\ty{-}$; \item [(ii)]
		$\abs{\theta_\ga(t)-\vphi}<\frac{\pi}{2}-2\eps$ for all $t\nin
		\Int\big(\bcup_{k=1}^nJ_k\big)$; \item [(iii)]  $J_k$ contains at least
			one closed subinterval $I_k$ such that
			$\vert\theta_\ga(t)-\vphi_{\sig(k)}\vert<\eps$ for all $t\in I_k$
			and $\ga|_{I_k}$ is stretchable with respect to $\vphi_{\sig(k)}$.
	\end{enumerate} \end{defn} Condition (i) means that $\ta_\ga$ is far from
	$\mp ie^{i\vphi}$ throughout $J_k$ if $\sig(k)=\pm$, while (iii) states
	roughly that there should exist a subinterval of $J_k$ where $\ta_\ga$ is
	vertical enough with respect to the axis $e^{i\vphi}$ to allow $\ga$ to be
	stretched in the direction of $\sig(k)ie^{i\vphi}$. Outside of $\bcup J_k$,
	$\ta_\ga$ is far from both $ie^{i\vphi}$ and $-ie^{i\vphi}$. 

\begin{rem}\label{R:automaticineq} The combination of (i) and (ii) in
	\dref{D:quasicritical} implies that $\theta_\ga([0,1])\subs
	(\vphi_--\eps,\vphi_++\eps)$.  \end{rem}

\begin{rem}\label{R:opencond} Being quasicritical of type $\sig$ is an open
	condition on $(\ga,\vphi,\eps)$. In fact, the same intervals $J_k$ satisfy
	(i)--(iii) for the triple $(\eta,\psi,\de)$ if the latter is close enough to
	$(\ga,\vphi,\eps)$.  \end{rem} \begin{lem}\label{L:criticalisquasi} Let
	$\ga\in \sr M(Q)$ be a critical curve of type $\sig$. Then $\ga$ is
	$(\bar\vphi^\ga,\eps)$-quasicritical of type $\sig$ for all sufficiently
	small $\eps>0$.  \end{lem} \begin{proof} Immediate from \rref{L:lowering}\?(a)
	and the definition of critical curves, given in \dref{D:basic}
\end{proof}

We will sometimes abuse the terminology by saying that $I$ is a stretchable
interval for $\ga$ if $\ga|_I$ is stretchable (with respect to $\vphi_{\pm}$).
Notice that there is a lot of freedom in the choice of the intervals $J_k$ and
their stretchable subintervals. The next two results compensate for this
ambiguity.

\begin{lem}\label{L:Js} Let $\ga\in \sr M(Q)$ be $(\vphi,\eps)$-quasicritical of
type $\sig$, $n=\abs{\sig}$.  \begin{enumerate} \item [(a)] Let $0<\de\leq
			2\eps$ and $W_\al\subs [0,1]$ $(\al \in A)$ be all the connected
			components of \begin{equation*}
				W=\set{t\in
				[0,1]}{\abs{\theta_\ga(t)-\vphi}>\tfrac{\pi}{2}-\de}.
			\end{equation*} Then there exists a decomposition $A=A_1\du\dots\du
			A_n$ such that for any choice of $J_1<\dots<J_n$ as in
			\dref{D:quasicritical}, $W_\al\subs J_k$ if and only if $\al\in A_k$
			$(k\in [n])$.  \item [(b)]  Let $J_1^{j}<\dots<J_n^{j}$,
				$J_k^{j}=[a_k^{j},b_k^{j}]$, be as in \dref{D:quasicritical}
				$(j\in [m])$. For each $k\in [n]$, set $a_k'=\max_j a_k^j$ and
				$b_k'=\min_j b_k^j$. Then the intervals $J'_k=[a_k',b_k']$ also
				satisfy (i)--(iii).  \item [(c)] Let $J_1'<\dots<J_n'$ be as in
					\dref{D:quasicritical} and $J_1<\dots<J_n$ be such that
					$J_k\sups J_k'$ for each $k\in [n]$. Then the $J_k$ also
					satisfy (i)--(iii).  \end{enumerate} \end{lem}
	\begin{proof}The proof of each part will be given separately.

	(a): Let $J_1<\dots<J_n$, $J_1'<\dots<J_n'$ be intervals as in
	\dref{D:quasicritical}. Set $A_k=\set{\al\in A}{W_\al\subs J_k}$. Then
	$A=A_1\du\dots\du A_n$ since (ii) of \dref{D:quasicritical} implies that any
	$W_\al$ must be completely contained in some $J$. We claim that $A_k'=A_k$
	for each $k\in [n]$, where $A_k'=\set{\al\in A}{W_\al\subs J'_k}$. This
	follows from the following simple observations (which also hold with $A'$ in
	place of $A$): \begin{enumerate} \item [\sbu] Each $A_k$ is nonempty, by
			(iii) of \dref{D:quasicritical}.  \item [\sbu] If $\al\in A_k$,
			$\al'\in A_{k'}$ with $k<k'$, then $W_\al<W_{\al'}$; indeed,
		$J_k<J_{k'}$.  \item [\sbu]  If $\al\in A_k$, then
			$\sign(\theta_\ga(t)-\vphi)=\sig(k)$ for all $t\in W_\al$, by (i) of
			\dref{D:quasicritical}.  \end{enumerate} Suppose that $\al\in
	A_1\cap A_k'$ for some $k>1$. Then the third observation implies that $k\geq
	3$. Choose $\be\in A_2'$. By the second observation, $W_\be<W_\al$. Hence
	$\be\in A_1\cap A'_2$, contradicting the third observation. It follows that
	$A_1=A_1'$. An entirely similar argument shows that if $A'_j=A_j$ for all
	$j\in [k]$, then $A'_{k+1}=A_{k+1}$ as well.
	
	(b): Let $j_0,j_1\in [m]$ be such that $a_k'=a_k^{j_0}$ and
	$b_k'=b_{k}^{j_1}$. By part (a), if $\al\in A_k$, then $W_\al\subs
	J_k^{j_0}\cap J_k^{j_1}$. In particular, $a_k'<b_k'$ and
	\begin{equation*}\label{E:intersections0} [a_k',b_k']= J_k^{j_0}\cap
		J_k^{j_1}.  \end{equation*} Since the latter two intervals satisfy
	condition (i) by hypothesis, so does $[a_k',b_k']$. Set $\de=2\eps$ in the
	definition of $W$. If $I$ is a stretchable subinterval of $J_k^{j_0}$ as in
	(iii), then $I\subs W_\al$ for some $\al\in A_k$. By (a), $W_\al\subs
	J_k^{j_0}\cap J_k^{j_1}=[a_k',b_k']$, hence the latter satisfies (iii).  To
	establish (ii), let $j_2\in [m]$ be such that $a_{k+1}'=a_{k+1}^{j_2}$. As
	above, part (a) implies that $a_{k}^{j_2}<b_k'$ and
	$a_{k+1}'<b_{k+1}^{j_1}$. Hence \begin{equation}\label{E:intersections}
		[b_{k}',a_{k+1}']= \big[b_k',b_{k+1}^{j_1}\big]\cap
		\big[a_k^{j_2},a_{k+1}'\big].  \end{equation} Moreover,
	\begin{equation*}
	\big[b_k',b_{k+1}^{j_1}\big]=
		\big[b_k',a_{k+1}^{j_1}\big]\cup J_{k+1}^{j_1}\quad \text{and}\quad
		\big[a_k^{j_2},a_{k+1}'\big]= J_k^{j_2}\cup [b_{k}^{j_2},a_{k+1}'\big].
	\end{equation*} By (i) and (ii) of \dref{D:quasicritical}, any $t\in
	\big[b_k',b_{k+1}^{j_1}\big]$ thus satisfies
	$\vert\theta_\ga(t)-\vphi_{-\sig(k+1)}\vert>2\eps$ and any $t\in
	\big[a_k^{j_2},a_{k+1}'\big]$ satisfies
	$\vert\theta_\ga(t)-\vphi_{-\sig(k)}\vert>2\eps$. Together with
	\eqref{E:intersections}, this implies that (ii) holds for the $J_k'$.

(c): Conditions (ii) and (iii) of \dref{D:quasicritical} are obviously satisfied
by the $J_k$. Suppose that $ t\in J_k $ but $
\abs{\theta_\ga(t)-\vphi_{-\sig(k)}}< 2\eps $. Then $ t\in J_{k'}' $ with $
\sig(k')=-\sig(k) $, contradicting the fact that $ J_k $ and $J_{k'}\sups J'_{k'}$ are
disjoint.
\end{proof}

\begin{unotn} In all that follows, $K$ denotes (the geometric realization of) 
	a finite simplicial complex;
	actually, most of the time all that is required is that $K$ be a compact
	Hausdorff topological space. \end{unotn} \begin{lem}\label{L:continuousJs}
	Let $\sig$ be a sign string of length $n$ and \begin{equation*}
		p\mapsto \ga^p\in \sr M(Q),\ \ p\mapsto \vphi^p\in \R,\ \ p\mapsto
		\eps^p\in \R^+\quad (p\in K) \end{equation*} be continuous maps such
	that $\ga^p$ is $(\vphi^p,\eps^p)$-quasicritical of type $\sig$ for all
	$p\in K$.  Then: \begin{enumerate} \item [(a)] There exist continuous
				functions $a_k,\?b_k\colon K\to [0,1]$ such that for all $p\in
				K$, the intervals $J_k(p)=[a_k(p),b_k(p)]$ \tup($k\in [n]$\tup)
				satisfy \dref{D:quasicritical} when
				$(\ga,\vphi,\eps)=(\ga^p,\vphi^p,\eps^p)$\tup.  \item [(b)]
					There exist an open cover $(U_i)_{i\in [l]}$ of $K$ and real
					numbers $c_{i,k}<d_{i,k}$ $(i\in [l],~k\in [n])$ such that
					for each $p\in \ol{U}_i$ and $k\in [n]$,
					$I_{i,k}:=[c_{i,k},d_{i,k}]\subs J_k(p)$, $\ga^p|_{I_{i,k}}$
					is stretchable with respect to $\vphi^p_{\sig(k)}$ and
					\begin{equation}\label{E:verti}
						\theta_{\ga^p}\big(I_{i,k}\big)\subs
						\big(\vphi^p_{\sig(k)}-\eps^p,\vphi^p_{\sig(k)}+\eps^p).
				\end{equation} \end{enumerate} \end{lem}

\begin{urem}
	The inclusion $I_{i,k}\subs J_k(p)$ in (b) is asserted
	to hold only when $p\in \ol{U}_i$. Nonetheless, it will hold for such $p$
	independently of the choice of the $J_k(p)$ in (a). 
	
	It is generally impossible to obtain globally (and continuously) defined
	intervals $[c_k(p),d_k(p)]$ restricted to which $\ga^p$ is stretchable. The
	problem is similar to that of choosing points $t(p)\in [0,1]$ where a family
	$f^p\colon [0,1]\to \R$ of continuous functions attain their maxima.
\end{urem}

\begin{proof}[Proof of \lref{L:continuousJs}] Let $p\in K$. Choose intervals
	$[a_1,b_1]<\dots<[a_n,b_n]$ satisfying (i) and (ii) of
	\dref{D:quasicritical} for $(\ga,\vphi,\eps)=(\ga^p,\vphi^p,\eps^p)$ and
	subintervals $[c_k,d_k]\subs [a_k,b_k]$ as in (iii). Since these conditions
	are open, they actually hold for the same choice of intervals for all $q$ in
	the closure of some neighborhood  $U_p$ of $p$. Let $(U_i)_{i\in
	[l]}$ be a finite subcover of the cover $(U_p)_{p\in K}$ so obtained, and
	let $a_{i,k},\,b_{i,k},\,c_{i,k},\,d_{i,k}\in [0,1]$ $(i\in [l], k\in [n])$
	be the endpoints of the corresponding intervals.
	
	Let $\rho_i\colon K\to [0,1]$ ($i\in [l]$) form a partition of unity
	subordinate to the cover $(U_i)$, \begin{equation*}
		a_k(p):=\sum_{i=1}^l\rho_i(p)a_{i,k}(p),\ \
		b_k(p):=\sum_{i=1}^l\rho_i(p)b_{i,k}(p)\text{\ \ and\ \
		}J_k(p):=[a_k(p),b_k(p)]\quad (k\in [n]).  \end{equation*} Because
	$a_{i,k}<b_{i,k}<a_{i,k+1}$ for each $i$ and $k$ by hypothesis, the
	definition of $J_k(p)$ makes sense and $J_1(p)<\dots<J_n(p)$ holds for all
	$p\in K$. Now fix $p$ and let $i_1,\dots,i_m\in [l]$ be all the indices $i$
	such that $\rho_i(p)>0$. Set \begin{equation*}
	a_k':=\max_{j\in
		[m]} a_{i_j,k}(p),\ \  b_k':=\min_{j\in [m]} b_{i_j,k}(p)\quad (j\in
		[m]).  \end{equation*} Then $[a_k',b_k']\subs J_k(p)$, hence the
	combination of (b) and (c) of \lref{L:Js} shows that $J_k(p)$ satisfies
	(i)--(iii) for each $k\in [n]$. This proves (a).
	
	Fix $i\in [l]$. By the choice of the intervals $I_{i,k}:=[c_{i,k},d_{i,k}]$,
	the restriction of $\ga^p$ to $I_{i,k}$ is stretchable with respect to $\vphi^p_{\sig(k)}$ and
	\eqref{E:verti} holds whenever $p\in \ol{U}_i$. Again by choice,
	$I_{i,k}\subs [a_{i,k},b_{i,k}]$. Since the $[a_{i,k},b_{i,k}]$ and the
	$J_k(p)$ satisfy (i)--(iii) provided that $p\in \ol{U}_i$, \lref{L:Js}\?(a)
	implies that $I_{i,k}\subs J_k(p)$ for such $p$ and each $k\in [n]$. This
	proves (b).  \end{proof}
	
	\begin{lem}\label{L:disjoint} In the situation of \lref{L:continuousJs}, 
		$(U_i)_{i\in [l]}$ and $I_{i,k}=[c_k^i,d_k^i]$ can be chosen so that: 
	\begin{enumerate} \item [(a)] If $i<i'$ and
			$\ol{U}_{i}\cap \ol{U}_{i'}\neq \emptyset$, then for each $k\in
			[n]$, either $I_{i,k}\subs I_{i',k}$ or $ I_{i,k}\cap
			I_{i',k}=\emptyset $. 
		\item [(b)] For all $k\in [n]$, $i\in [l]$ and $p\in
				\ol{U}_i$, either
				$\big\vert{\theta_{\ga^p}(c_{i,k})-\vphi^p_{\sig(k)}}\big\vert>\frac{1}{2}\eps^p$
				or $c_{i,k}=0$,  and either
				$\big\vert{\theta_{\ga^p}(d_{i,k})-\vphi^p_{\sig(k)}}\big\vert>\frac{1}{2}\eps^p$
				or $d_{i,k}=1$.  
			\end{enumerate} \end{lem}  
	
\begin{urem} The purpose of part (a) is to guarantee that when
	$\ga^p|_{I_{i,k}}$ is stretched for $p\in U_i\cap U_{i'}$, the
	``stretchability'' of $\ga^p|_{I_{i',k}}$ will not be affected. By
	\rref{L:lowering}\?(d) and \lref{L:regularity}\?(g), this can be arranged simply by stretching these arcs
	successively for each $i=1,\dots,l$. Part (b) will be used to
	ensure that stretching $\ga^p$ will not affect its property of being
	quasicritical of type $\tau$ for $ \tau\neq \sig $.
\end{urem}	

\begin{proof} Let $U_i$ be open sets as in \lref{L:continuousJs}, with
	associated stretchable intervals $I_k(U_i):=I_{i,k}\subs J_k(p)$, for $k\in
	[n]$ and $p\in \ol{U}_i$. We shall write $U_i\po U_{i'}$ if  $\ol{U}_i\cap
	\ol{U}_{i'}=\emptyset$ or if $\ol{U}_i\cap \ol{U}_{i'}\neq \emptyset$ and
	for every $k \in [n]$, either $I_k(U_i)\subs I_k(U_{i'})$ or $I_k(U_i)\cap
	I_k(U_{i'})=\emptyset$; it is not required that the same option hold for
	every $k$. (This is generally not a transitive relation.) The complement of
	a set $W$ in $K$ will be denoted by $W^c$. The rough idea behind the proof
	is to repeatedly apply the following procedure: If $\ol{U}_{i_1}\cap\dots
	\cap \ol{U}_{i_\mu}$ is nonempty, then we excise it from each of the open
	sets $U_{i_j}$ and add a new open set $V$ to the cover which contains the
	intersection but is still sufficiently small. If $I_k(V)$ is taken to be a
	component of $\bcup_j I_k(U_{i_j})$ for each $k$, then  $U_i\po V$ for every
	$i=i_1,\dots,i_\mu$.
	
	 Let $m$ be the largest integer for which there exist distinct
	 $i_1,\dots,i_{m}\in [l]$ with $\ol{U}_{i_1}\cap \dots\cap
	 \ol{U}_{i_{m}}\neq \emptyset$. Note that there are only finitely many such
	 $m$-tuples. Choose one of them, say $T_m=\se{i_1,\dots,i_{m}}$, and let
	 $V_{T_m}$ be an open set such that \begin{equation*}
		 \bcap_{\mu=1}^{m} \ol{U}_{i_\mu}\subs V_{T_m}\subs \ol{V}_{T_m}\subs
		 \bcap_{i\neq i_j}\ol{U}_i^c.  \end{equation*} Such a set exists because
	 $\ol{U}_{i_1}\cap \dots\cap \ol{U}_{i_{m}}\subs \ol{U}_i^c$ for every
	 $i\neq i_j$, by maximality of $m$.  Set \begin{equation*}
		 \text{(new)}\,U_{i_j}:=\text{(old)}\,U_{i_j}\ssm \bcap_{\mu=1}^{m}
		 \ol{U}_{i_\mu}\quad (j\in [m]).  \end{equation*} For each $k\in [n]$,
	 take $I_k\big(U_{i_j}\big)$ to be the same intervals as for the original
	 sets $U_{i_j}$ and $I_k\big(V_{T_m}\big)$  to be any connected component of
	 $\bcup_{j=1}^{m}I_k\big(U_{i_j}\big)$. Fix $k\in [n]$; if $p\in
	 \bcap_{\mu=1}^{m} \ol{U}_{i_\mu}$, then every interval $I_k(U_{i_j})$
	 $(j\in [m]$) satisfies the conditions stated in \lref{L:continuousJs}\?(b).
	 Therefore, by \rref{L:lowering}\?(d), if $V_{T_m}$ is sufficiently small,
	 then $I_k\big(V_{T_m}\big)$ satisfies these conditions for all $p\in
	 \ol{V}_{T_m}$. Further, by construction $I_k\big(V_{T_m}\big)$ either
	 contains or is disjoint from $I_k(U_{i_j})$ for each $j,\,k$. Thus:
 \begin{enumerate} \item [\sbu] The open sets $U_i$ $(i\in [l])$ and $V_{T_m}$
		 cover $K$.  \item [\sbu] If $\ol{U}_i\cap \ol{V}_{T_m}\neq \emptyset$
			 then ${i}={i_j}$  for some $j$. Hence, $U_i\po V_{T_m}$ for every
			 $i\in [l]$.  \item [\sbu]  No new $m$-fold intersection has been
				 created among the $\ol{U}_i$.  \end{enumerate}
	
	 If there still exists an $m$-tuple $T'_m=\se{i_1',\dots,i_{m}'}$ such that
	 $\ol{U}_{i_1'}\cap \dots\cap \ol{U}_{i_{m}'}\neq \emptyset$, the
	 construction is repeated to excise the latter from each $U_{i_j'}$ and
	 create an open set $V_{T'_m}$ such that \begin{equation*}
		 \bcap_{\mu=1}^{m} \ol{U}_{i'_\mu}\subs V_{T'_m}\subs \ol{V}_{T'_m}\subs
		 \bcap_{i\neq i_j'}\ol{U}_{i'}^c\cap \big(\ol{V}_{T_m}\big)^c.
	 \end{equation*} Such a set exists because there are no $(m+1)$-fold
	 intersections among the $\ol{U}_i$ and  $i'_j\nin \se{i_1,\dots,i_m}$ for
	 at least one $j\in [m]$. 
	 By definition,
	 $\ol{V}_{T_m}\cap \ol{V}_{T'_m}=\emptyset$, and $\ol{V}_{T'_m}\cap
	 \ol{U}_i=\emptyset$ unless $i=i_j'$ for some $j\in [m]$. Again, let
	 $I_k\big(V_{T'_m}\big)$ be a connected component of
	 $\bcup_{j=1}^{m}I_k\big(U_{i_j'}\big)$ for each $k$, so that $U_i\po
	 V_{T'_m}$ for all $i\in [l]$. If $V_{T'_m}$ is sufficiently small, then all
	 of the conditions in (b) are satisfied by the $I_k\big(V_{T'_m}\big)$
	 whenever $p\in \ol{V}_{T'_m}$. After finitely many iterations, there will
	 be no more $m$-tuples of indices in $[l]$ for which the corresponding
	 $\ol{U}_{i}$ intersect. Notice that by construction: \begin{enumerate}
		 \item [\sbu] $V_{T_m}\po V_{T'_m}$ for any $T_m\neq T'_m$, since their
		 closures are disjoint.  \item [\sbu] $U_i\po V_{T_m}$ for any $T_m$ and
		 $i\in [l]$.  \item [\sbu] Every $m$-fold intersection among  the
			 $\ol{U}_i$ is empty.  \end{enumerate}
	
	Now the same procedure is carried out for $(m-1)$-fold intersections among
	the $\ol{U}_i$. Assume that $\ol{V}_{T_{m-1}^{\nu}}$ has been defined for
	all $\nu=1,\dots,\nu_0-1$, where each $T_{m-1}^{\nu}\subs [l]$ has
	cardinality $m-1$, with $\ol{U}_i\cap \ol{V}_{T_{m-1}^{\nu}}\neq \emptyset$
	only if $i\in T_{m-1}^{\nu}$. If $T_{m-1}^{\nu_0}=\se{i_1,\dots,i_{m-1}}$ is
	such that $\ol{U}_{i_1}\cap \dots\cap \ol{U}_{i_{m-1}}\neq \emptyset$,
	choose a sufficiently small open set $V_{T_{m-1}^{\nu_0}}$ satisfying
	\begin{equation*}
	\bcap_{\mu=1}^{m-1} \ol{U}_{i_\mu}\subs
		V_{T_{m-1}^{\nu_0}}\subs \ol{V}_{T_{m-1}^{\nu_0}}\subs \bcap_{i\neq
		i_j}\ol{U}_{i}^c\cap
		\bcap_{\nu=1}^{\nu_0-1}\big(\ol{V}_{T_{m-1}^{\nu}}\big)^c,
	\end{equation*} excise $\bcap_{\mu=1}^{m-1} \ol{U}_{i_\mu}$ from each
	$U_{i_j}$  and let $I_k\big(V_{T_{m-1}^{\nu_0}}\big)$ be a connected
	component of $\bcup_{j=1}^{m-1}I_k\big(U_{i_j}\big)$. The choice of
	$V_{T_{m-1}^{\nu_0}}$ is possible because by hypothesis there are no
	$m$-fold intersections among the $\ol{U}_i$ and for each $\nu\leq \nu_0-1$,
	we have $i_j\nin T^\nu_{m-1}$ for at least one $j\in [m-1]$. At the end of
	this step we have sets $U_i$ and $V_{T}$ (with $\abs{T}=m-1$ or $m$)
	covering $K$ such that: \begin{enumerate} \item [\sbu] $V_{T}\po V_{T'}$
			whenever $\abs{T} \leq \abs{T'}$.  \item [\sbu] $U_i\po V_{T}$ for
			every $i\in [l]$ and every set $V_{T}$.  \item [\sbu] There exists
				no nonempty $(m-1)$-fold intersection among  the $\ol{U}_i$.
		\end{enumerate} 

Continuing this down to twofold intersections, we obtain open sets $V_{T}$ and
$U_i$ with $\abs{T}=2$ and $\ol{U}_i\cap \ol{U}_{i'}=\emptyset$ whenever $i\neq
i'$. Finally, for each $i\in [l]$, set $V_{\se{i}}=U_i$. Then the sets $V_T$
form an open cover of $K$ and $V_T \po V_{T'}$ whenever $\abs{T}\leq \abs{T'}$.
To establish (a) we simply relabel the $V_T$ in order of nondecreasing
$\abs{T}$, for $\abs{T}=1,\dots,m$.

By \rref{L:lowering}\?(d), the original intervals $I_{i,k}$ given by
\lref{L:continuousJs} can always be enlarged so as to satisfy the condition on
the enpoints stated in (b). Furthermore, if some intervals $I_1,\dots,I_m$ satisfy
(b), then so does any component of $\bcup_{j=1}^mI_j$. Hence the proof of (a)
preserves this property.
\end{proof}

\begin{lem}\label{L:multiquasi} Let $\sig_1\prec\dots\prec\sig_m$ be sign
	strings and $\ga\in \sr M(Q)$ be $(\vphi,\eps_j)$-quasicritical of type
	$\sig_j$ for each $j\in [m]$. Then $\eps_{j+1}>2\eps_j$ for each $j\in
	[m-1]$.  \end{lem} \begin{proof} Clearly, the lemma can be deduced from the
	special case where $m=2$. Let $n=\abs{\sig_2}$, $l=\abs{\sig_1}$ and let
	$J_1<\dots<J_n$, $J_1'<\dots<J_l'$ be intervals as in \dref{D:quasicritical}
	for $(\sig,\eps)=(\sig_2,\eps_2)$ and $(\sig_1,\eps_1)$, respectively. For
	each $k\in [n]$, let $I_k\subs J_k$ be a subinterval where
	$\abs{\theta_\ga-\vphi}>\frac{\pi}{2}-\eps_2$ throughout, as guaranteed by
	(iii). By (ii), if $t\nin \bcup_{i\in [l]}J'_i$, then
	$\abs{\theta_\ga(t)-\vphi}< \frac{\pi}{2}-2\eps_1$. Therefore, if
	$\eps_2\leq 2\eps_1$, then each $I_k$ must be contained in a $J_i'$.
	Further, because $n>l$, there must exist $k\in [n-1]$, $i\in [l]$ such that
	$I_{k}\cup I_{k+1}\subs J_i'$. From $\sig_2(k)=-\sig_2(k+1)$ it follows that
	\begin{equation*}
	\theta_\ga(J_i')\cap
		\big(\vphi_+-\eps_2,\vphi_++\eps_2\big)\neq \emptyset \quad
		\text{and}\quad \theta_\ga(J_i')\cap
		\big(\vphi_--\eps_2,\vphi_-+\eps_2\big)\neq \emptyset.  \end{equation*}
	But this contradicts (i) of \dref{D:quasicritical} (for $\sig=\sig_1$).
	Hence, $\eps_2>2\eps_1$.  \end{proof}

\begin{lem}\label{L:convexity} Let $\sig$ be a sign string, $0<\eps<\eps'$ and
	suppose that $\ga\in \sr M(Q)$ is simultaneously $(\vphi,\eps)$- and
	$(\vphi,\eps')$-quasicritical of type $\sig$. Then $\ga$ is
	$(\vphi,\de)$-quasicritical of type $\sig$ for any $\de\in [\eps,\eps']$.
\end{lem} \begin{proof} Let $n=\abs{\sig}$ and $J_1<\dots<J_n$,
	$J_1'<\dots<J_n'$ be as in \dref{D:quasicritical}, corresponding to
	$\eps,\eps'$, respectively. The inequalities $\eps\leq \de\leq \eps'$ and
	\rref{R:automaticineq} imply that the intervals $J_k'$ still satisfy (i) and
	(ii) of \dref{D:quasicritical} if $\eps'$ is replaced by $\de$. An argument
	similar to the proof of \lref{L:Js}\?(a) shows that if $I_k\subs J_k$ is any
	subinterval where $\abs{\theta_\ga-\vphi}>\frac{\pi}{2}-\eps\geq
	\frac{\pi}{2}-\de$ throughout, then $I_k\subs J_k'$. By (iii), for each
	$k\in [n]$, there exists such an $I_k$ which, additionally, is stretchable.
	Hence the $J_k'$ also satisfy (iii) if $\eps'$ is replaced by $\de$.
\end{proof}

\begin{rem}\label{R:smallerandsmaller} Let $0<\de\leq \eps$, $\ga$ be
	$(\vphi,\eps)$-quasicritical of type $\sig$, and $J_k$ ($k\in [n]$) be
	intervals as in \dref{D:quasicritical} for the pair $(\vphi,\eps)$. Suppose
	that $\theta_\ga([0,1])\subs (\vphi_--\de,\vphi_++\de)$ and  that each
	$J_k$ contains a stretchable subinterval $I_k$ where $\abs{\theta
		_\ga-\vphi_{\sig(k)}}<\de$ throughout. Then the $J_k$ also satisfy
		(i)--(iii) of \dref{D:quasicritical} for the pair $(\vphi,\de)$, hence
		$\ga$ is $(\vphi,\de)$-quasicritical of type $\sig$. \end{rem}

\begin{lem}\label{L:fiberisinterval} Let $\ga\in \sr M(Q)$ be a critical curve
	of type $\sig$. Let \begin{equation*}
	S=\set{\vphi\in
		\R}{\text{there exists\ } \eps>0\text{ for which }\ga\text{ is
		$(\vphi,\eps)$-quasicritical of type $\sig$}}.  \end{equation*} Then $S$
	is an open interval containing $\bar\vphi^\ga$.  \end{lem} \begin{proof} Let
	$\bar{\vphi}=\bar\vphi^\ga$ be as in \eqref{E:average}. By
	\rref{R:opencond}, $S$ is open  and by \rref{L:criticalisquasi},
	$\bar\vphi\in S$. Suppose that $\ga$ is $(\vphi,\eps)$-quasicritical of type
	$\sig$; no generality is lost in assuming that $\bar\vphi\leq \vphi$. Since
	$\ga$ is critical, $\inf_{t\in [0,1]}\theta_\ga(t)=\bar\vphi_-$. Hence, by
	\rref{R:automaticineq}, \begin{equation}\label{E:pos} \eps>\vphi-\bar\vphi.
	\end{equation} Let $\psi\in (\bar\vphi,\vphi)$,  $\de=\eps-(\vphi-\psi)$ and
	let $J_1<\dots<J_n$ be as in \dref{D:quasicritical} for the pair
	$(\vphi,\eps)$. We claim that these intervals also  satisfy (i)--(iii) for
	the pair $(\psi,\de)$.
	
	Notice that $\theta_\ga([0,1])= [\bar\vphi_-,\bar\vphi_+]\subs
	(\psi_--\de,\psi_++\de)$, as a consequence of \eqref{E:pos}. It is also easy to
	check that \begin{equation*}
	\psi_+-2\de>\vphi_+-2\eps\text{\ \
		and\ \ }\psi_-+2\de<\vphi_-+2\eps.  \end{equation*}
	Consequently, the $J_k$ satisfy (i), (ii) of \dref{D:quasicritical} for the
	pair $(\psi,\de)$. 
	
	Let $t_1<\dots<t_n$ be such that $\theta_\ga(t_k)=\bar\vphi_{\sig(k)}$. 
	Using \eqref{E:pos}, one deduces that each $t_k$ must be contained in an interval $J_{k'}$ with
	$\sig(k')=\sig(k)$. Therefore, no two of the $t_k$ can be contained in the
	same $J$, so that $t_k\in J_k$ for all $k\in [n]$. Since $\bar{\vphi}_-<
	\psi_-$, if $\sig(k)=\ty{-}$ then $J_k$ must contain some $t$ such that
	$\theta_\ga(t)=\psi_-$. In particular, by \rref{L:lowering}\?(a), condition
	(iii) of \dref{D:quasicritical} is satisfied by $J_k$ for the pair
	$(\psi,\de)$ whenever $\sig(k)=\ty{-}$. If $\sig(k)=\ty{+}$, let $I\subs
	J_k$ be an interval as in (iii) for the pair $(\vphi,\eps)$. By
	\rref{L:lowering}\?(b), this interval is also stretchable with respect to
	$\psi_+$. Moreover, \begin{equation*}
		\psi_+-\de=\vphi_+-\eps<\theta_{\ga}(t)\leq
		\bar{\vphi}_+<\psi_++\de\text{\ \ for all $t\in I$}; \end{equation*}
	hence $J_k$ also satisfies (iii) for the pair $(\psi,\de)$ in case
	$\sig(k)=\ty{+}$.  \end{proof}

\begin{defn}[$\sr N(Q)$, $\sr V_\ast$]\label{D:Vs} Let $Q=(q,z)\in \C\times \Ss^1$, $z\neq -1$. Let $R(Q)$
	denote the open interval of size $\pi-\abs{\theta_1}$ centered at
	$\frac{\theta_1}{2}$, where $e^{i\theta_1}=z$ and $\abs{\theta_1}<\pi$.  Let
	$\sr U_c,~\sr U_d$ be the open subsets of $\sr M(Q)$ consisting of all all
	condensed (resp.~diffuse) curves. Define \begin{alignat*}{9}
	\sr V_d&:=\sr U_d\times R(Q);\\ 
	\sr V_c&:=\set{(\ga,\vphi)\in \sr M(Q)\times
	R(Q)}{\theta_\ga([0,1])\subs (\vphi_-,\vphi_+)}.
	\end{alignat*} If $\sr M(Q)$ does not
	contain critical curves of type $\sig$, set $\sr V_\sig:=\emptyset$.
	Otherwise, define \begin{equation*}
	\sr
		V_\sig:=\set{(\ga,\vphi)\in \sr M(Q)\times R(Q)}{\text{$\ga$ is
			$(\vphi,\eps)$-quasicritical of type $\sig$ for some $\eps\in
			(0,\tfrac{\pi}{4})$}}.  \end{equation*} The union of $\sr V_c$, $\sr
		V_d$ and all the $\sr V_\sig$ will be denoted by $\sr N(Q)$, and the
		cover of $\sr N(Q)$ by these sets will be denoted by $\fr V$. Note that
		each $\sr V_\ast$ is an open subset of $\sr M(Q)\times \R$, hence so is
		$\sr N(Q)$. For sign strings $\sig_1\prec\dots\prec\sig_m$, the
		intersection $\sr V_{\sig_1}\cap\dots\cap\sr V_{\sig_m}$ will be denoted
		by $\sr V_{(\sig_1,\dots,\sig_m)}$. Similarly,  $\sr
		V_{(c,\sig_1,\dots,\sig_m)}:=\sr V_c\cap \sr V_{(\sig_1,\dots,\sig_m)}$
		and $\sr V_{(d,\sig_1,\dots,\sig_m)}:=\sr V_d\cap \sr
		V_{(\sig_1,\dots,\sig_m)}$.  \end{defn}

\begin{urem}\label{R:center} Observe that
	$R(Q)=\big(\theta_1-\frac{\pi}{2},\frac{\pi}{2}\big)$ if $\theta_1\geq 0$
	and $R(Q)=\big(-\frac{\pi}{2},\theta_1+\frac{\pi}{2}\big)$ if $\theta_1\leq
	0$. In either case, it consists of all $\vphi\in \R$ such that
	$\vphi_-<0\,,\theta_1<\vphi_+$.  \end{urem}

\begin{lem}\label{L:N=M} Let $\pr \colon \sr N(Q)\to \sr M(Q)$ be the
	restriction of the canonical projection $\sr M(Q)\times \R\to \sr M(Q)$. Let $K$ be
	any compact space and $g\colon K\to \sr M(Q)$ a continuous map. Then there
	exists $\te{g}\colon K\to \sr N(Q)$ such that $\pr\circ \te{g}=g$.
\end{lem} \begin{proof}Let $g\colon p\mapsto \ga^p\in \sr M(Q)$ and
	$\bar\vphi^p:=\bar\vphi^{\ga^p}$, as in \eqref{E:average}. Let $\om(p)$
	denote the amplitude of $\ga^p$. 
	Since $\frac{\theta_1}{2}$ always lies in $R(Q)$, if $\ga^p$ is diffuse then
	$\big(\ga^p,\frac{\theta_1}{2}\big)\in \sr V_d$. If $\ga^p$ is
	condensed, then $\bar\vphi^p$ also lies in $R(Q)$ and $(\ga^p,\bar\vphi^p)\in \sr V_c$.
	Finally, if $\ga^p$ is critical, then $\bar\vphi^p\in \ol{R(Q)}$.
	
	Using \lref{L:fiberisinterval} and compactness of $K$, choose $s_0\in (0,1]$
and $\de>0$ so small that: \begin{enumerate} \item [\sbu] $\ga^p$ is
			$(\psi,\eps)$-quasicritical of type $\sig$ (for some $\sig$ and
			$\eps>0$, whose values are irrelevant) for
			$\psi=(1-s_0)\bar\vphi^p+s_0\frac{\theta_1}{2}$ whenever
		$s\in [0,s_0]$ and $\abs{\om(p)-\pi}\leq 2\de$.
	\end{enumerate}
	Further reducing $ \de>0 $ if necessary, it can be achieved that 
\begin{enumerate}
	\item [\sbu]  $(\ga^p,\psi)\in \sr V_d$ for
		$\psi=(1-s)\bar\vphi^p+s\frac{\theta_1}{2}$, whenever $s\in [s_0,1]$ and
		$\pi \leq \om(p)\leq \pi+2\de$.
\end{enumerate}
Let $s\colon
	\R\to [0,1]$ be an increasing continuous function satisfying:
	\begin{equation*}
	s(u)=\begin{cases} 
		0 & \text{ if\ \ $u\leq \pi-2\de$;} \\
		s_0 & \text{ if\ \ $\abs{u-\pi}\leq \de$;} \\ 
		1 & \text{ if\ \ $u\geq \pi+2\de$;}  
	\end{cases}
	\end{equation*} and set
		$\vphi^p:=[1-s(\om(p))]\bar\vphi^p+s(\om(p))\frac{\theta_1}{2}$. Then
		$\te{g}(p)=(\ga^p,\vphi^p)\in \sr N(Q)$ for all $p\in K$.  \end{proof}

\begin{cor}\label{C:contractible} If $\sr N(Q)$ is contractible, then so is $\sr
	M(Q)$.  \end{cor} \begin{proof} Indeed, $\pr\colon \sr N(Q)\to \sr M(Q)$
	induces surjections on homotopy groups and a weakly contractible Hilbert
	manifold is contractible.  \end{proof}

\begin{lem}\label{L:Moore} Let $p\colon X\to Y$ be a continuous map between
	topological spaces. Suppose that $X\iso \Ss^n$ for some $n\in \N$ and that
	given any compact space $K$ and any map $g\colon K\to Y$, there exists
	$\te{g}\colon K\to X$ such that $p\te{g}=g$. Then $Y$ is either weakly
	contractible or a homology $n$-sphere.  \end{lem} \begin{proof} The
	hypothesis immediately implies that $Y$ is a Moore space $M(\Z/(k),n)$ for
	some $k\in \N$. Let $K$ be a CW complex obtained by attaching an
	$(n+1)$-cell to $\Ss^n$ via a map of degree $k$. Let $g\colon K\to Y$ be
	such that $g_\ast\colon H_\ast(K)\to H_{\ast}(Y)$ is an isomorphism. By
	hypothesis, $g$ factors through $X$. Since $H_n(X)\iso \Z$, this implies
	that either $k=0$ or $k=1$.  \end{proof}

The homotopy type of $\sr M(Q)$ will be determined as follows. If $\sr M(Q)$
contains no critical curves, then $\sr M(Q)\home \E$ or $\E\times \Ss^0$
depending on whether $\sr U_c=\emptyset$ or not; see Theorem 6.1 in
\cite{SalZueh1}. Otherwise, let $n$ denote the greatest length $\abs{\sig}$
among those sign strings $\sig$ for which $\sr V_\sig\neq \emptyset$. In
\S\ref{S:combinatorics} the cover $\fr V$ will be shown  to have the same
combinatorics as that in \eqref{E:half-spaces}, and in \S\ref{S:topology} it
will be shown that $\fr V$ is a good cover of $\sr N(Q)$. Then \lref{L:Moore},
together with an easy topological lemma, will imply that either 
$\sr M(Q)$ is contractible or it has the homotopy type of
$\Ss^{n-1}$. Finally, if $\sr N(Q)\iso \Ss^{n-1}$, then $\sr M(Q)\iso \Ss^{n-1}$
as well, because in this case a non-nullhomotopic map $\Ss^{n-1}\to \sr M(Q)$
can be constructed explicitly; this is done in \S \ref{S:generator}.

\begin{lem}\label{L:continuouseps} Let $\sig_1\prec\dots\prec\sig_m$ be sign
	strings and $f\colon K\to \sr V_{(\sig_1,\dots,\sig_m)}$, $p\mapsto
	(\ga^p,\vphi^p)$, be a continuous map. Then there exist continuous
	$\eps_j\colon K\to \R^+$, $p\mapsto \eps_j^p$, such that for each $p\in K$,
	$\ga^p$ is $(\vphi^p,\eps_j^p)$-quasicritical of type $\sig_j$. Moreover,
	$\eps_{j+1}>2\eps_j$ for each $j\in [m-1]$ throughout $K$.  \end{lem}
\begin{proof} By \rref{R:opencond}, such functions can be defined on a
	neighborhood of every $p\in K$. Globally defined $\eps_j\colon K\to \R^+$
	($j\in [m]$) are obtained through convex combinations using partitions of unity; this works in view 
	of \lref{L:convexity}. The last assertion is just a restatement of 
	\lref{L:multiquasi}.  \end{proof}

\begin{defn}\label{D:h} Let $(\ga,\vphi)\in \sr V_\sig$, $n=\abs{\sig}$, and let
	$J_k$ ($k\in [n]$) be intervals satisfying the conditions in
	\dref{D:quasicritical} for some $\eps\in (0,\frac{\pi}{4})$. Define $h\colon
	\sr V_\sig\to \R^n$ by: \begin{equation}\label{E:h}
		h_k(\ga,\vphi)=\begin{cases} \sup_{t\in J_k}\se{\theta_\ga(t)-\vphi_+} &
			\text{ if\ \  $\sig(k)=\ty{+}$;} \\ \inf_{t\in
			J_k}\se{\theta_\ga(t)-\vphi_-} & \text{ if\ \ $\sig(k)=\ty{-}$;} \\
		\end{cases}\quad (k\in [n]).  \end{equation} \end{defn}

\begin{rem}\label{R:welldefined} Even though $\eps$ and the $J_k$ are not
	uniquely determined,  \lref{L:Js}\?(a) implies that $h$ is well-defined.
	Furthermore, it is continuous. Indeed, by \rref{R:opencond}, for
	$(\eta,\psi)$ sufficiently close to $(\ga,\vphi)$, we may choose the same
	intervals $J_k$ in \dref{D:quasicritical} for $(\eta,\psi)$ as for
	$(\ga,\vphi)$; but for fixed $J_k\subs [0,1]$, it is clear that \eqref{E:h}
	depends continuously upon $(\ga,\vphi)$. 
\end{rem}

\begin{unotn}
Given intervals $I_1,\dots,I_n$, let $I_1\ast \dots\ast I_n$ denote the smallest
closed interval containing $I_1\cup \dots\cup I_n$.
\end{unotn}

\begin{lem}\label{L:convenient} Let $\sig_1\prec \sig_2$ be sign strings and
	suppose that $\ga\in \sr M(Q)$ is $(\vphi,\eps_j)$-quasicritical of type
	$\sig_j$, $j=1,2$. Let $\abs{\sig_1}=l$, $\abs{\sig_2}=n$ and
	$J_1<\dots<J_n$ be intervals as in \dref{D:quasicritical} for the pair
	$(\sig_2,\eps_2)$. Then there exist intervals $J_1'<\dots<J_l'$ satisfying
	\dref{D:quasicritical} for $(\sig_1,\eps_1)$ such that: \begin{enumerate}
		\item [(a)] Each $J_i'$ has the form $J_{k}\ast J_{k'}$, for some $k\leq
		k'\in [n]$ depending on $i\in [l]$.  \item [(b)] If $k\in [n]$ is such
			that $\abs{h_k(\ga,\vphi)}\leq 2\eps_1$, then $J_k\subs J_i'$ for
			some $i\in [l]$.  \item [(c)]  For each $i\in [l]$, there exists
				$k\in [n]$ such that $\abs{h_k(\ga,\vphi)}<\eps_1$ and $J_k\subs
				J_i'$.  \end{enumerate} \end{lem} 
\begin{proof} Let
		$k_1<\dots<k_m$ be all the indices $k\in [n]$ such that
		$\abs{h_k(\ga,\vphi)}\leq 2\eps_1$. Define $\tau\colon [m]\to \se{\pm}$
		by $\tau(j)=\sig_2(k_j)$. For each $j\in [m]$, choose $t_{j}\in J_{k_j}$
		such that $\theta_\ga(t_{j})=\vphi_{\sig_2(k_j)}+h_{k_j}(\ga,\vphi)$.
		Let $J''_1<\dots<J''_{l}$ be any intervals as in \dref{D:quasicritical}
		for the pair $(\sig_1,\eps_1)$. Then: \begin{enumerate} \item [\sbu]
					Each $t_{j}$ must be contained in some $J''_i$ with
					$\sig_2(k_j)=\sig_1(i)$. This follows immediately from
					condition (ii) of \dref{D:quasicritical} for the pair
					$(\sig_1,\eps_1)$.  \item [\sbu] For each $i\in [l]$,
						$J_i''$ must contain one of the $t_{j}$. Indeed, by
						(iii) of \dref{D:quasicritical}, for any $i$ there
						exists $s_i\in J_i''$ such that
						$\vert\theta_\ga(s_i)-\vphi_{\sig_1(i)}\vert<\eps_1$. By
						\lref{L:multiquasi}, $2\eps_1<\eps_2$, hence $s_i\in
						J_k$ for some $k$, which forces
						$\abs{h_k(\ga,\vphi)}<\eps_1$. Therefore $k=k_j$ for
						some $j$, and it follows that $t_{j}$ must be
						contained in $J_i''$.  \end{enumerate}  Let $\vrho$ be
				the reduced string of $\tau$. The first assertion implies that
				$\vrho$ is a substring of $\sig_1$, while the second one implies
				that it cannot be a proper substring. Consequently
				$\vrho=\sig_1$.
	
	Thus, there exists a decomposition of $\se{k_1,\dots,k_m}$ as the disjoint
	union of nonempty sets $S_1<\dots<S_l$ with $\sig_2(k)=\sig_1(i)$ whenever
	$k\in S_i$. Set $J_i'=\Ast_{k\in S_i}J_k$. Then $J_1'<\dots<J_l'$, and parts
	(a) and (b) hold by construction. Moreover,
	$\abs{\theta_\ga(t)-\vphi}<\frac{\pi}{2}-2\eps_1$ if $t\nin
	\Int\big(\bcup_iJ_i'\big)$: If $t\nin \bcup_kJ_k$, then this is obvious from
	(ii) of \dref{D:quasicritical}, since $\eps_2>2\eps_1$ by
	\lref{L:multiquasi}; if $t\in J_k$ for some $k$, then necessarily
	$\abs{h_k(\ga,\vphi)}>2\eps_1$, hence again the inequality holds. This
	proves that condition (ii) of \dref{D:quasicritical} is satisfied by the
	$J'_i$. Condition (i) is also easily verified using that $\eps_2>2\eps_1$.
	
	Since $\ga$ is $(\vphi,\eps_1)$-quasicritical, there exist intervals
	$I_1<\dots<I_l$ such that $I_i$ is stretchable and
	\begin{equation}\label{E:vert}
		\abs{\theta_\ga(t)-\vphi_{\sig_1(i)}}<\eps_1\quad\text{for all $t\in
		I_i$ and $i\in [l]$.} \end{equation} The inequality implies that each of
	these intervals must be contained in some $J'$, and no two subsequent
	intervals may be contained in the same $J'$. Hence $I_i\subs J_i'$ for each
	$i\in [l]$. This proves that condition (iii) of \dref{D:quasicritical} is
	satisfied by the $J'$. Since $\eps_1<\eps_2$, \eqref{E:vert} also implies
	that each $I_i$ must be contained in some $J_{k}$ with
	$\abs{h_k(\ga,\vphi)}<\eps_1$, so that $J_k\subs J_i'$ by the definition of
	the $J'$. This proves part (c).  \end{proof}


\section{Incidence data of the cover of $\sr N(Q)$}\label{S:combinatorics}

\subsection*{Good covers of Hilbert manifolds} An open cover $\fr
U=(U_{\nu})_{\nu\in I}$ of a space is \tdef{good} if for any finite $J\subs I$,
the intersection $\bigcap_{\nu\in J}U_\al$ is either empty or contractible. Let
$\fr V=(V_{\nu})_{\nu \in I}$ be a good cover of another space, indexed by the
same set $I$. Then $\fr U$ and $\fr V$ will be called  (combinatorially)
\tdef{equivalent} when for any finite $J\subs I$, $\bigcap_{\nu\in
J}U_\nu=\emptyset$ if and only if $\bigcap_{\nu\in J}V_\nu=\emptyset$. Recall
that the \tdef{nerve} $K_{\fr U}$ of an open cover $\fr U$ of a space is a
simplicial complex whose $n$-simplices correspond bijectively to the nonempty
$(n+1)$-fold intersections of distinct elements of $\fr U$, for each $n\in \N$.

\begin{lem}\label{L:nerve} If $\fr U$ is a good cover of a paracompact space
	$X$, then $X$ is homotopy equivalent to the nerve $K_{\fr U}$.  \end{lem}

\begin{proof} See \cite{Hatcher}, Corollary 4G.3 or \cite{Weil}, p.~141.
\end{proof}

Because the spaces $\sr L_{\kappa_1}^{\kappa_2}(P,Q)$ are closed submanifolds of
the separable Hilbert space $\E$ (see Definition 1.6 of \cite{SalZueh1}), they
are second-countable and metrizable. It follows that they are also paracompact. It will be tacitly assumed below that
all Hilbert manifolds are separable and metrizable. 

\begin{cor}\label{C:goodcovers} If two Hilbert manifolds $\sr M$ and $\sr N$
	admit equivalent good covers, then $\sr M\home \sr N$.  \end{cor}
\begin{proof} Let $\fr U$ and $\fr V$ be equivalent good covers of $\sr M$ and
	$\sr N$, respectively. Let $K$ be the nerve of $\fr U$, which is
	homeomorphic to the nerve of $\fr V$ by hypothesis. By \lref{L:nerve}, there
	exist homotopy equivalences $\sr M\to K$ and $K\to \sr N$. The corollary
	thus follows from the fact that a homotopy equivalence between two Hilbert
	manifolds is homotopic to a homeomorphism, see \cite{SalZueh1}, Lemma
	1.7\?(b).  \end{proof}

\begin{cor}\label{C:goodcovers1} If a Hilbert manifold $\sr M$ and a
	finite-dimensional manifold $N$ admit equivalent good covers, then $\sr
	M\home \E \times N$.\qed \end{cor}

\subsection*{Incidence data of the cover of $\sr N(Q)$} The purpose of this
subsection is to determine which of the open sets $\sr V_\ast\subs \sr N(Q)$
described in \dref{D:Vs} intersect each other.  \begin{lem}\label{L:opposite}
	Suppose that $\ga\in \sr M(Q)$ is simultaneously
	$(\vphi,\eps)$-quasicritical of type $\sig$ and
	$(\vphi,\eps')$-quasicritical of type $\sig'$, for some $\vphi\in \R$,
	$\eps,\eps'\in (0,\frac{\pi}{4})$ and sign strings $\sig,\sig'$. Then
	$\sig'\neq -\sig$.  \end{lem}

\begin{proof} No generality is lost in assuming that $\eps\leq \eps'$. Let
	$n=\abs{\sig}$, $l=\abs{\sig'}$ and $J_1<\dots<J_{n}$, $J'_1<\dots<J'_{l}$
	be intervals as in \dref{D:quasicritical}, for the pairs $(\sig,\eps)$ and
	$(\sig',\eps')$, respectively. For each $k\in [n]$, choose an interval
	$I_k\subs J_k$ such that \begin{equation*}
		\theta_\ga(I_k)\subs (\vphi_{\sig(k)}-\eps,\vphi_{\sig(k)}+\eps).
	\end{equation*} Then for each $k\in [n]$, $I_k$ must be contained in some
	$J'_i$ with $\sig(k)=\sig'(i)$. In particular, $I_k$ and $I_{k+1}$ are not
	contained in the same $J'$ for any $k$. Therefore, either $l>n$ or $l=n$ and
	$\sig'=\sig$.  \end{proof}

\begin{lem}\label{L:substrings} Let $\sig_k$ $(2\leq k\leq n)$ be sign strings
	satisfying $\abs{\sig_k}=k$. Then there exist intervals $R_2\subs\dots \subs
	R_n=[n]$, $\abs{R_k}=k$, such that for each $k=2,\dots,n$, if
	$R_k=\se{r_1<\dots<r_k}$, then $\sig_n(r_i)=\sig_k(i)$ for all $i\in [k]$.
	\qed \end{lem} In words, we can find nested copies of each $\sig_k$ inside
of $\sig_n$ by an appropriate choice of the $R_k$. The proof is an easy
induction which will be left to the reader.

\begin{lem}\label{L:flattening} Let $\ka_1\in (0,1)$. Suppose that $\al\in \sr
	L_{-1}^{+1}(P,Q)$ is condensed, $\ta_\al(0)=\ta_\al(1)$ and
	$\ka_\al([0,1])\subs [-\ka_1+\ka_1]$, but $\al$ is not a line segment. Then
	for all sufficiently small $\eps>0$,  there exists a homotopy $s\mapsto
	\al_s\in \sr L_{-1}^{+1}(P,Q)$ $(s\in [0,1])$ with $\al_1=\al$ and
	$\om(\al_1)-\om(\al_0)=\eps$.\footnote{Recall that
	$\om(\ga)=\sup\theta_\ga-\inf\theta_\ga$ denotes the amplitude of $\ga$.}
\end{lem} \begin{proof} Let $\ka_0\in (\ka_1,1)$ and $H$ be as in Proposition
	3.4 of \cite{SalZueh1}. Then $u\mapsto \al_u=H(u,\al)$ $(u\in [0,1])$, the
flattening of $\al=\al_1$ with curvature $\ka_0$, is a deformation within $\sr
L_{-1}^{+1}(P,Q)$ such that $\om(\al_u)$ is an increasing function of $u$.
Moreover, $\de=\om(\al_1)-\om(\al_0)>0$ by Lemma 3.16 of \cite{SalZueh1} and the
hypotheses on $\al$. Hence, for any $\eps\in (0,\de]$,  there exists $u_0\in
[0,1]$ such that $\om(\al_1)-\om(\al_{u_0})=\eps$.  \end{proof}

\begin{lem}\label{L:multiple} Let $\sig_k$ $(2\leq k\leq n)$ be sign strings
	satisfying $\abs{\sig_k}=k$. Suppose that $\sr M(Q)$ contains critical
	curves of type $\sig_n$. Then $\sr V_{(c,\sig_2,\dots,\sig_n)}$ and $\sr
	V_{(d,\sig_2,\dots,\sig_n)}$ are nonempty.  \end{lem}

\begin{proof} Write $Q=(q,z)\in \C\times \Ss^1$. By Proposition 5.3 of
	\cite{SalZueh1}, the region  \begin{equation*}
		R_{\sig_n}=\set{p\in \C}{\sr M(P) \text{ contains critical curves of
		type $\sig_n$},~P=(p,z)} \end{equation*} is open in $\C$. Hence, there
	exists $\ka_1\in (0,1)$ such that $\ka_1 q\in R_{\sig_n}$. Let $\te
	Q=(\ka_1q,z)$. If $\te\eta\in \sr M(\te Q)$ is a critical curve of type
	$\sig_n$, then the dilated curve $\eta=\frac{1}{\ka_1}\te \eta$ is a
	critical curve of type $\sig_n$ in $\sr M(Q)$ whose curvature takes values
	in $[-\ka_1,+\ka_1]$. A quasicritical curve of the required type can be
	obtained by modifying $\eta$ in neighborhoods of the points $\eta(t)$ where
	$\theta_\eta(t)=\bar\vphi^\eta\pm\frac{\pi}{2}$.
	
	By Corollary 5.7 of \cite{SalZueh1}, the set of all $\vphi\in \R$ such that
	$\sr M(\te{Q})$ contains a critical curve $\te{\eta}$ with
	$\bar\vphi^{\te{\eta}}=\vphi$ is an open interval. Hence it may be assumed
	that $0,\theta_1\in (\bar\vphi^\eta_-,\bar\vphi^\eta_+)$, so that
	\begin{equation}\label{E:offboundary}
		\mu=\min\big\{\vert{\bar\vphi^\eta_{\pm}}\vert,~\vert{\bar\vphi^\eta_\pm-\theta_1}\vert,\tfrac{\pi}{4}\big\}>0,
	\end{equation} where $\theta_1=\theta_\eta(1)$ is the unique number in
	$(-\pi,\pi)$ such that $e^{i\theta_1}=z$. Since \begin{equation*}
		C=\set{t\in [0,1]}{\theta_\eta(t)=\bar\vphi^\eta_{\pm}} \end{equation*}
	is compact, it intersects only finitely many components $V_1<\dots<V_l$ of
	\begin{equation*}
	V=\set{t\in
		[0,1]}{\abs{\theta_\eta(t)-\vphi_\eta}> \tfrac{\pi}{2}-\mu}.
	\end{equation*} Observe that $V_i$ is an open subinterval of $(0,1)$ for
	each $i\in [l]$, and either the maximum or the minimum of
	$\theta_\eta|_{\ol{V}_i}$ is attained at both endpoints. Let
	\begin{equation*}
		\la=\frac{\pi}{2}-\sup\set{\abs{\theta_\eta(t)-\bar\vphi^\eta}}{t\nin
			{\textstyle\bcup_i V_i}}>0.  \end{equation*} By grafting $\eta$ at
		points of $C$ if necessary (see Definition 4.13 and Figure~9 of
		\cite{SalZueh1}), it may be assumed that for each $i$,
		$\eta|_{\ol{V}_i}$ contains a line segment of some large length $L$
		where $\theta_\eta=\bar\vphi^\eta_{\pm}$.
	
	Let $\al_i=\eta|_{\ol{V}_i}$. Then each $\al_i$ satisfies the hypothesis of
	\lref{L:flattening}. Hence there exists $\de$, $0<2\de<\min\se{\la,\mu}$,
	such that for each $i\in [l]$, if $\eps_i\leq \de$ then $\al_i$ can be
	deformed (keeping initial and final frames fixed) to a curve $\be_i$ such
	that $\om(\al_i)-\om(\be_i)=\eps_i$. We claim that an appropriate choice of
	the $\eps_i$ yields a curve of the required type.
	
	Since $\eta$ is a critical curve of type $\sig_n$, there is  a partition of
	$[l]$ into sets $A_1<\dots<A_n$ such that for every $k \in[n]$ and $i\in
	A_k$, there exists $t\in V_i$ for which
	$\theta_\eta(t)=\bar\vphi^\eta_{\sig_n(k)}$. In view of \lref{L:substrings},
	no generality is lost in assuming that $\sig_k=\sig_n|_{[k]}$ for each
	$k=2,\dots,n$. Set $\eps_i=0$ if $i\in A_1\cup A_2$ and $\eps_i=8^{k-n}\de$
	if $i\in A_k$ for $k>2$. Let $\ga$ be the curve which results by deforming
	each $\al_i$ to $\be_i$, as described above. Notice that
	$\bar\vphi^\ga=\bar\vphi^\eta$. Furthermore: \begin{enumerate} \item [(a)]
				$\ga$ is not condensed, because for every $i_1\in A_{1}$,
				$i_2\in A_{2}$, there exist $t_1\in V_{i_1}$, $t_2\in V_{i_2}$
				such that $\theta_\ga(t_i)=\bar\vphi^\ga_{\sig_2(i)}$ ($i=1,2$).
			\item [(b)]  $\ga$ is not diffuse, since $\eta$ is not diffuse and
			each $\be_i$ was obtained from $\al_i$ by a deformation which
		decreases amplitude.  \item [(c)] $\ga$ is
			$\big(\bar\vphi^\ga\?,\?8^{k-n}2\de\big)$-quasicritical of type
			$\sig_k$ for each $k=2,\dots,n$ by construction. Indeed, setting
			$J_k=\Ast_{i\in A_k}V_i$ (the smallest closed subinterval containing
			these $V_i$) for each $k\in [n]$, $J_1<\dots<J_k$ satisfy (i) and
			(ii) of \dref{D:quasicritical} for $\eps=8^{k-n}2\de$ $(k\geq 2)$
			since  
			\begin{equation*}
				8^{k-n}\de<8^{k-n}2\de<\frac{1}{2} 8^{k+1-n}\de .
			\end{equation*}
			Condition (iii) is a consequence of \rref{L:lowering}\?(c) and our
			assumption that each arc
			$\eta|_{V_i}$ contains a line segment of some large length $L$ where
			$\theta_\eta=\bar\vphi^\eta_{\pm}$.  \end{enumerate}
	
		Therefore, by \eqref{E:offboundary}, $(\ga,\bar\vphi^\ga)\in \sr
		V_{(\sig_1,\dots,\sig_m)}$. By Proposition 5.1 of \cite{SalZueh1}, the
		boundaries of $\sr U_c$ and $\sr U_d$ in $\sr M(Q)$ are both equal to
		the set of all critical curves in $\sr M(Q)$. Therefore, by
		\rref{R:opencond}, a slight perturbation of $\ga$ yields  a curve
		$\te\ga$ such that  $(\te{\ga},\bar\vphi^\ga)\in \sr
		V_{(c,\sig_1,\dots,\sig_m)}$ or $(\te{\ga},\bar\vphi^\ga)\in \sr
		V_{(d,\sig_1,\dots,\sig_m)}$.  \end{proof}

Let us say that $\tau$ is a \tdef{top} sign string for $\sr M(Q)$ if the latter
contains critical curves of type $\tau$, but does not contain critical curves of
type $\tau'$ for any sign string $\tau'$ with $\abs{\tau'}>\abs{\tau}$. Set
$n=\abs{\tau}$. Proposition 5.3 of \cite{SalZueh1} determines whether $\sr M(Q)$
contains critical curves of type $\sig$ in terms of $Q$, for any sign string
$\sig$. Notice in particular that $\sr M(Q)$ always admits a top sign string
$\tau$, except in case it does not contain critical curves at all. 
\begin{prop}\label{P:combinatorics} Let $\tau$ be a top sign string for $\sr
	M(Q)$, $n=\abs{\tau}$, $\fr V$ be the cover of $\sr N(Q)$ described in
	\dref{D:Vs} and $\fr U=\se{U_{\pm k}}_{k\in [n]}$, where $U_{\pm k}\subs
	\Ss^{n-1}$ are as in \eqref{E:half-spaces}.  \begin{enumerate} \item [(a)] If
				$\sr M(Q)$ contains critical curves of type $-\tau$, then
				\eqref{E:combinatorics} defines a combinatorial equivalence
				between $\fr V$ and the cover $\fr U$ of $\Ss^{n-1}$.
			\item [(b)]  If $\sr M(Q)$ does not contain critical curves of type
				$-\tau$, then \eqref{E:combinatorics} defines a combinatorial
				equivalence between $\fr V$ and the cover $\fr U\ssm
				\se{U_{-n}}$ of $ \Ss^{n-1}\ssm \se{(0,0,\dots,-1)} $.
		\end{enumerate} \end{prop} \begin{proof} It is clear that $\sr V_c\cap
		\sr V_d=\emptyset$, and by \cref{L:opposite}, $\sr V_\sig\cap \sr
		V_{-\sig}=\emptyset$ for any sign string $\sig$. On the other hand,
		\lref{L:multiple} implies that an intersection of nonempty sets in $\fr
		V$ is empty only if it involves one such pair. The combinatorics of $\fr
		V$ is thus the same as that of $\fr U$, as asserted.  \end{proof}

\section{Topology of the cover of $\sr N(Q)$}\label{S:topology}

\begin{prop}\label{P:good} Let $\sig_1\prec\dots\prec\sig_m$ be sign strings.
	Then the subspaces $\sr V_{(\sig_1,\dots,\sig_m)},~\sr
	V_{(c,\sig_1,\dots,\sig_m)}$ and $\sr V_{(d,\sig_1,\dots,\sig_m)}$ of $\sr
	N(Q)$ are either empty or contractible.  \end{prop} Let $\sr V$ denote any
of these subspaces. Since $\sr V$ is a Hilbert manifold, it suffices to prove
that it is either empty or weakly contractible. Given a family
$(\ga^p,\vphi^p)\in \sr V$, for $p$ ranging over a compact space, the idea is to
stretch each $\ga^p$ in the direction of $\pm ie^{i\vphi^p}$ so that it becomes
nearly critical (see Figure \ref{F:surgery}), and then flatten it piecewise to
obtain a concatenation of circles and line segments of a special form (see
Figure \ref{F:pulleys}). The results of \S\ref{S:decomposition} are then used to
conclude that the resulting family is contractible. The proof is quite technical
since the conditions in \dref{D:quasicritical} need to be verified at each step;
it will be split into several lemmas. 

For the sake of convenience, a curve $\ga\in \sr M(Q)$ will be called \tdef{of
the form cl} if it is the concatenation of an arc of circle of amplitude $<\pi$ and
a line segment, where either of these may degenerate to a point and the circle
has radius $\frac{1}{\ka_0}$; the value of $\ka_0$ will be clear from the
context. The analogous abbreviation for a more general word on $\se{c,l}$ will
also be used.

\begin{figure}[ht] \begin{center} \includegraphics[scale=.20]{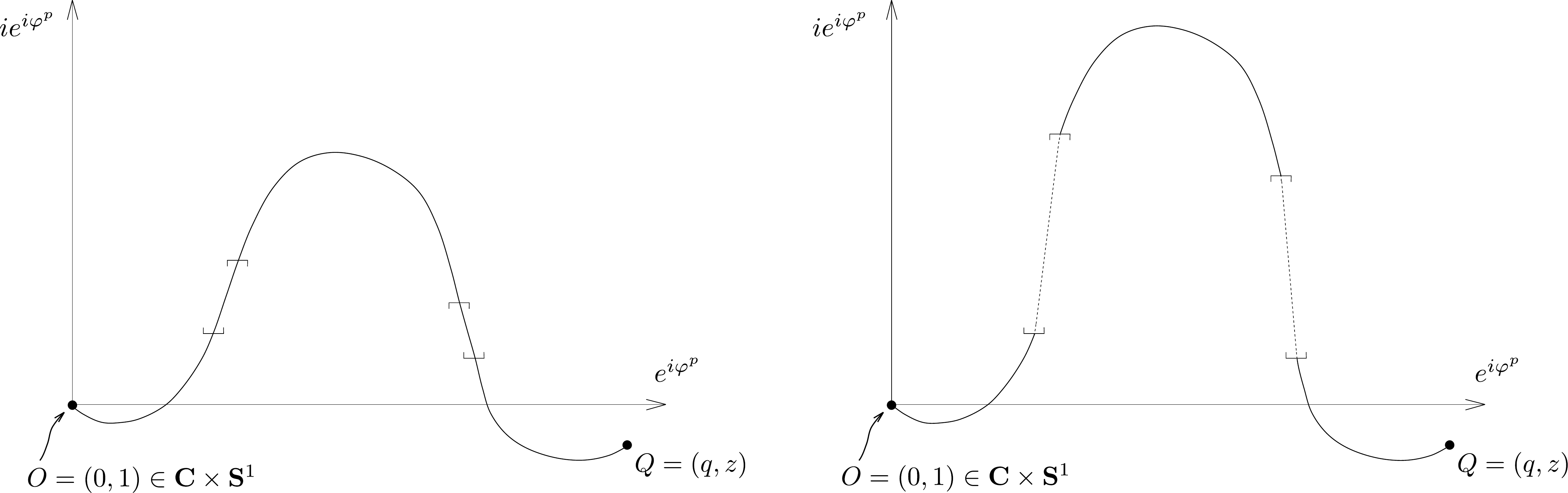}
		\caption{Streching a curve in the direction of $\pm ie^{i\vphi^p}$.}
		\label{F:surgery} \end{center} \end{figure}

\begin{lem}\label{L:stretching} Let $g_0\colon K\to \sr
	V_{(c,\sig_1,\dots,\sig_m)} $, 
	$g_0(p)\mapsto (\ga^p_0,\vphi^p)$,  be a
	continuous map. Then for all
	sufficiently large $C>0$, there exists a homotopy $g_s\colon K\to \sr
	V_{(c,\sig_1,\dots,\sig_m)} $ \tup($s\in [0,m]$\tup),
	$g_s(p)=(\ga_s^p,\vphi^p)$, such that for each $p\in K$ and $j\in [m]$\tup:
\begin{enumerate} 
	\item [(i)]  $\ga^p_m$ is ($\vphi^p,\de_j)$-quasicritical of
			type $\sig_j$, where $ \de_j=\arccot\big(C^{2(m-j)+1}\big)$; 
	\item [(ii)] If $J_{j,1}(p)<\dots<J_{j,\abs{\sig_j}}(p)$ are intervals
		satisfying \dref{D:quasicritical} for the quadruple
		$(\ga^p_m,\vphi^p,\de_j,\sig_j)$, then for each $k\in [\abs{\sig_j}]$,
		there exists an interval $I\subs J_{j,k}(p)$ such that
		$\big\vert{\theta_{\ga_m^p}(t)-\vphi^p_{\sig_j(k)}}\big\vert<\de_j$ for
		all $t\in I$ and $\ga^p_m|_{I}$ is a line segment of length greater than
		$\cot(\de_j)$.  
\end{enumerate}

\end{lem}

\begin{proof} 
	Let $\ga^p_0$ be denoted simply by $\ga^p$. By Corollary 1.11 of
	\cite{SalZueh1}, it may be assumed that each $\ga^p$ is smooth and that all
	of its derivatives depend continuously on $p\in K$. Let $R>0$ be such that
	the image of $\ga^p$ is contained in the open disk of radius $R$ centered at
	the origin for all $p\in K$.  Take $\ka_0\in (\frac{1}{2},1)$ large enough
	so that $\ka_{\ga^p}([0,1])\subs [-\ka_0,+\ka_0]$ for every $p\in K$. For
	each $j\in [m]$, let $\eps_j\colon K\to \R^+$ ($j\in [m]$) be as in
	\lref{L:continuouseps},  $n_j=\abs{\sig_j}$ and let $J_{j,k}(p)$,
	$I_{i,j,k}$ be the intervals corresponding to $\sig_j$ as in
	\lref{L:continuousJs} and \lref{L:disjoint}, for $k\in [n_j]$ and some open
	cover $(U_{i,j})_{i\in [l_j]}$ of $K$. For each $j\in [m]$, let
	$\rho_{i,j}\colon K\to [0,1]$ $(i\in [l_j])$ form a partition of unity
	subordinate to the cover $(U_{i,j})_{i\in [l_j]}$. By choosing a larger
	$\ka_0\in (\frac{1}{2},1)$ if necessary, it may be assumed that
	$\ga^p|_{I_{i,j,k}}$ is $\ka_0$-stretchable with respect to
	$\vphi^p_{\sig_j(k)}$ for each $p\in \ol{U}_{i,j}$ and $k\in [n_j]$. Again
	by compactness, there exists $\vka_1>0$ such that the inequality in
	\lref{L:regularity}\?(f) is satisfied by the function corresponding to the
	stretching of $\ga^p|_{I_{i,j,k}}$ for all $j\in [m]$, $i\in [l_j]$, $k\in
	[n_j]$ and $p\in \ol{U}_{i,j}$. Take $C>0$ to be so large that \begin{equation}\label{E:C}
		C>8\max\big\{R+\pi\ka_0^{-1}\,,\,m\vka_1^{-1}\,,\,\sup_{p\in
		K}\cot(\eps_1^p)\big\}.
		\end{equation}

For the sake of simplicity, it will be assumed that $I_{i,j,k}\cap
I_{i',j,k}=\emptyset$ whenever $i\neq i'$ and ${U}_{i,j}\cap {U}_{i',j}\neq
\emptyset$. The only difference if this did not occur is that it would be
necessary to stretch the restriction of $\ga^p$ to these intervals one $i$ at a
time; see \lref{L:disjoint} and the remark following it. 

Define a homotopy $(s,p)\mapsto \ga^p_s$ ($s\in [0,m]$) inductively as follows
(compare Figure \ref{F:surgery}). For $s\in [j-1,j]$, let $\ga_s^p$ be obtained
from $\ga_{j-1}^p$ by stretching its restriction to $I_{i,j,k}$ linearly with
$s$ in the direction of $\exp\big(i\vphi^p_{\sig_j(k)}\big)$ by:\footnote{When
appearing inside or multiplying an exponential, the  letter $i$ denotes the imaginary unit, not an
index.} \begin{equation}\label{E:amount} \qquad
	l_j\rho_{i,j}(p)\?C^{2(m+1-j)}\text{\ \ for each $i\in [l_j],~k\in [n_j]$}.
\end{equation} Actually, if $n_j$ is odd, then one must introduce different
constants into the above formula for $k$ even and for $k$ odd to guarantee that
$\gen{\ga_s(1),ie^{i\vphi^p}}$ is constant for $s\in [j-1,j]$. Since these
factors do not affect the estimates below, they will be ignored.

It is an immediate consequence of \lref{L:regularity}\?(c) that
$(\ga^p_s,\vphi^p)\in \sr V_c$ for all $s\in [0,m]$ since this is true when
$s=0$.  Let $\de_j$ be as in the statement. We claim that for each $p\in K$ and $j\in [m]$:
\begin{enumerate} \item [(a)] If $s\in [0,j]$, then $\ga^p_s$ is
		$(\vphi^p,\eps_j^p)$-quasicritical of type $\sig_j$.  \item [(b)] If
			$s\in [j,m]$, then $\ga^p_s$ satisfies (i) and (ii) (with $s$ in
			place of $m$).  \end{enumerate} In particular, $(\ga^p_s,\vphi^p)\in
	\sr V_{(c,\sig_1,\dots,\sig_m)}$ for all $s\in [0,m]$ as claimed.

To establish (a), we prove by induction on $j'\in [j]$ that the intervals
$J_{j,k}(p)$ ($k\in [n_j]$) satisfy the conditions in \dref{D:quasicritical} for
the quadruple $(\ga^p_s,\vphi^p,\eps_j^p,\sig_j)$ for any $s\in [j'-1,j']$. By
hypothesis, this is true when $s=0$. By \lref{L:multiquasi}, $2\eps_{j'}^p<
\eps_{j}^p$ for all $p\in K$. Hence, by \lref{L:Js}\?(a), for any $i'\in
[l_{j'}]$ and $k'\in [n_{j'}]$, the interval $I_{i',j',k'}$ is contained in some
$J_{j,k}(p)$ whenever $p\in U_{i',j'}$. It follows immediately from \lref{L:regularity}\?(b) that
the $J_{j,k}(p)$ satisfy (i) and (ii) of \dref{D:quasicritical} for
$(\ga^p_s,\vphi^p,\eps_{j}^p,\sig_{j})$ and all $s\in [j'-1,j']$. If
$J_{j,k}(p)$ contains $I_{i',j',k'}$ for some $i'\in [l_{j'}]$ with $p\in
{U}_{i',j'}$, then condition (iii) of \dref{D:quasicritical} is satisfied by
$I_{i',j',k'}$ for all $s\in [j'-1,j']$ by \lref{L:regularity}\?(g). If not,
then $J_{j,k}(p)$ is disjoint from  $I_{i',j',k'}$ whenever $p\in U_{i',j'}$, so
that $\theta_{\ga^p_s}(t)=\theta_{\ga^p_{j'-1}}(t)$ for all $t\in J_{j,k}(p)$
and $s\in [j'-1,j']$.  In particular, if $I\subs J_{j,k}(p)$ satisfies (iii) for
$(\ga^p_s,\vphi^p,\eps_{j}^p,\sig_{j})$ when $s=j'-1$, then $I$ is not affected
by the stretching, hence it satisfies (iii) for this quadruple for all $s\in
[j'-1,j']$. This completes the proof of the induction step and of claim (a).

Now write $I_{i,j,k}=[c_{i,j,k},d_{i,j,k}]$. If $i\in [l_j]$ is such that
$\rho_{i,j}(p)\geq \frac{1}{l_j}$, then \begin{equation}\label{E:slopeestimate}
	\begin{aligned}
		\big\lan{\ga^p_{j}(d_{i,j,k})-\ga^p_{j}(c_{i,j,k})\,,\,\exp(i\vphi^p_{\sig_j(k)})}\big\ran&>
		C^{2(m+1-j)}-2R\text{, while} \\
		\big\vert\big\lan{\ga^p_{j}(d_{i,j,k})-\ga^p_{j}(c_{i,j,k})\,,\,\exp(i\vphi^p)}\big\ran\big\vert
		&< 2R.  \end{aligned} \end{equation} The first inequality is immediate
from \eqref{E:amount} and the hypothesis that the image of $\ga^p$ is contained
in the open disk $B_R(0)$. The second one comes from the fact that $ \lan
\ga_s^p(d_{i,j,k})-\ga_s^p(c_{i,j,k})\?,\?e^{i\vphi^p} \ran$ is actually
independent of $ s\in [0,j] $, as all stretchings are in the direction of $ \pm
ie^{i\vphi^p}$ and $ I_{i,j,k} $ is either disjoint or contains $ I_{i',j',k'} $
when $ \rho_{i',j'}(p)>0 $, by \lref{L:disjoint}\?(b) and the inequalities $
2\eps_j'<\eps_j $ $ (j'<j) $. By the definition of stretching,
$\ga^p_j|_{I_{i,j,k}}$ is a curve of the form $clc$. Using \eqref{E:C} we conclude
that $I_{i,j,k}\subs J_{j,k}(p)$ contains a subinterval $I$ such that
$\ga^p_j|_{I}$ is a line segment of length greater than
\begin{equation*}
C^{2(m+1-j)}-2R-2\pi\ka_0^{-1}>C^{2(m-j)+1}
\end{equation*} and slope greater in absolute value than
\begin{equation*}
	\frac{1}{2R}\big(C^{2(m+1-j)}-2R-2\pi\ka_0^{-1}\big)>
	\frac{1}{4R}C^{2(m+1-j)}> C^{2(m-j)+1}=\cot(\de_j).  \end{equation*} Hence
$\big\vert{\theta_{\ga^p_j}(t)-\vphi^p_{\sig_j(k)}}\big\vert<\de_j$ throughout
$I$, and by \rref{R:smallerandsmaller}, $\ga^p_j$ is
$(\vphi^p,\de_j)$-quasicritical of type $\sig_j$. This proves (b) when $s=j$.

We now establish (b) for all $s\in [j,m]$. Fix $p\in K$. Observe first that no
$t\in [0,1]$ can belong to two intervals $I_{i',j',k'}$ and $I_{i'',j'',k''}$
with $\rho_{i',j'}(p)>0$, $\rho_{i'',j''}(p)>0$ and
$\sig_{j'}(k')=-\sig_{j''}(k'')$. Moreover, if $j'$ is the smallest index such
that $t\in I_{i',j',k'}$ and $\rho_{i',j'}(p)>0$ for some $i'\in [l_{j'}]$ and
$k'\in [n_{j'}]$, then \begin{alignat*}{9}\label{E:wall}
	\big\vert{\tan\big(\theta_{\ga_s^p}(t)-\vphi^p_{\sig_{j'}(k')}\big)}\big\vert
	&\geq
	\frac{\vka_1}{mC^{2(m+1-j')}}>\frac{4}{C^{2(m+1-j')+1}}=4\tan(\de_{j'-1})\text{\
	\ for all $s\in [0,m]$}.  \end{alignat*} Here \eref{E:C} has been used; the
factor $m$ in the denominator of the second term  comes from the fact that $t$
belongs to at most $(m+1-j')\leq m$ such intervals. Since $4\tan x>\tan(2x)$ for
$x\in \big(0,\frac{\pi}{8}\big)$, \begin{equation}\label{E:wall2} \big\vert
	\theta_{\ga^p_s}(t)-\vphi^p_{\sig_{j'}(k')}\big\vert>2\de_{j'-1}\text{\ \
	for all $s\in [0,m]$}.	\end{equation} Now suppose that $t\in J_{j,k}(p)$
for some $k\in [n_j]$. There are three possibilities: \begin{enumerate} \item
			[\sbu]  If $t$ does not belong to any $I_{i',j',k'}$ with
			$\rho_{i',j'}(p)>0$, then
			$\theta_{\ga^p_s}(t)=\theta_{\ga^p_{0}}(t)$ for all $s\in [0,m]$ by
			construction, hence \begin{equation*}
			\big\vert{\theta_{\ga_s^p}(t)-\vphi^p_{-\sig_j(k)}}\big\vert>2\eps_j^p>2\de_j\
		\ \text{for all $s\in [0,m]$}.  \end{equation*} \item [\sbu] If $t\in
		I_{i',j',k'}$ with $\rho_{i',j'}(p)>0$ and $\sig_{j'}(k')=\sig_j(k)$,
		then \lref{L:regularity}\?(b) implies that \begin{equation*}
		\big\vert{\theta_{\ga_s^p}(t)-\vphi^p_{-\sig_j(k)}}\big\vert>\frac{\pi}{2}>2\de_j\
	\ \text{for all $s\in [0,m]$}.  \end{equation*} \item [\sbu]  If $t\in
	I_{i',j',k'}$ with $\rho_{i',j'}(p)>0$ and $\sig_{j'}(k')=-\sig_j(k)$, then
	$j'\geq j+1$, otherwise the inequality $2\eps_{j'}^p< \eps_j^p$ would
	immediately yield a contradiction. Hence, by \eqref{E:wall2},
	\begin{equation*}
		\big\vert\theta_{\ga_s^p}(t)-\vphi^p_{-\sig_j(k)}\big\vert>2\de_{j}\ \
		\text{for all $s\in [0,m]$}.  \end{equation*} \end{enumerate} Thus, in
		any case condition (i) of \dref{D:quasicritical} is satisfied by
		$(\ga^p_s,\vphi^p,\de_j,\sig_j)$ for all $s\in [0,m]$. Similarly, if
		$t\nin \Int\big(\bcup_{k=1}^{n_j}J_{j,k}(p)\big)$, then either $t$ does
		not belong to any $I_{i',j',k'}$ with $\rho_{i',j'}(p)>0$ or $t\in
		I_{i',j',k'}$ with $\rho_{i',j'}(p)>0$ for some $j'\geq j+1$. Then, by
		the same reason as in the first and third possibilities above,
		\begin{equation*}
			\big\vert\theta_{\ga_s^p}(t)-\vphi^p_{\pm}\big\vert>2\de_{j}\ \
			\text{for all $s\in [0,m]$}.  \end{equation*} This proves that
		condition (ii) of \dref{D:quasicritical} is also satisfied for all $s\in
		[0,m]$. Finally, we shall prove by induction on $j'$ ($j\leq j'\leq m$)
		that condition (iii) holds for all $s\in [j,j']$. For $j'=j$, this was
		established in the preceding paragraph; let $I\subs J_{j,k}(p)$ be as
		described there and assume that $j'>j$. By \lref{L:multiquasi},
		$\eps_{j'}^p>2\eps_j^p$, hence \lref{L:disjoint}\?(b) implies that if
		$\rho_{i',j'}(p)>0$ and $k'\in [n_{j'}]$, then either $I\subs
		[c_{i',j',k'},d_{i',j',k'}]$ or these two intervals are disjoint. If $I$
		is disjoint from any such interval, then
		$\theta_{\ga^p_s}(t)=\theta_{\ga^p_{j'-1}}(t)$ for all $t\in I$ and
		$s\in [j'-1,j']$. Hence $I\subs J_{j,k}(p)$ satisfies condition (iii) of
		\dref{D:quasicritical} for all such $s$, since by the induction
		hypothesis this is true when $s=j'-1$. Suppose then that $I \subs
		[c_{i',j',k'},d_{i',j',k'}]$ for some $i',\,k'$ with
		$\rho_{i',j'}(p)>0$. Using \lref{L:lowering}\,(e) and reducing $I$ if
		necessary, it can be assumed that $\ga_s^p|_I$ is a line segment for all
		$s\in [j'-1,j']$. Let $s_0\in [j'-1,j']$ correspond to the instant where
		the flattening deformation ends and the stretching begins. The same
		estimates as in \eqref{E:slopeestimate} show that the slope of
		$\ga_{s_0}^p|_I$ is greater than $\cot(\de_j)$. Since this is also true
		when $s=j'-1$ by the induction hypothesis, it follows from the
		monotonicity of $\theta_{\ga^p_s}(t)$ with respect to $s\in [j'-1,s_0]$
		(see Lemma 3.11 of \cite{SalZueh1}) that this holds for all $s\in
		[j'-1,s_0]$. For $s\in [s_0,j-1]$ the same conclusion holds by
		\cref{C:eight}. 
\end{proof}

\begin{lem}\label{L:pulleys} Let $ \varrho>0$ $g\colon [0,m]\times K\to \sr
	V_{(c,\sig_1,\dots,\sig_m)}$ be as in \lref{L:stretching} and
	$n=\abs{\sig_m}$. Then $g$ admits an extension to $[0,m+2n]\times K$,
	$g_{s}(p)=(\ga^p_s,\vphi^p)$, such that $\ga_{m+2n}^p$ is of the form
	\begin{equation*}
	\text{c}\underbrace{\text{lc\dots lc}}_{n}
	\end{equation*} and each line segment has length $ >8 $ and slope
	greater in absolute value than $ \varrho $, for all $p\in K$.  \end{lem}

\begin{figure}[ht] \begin{center} \includegraphics[scale=.22]{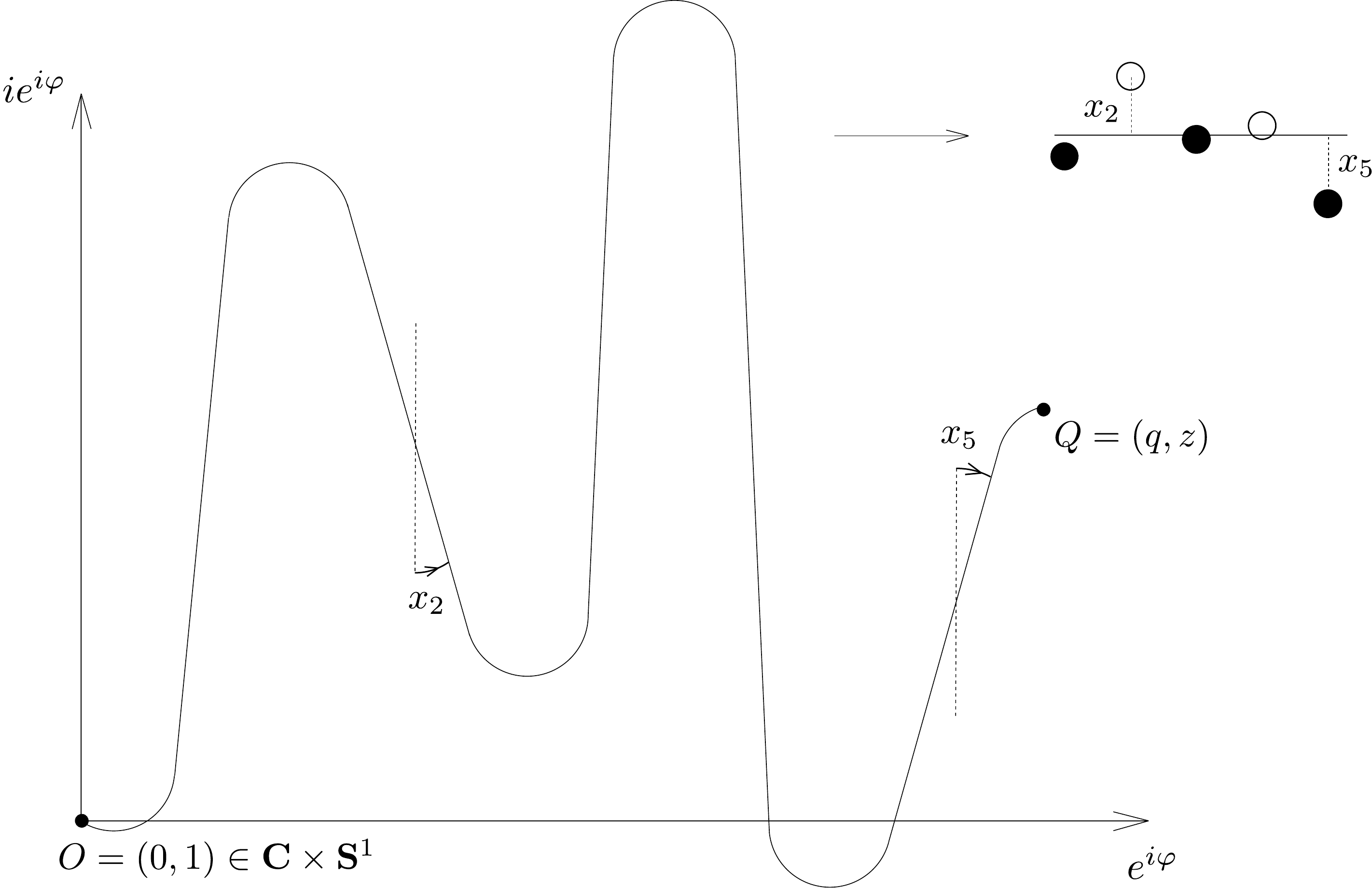}
		\caption{An illustration of a curve obtained by the homotopy in
			\lref{L:pulleys}.} \label{F:pulleys} \end{center} \end{figure}

\begin{proof} Again, we carry out the proof only for $ \sr
	V_{(c,\sig_1,\dots,\sig_m)} $, since the proof for $ \sr
	V_{(\sig_1,\dots,\sig_m)} $ is the same, except for a few omissions. 
	We retain the notation of the proof of \lref{L:stretching}. Let
	$J_k(p)=[a_k(p),b_k(p)]$ ($k\in [n]$) be intervals satisfying the conditions
	of \dref{L:stretching} for the quadruple
	$(\ga^p_0,\vphi^p,\eps_m^p,\sig_m)$, and hence the same conditions for
	$(\ga^p_m,\vphi^p,\de_m,\sig_m)$. Set $t_0(p)=0$, $t_{2n}(p)=1$,
	\begin{equation}\label{E:tnu} t_{2k-1}(p)=\frac{1}{2}[a_k(p)+b_k(p)] \ \
		(k\in [n])\text{\ \ and\ \ } t_{2k}(p)=\frac{1}{2}[b_k(p)+a_{k+1}(p)]\ \
		(k\in [n-1]).  \end{equation} Notice that $t_{2k-1}(p)\in J_k(p)$ for
	all $k\in [n]$ and $J_k(p)<t_{2k}(p)<J_{k+1}(p)$ for all $k\in [n-1]$. For
	each $\nu\in [2n]$, let $I_\nu(p)=[t_{\nu-1}(p),t_{\nu}(p)]$. Since
	$I_\nu(p)$ intersects exactly one $J_k(p)$ for all $\nu$, the amplitude
	\begin{equation*}
	\om\big(\ga_{m}^p|_{I_\nu(p)}\big)=\sup_{t\in
		I_\nu(p)}\big\{\theta_{\ga_{m}^p}(t)\big\}-\inf_{t\in
		I_\nu(p)}\big\{\theta_{\ga_{m}^p}(t)\big\} \end{equation*} is less than
	$\pi$ for all $\nu\in [2n]$. Thus, $\ga_{m}^p|_{I_\nu(p)}$ can be flattened;
	let \begin{equation*}
	\psi^p_\nu=\frac{1}{2}\Big(\sup_{t\in
		I_\nu(p)}\theta_{\ga_{m}^p}(t)+\inf_{t\in
		I_\nu(p)}\theta_{\ga_{m}^p}(t)\Big)\ \ (\nu\in [2n]).
	\end{equation*} Extend the homotopy of \lref{L:stretching} to $[0,m+1]\times
	K$ by letting $\ga_{m+s}^p|_{I_\nu(p)}$ be the flattening of
	$\ga_m^p|_{I_\nu(p)}$ in the direction of $e^{i\psi^p_\nu}$ $(s\in [0,1])$.
	It follows immediately from  \lref{L:regularity}\?(d) that
	$(\ga^p_{m+s},\vphi^p)\in \sr V_{c}$ for all $s\in [0,1]$. 

Again, it needs to be verified that $(\ga^p_{m+s},\vphi^p)\in \sr
V_{(\sig_1,\dots,\sig_m)}$. We claim that $\ga^p_{m+s}$ is
$(\vphi^p,\de_j)$-quasicritical of type $\sig_j$ for all $s\in [0,1]$ and $j\in
[m]$. Fix $p\in K$ and let \begin{equation*}
	J_{j,1}(p)<\dots<J_{j,n_j}(p)\quad (j\in [m]) \end{equation*} be intervals
as in \lref{L:stretching} for $\sig_j$. Using \lref{L:convenient}, it may be
assumed that any such interval has the form $J_{k_1}(p)\ast J_{k_2}(p)$ for some
$k_1,k_2\in [n]$. (It is however unnecessary to assume that the endpoints of
$J_{j,k}(p)$ depend continuously on $p$, since no constructions using them will
be carried out.) 

Let $\nu\in [2n]$. As $I_\nu(p)$ intersects exactly one $J_k(p)$, it intersects
at most one $J_{j,r}(p)$. Thus, if  $I_\nu(p)\cap J_{j,r}(p)\neq \emptyset$,
then $\big\vert{\theta_{\ga_{m+s}^p}(t)-\vphi^p_{-\sig_j(r)}}\big\vert>2\de_j$
for $s=0$ and all $t\in I_\nu(p)$. By \lref{L:regularity}\?(d), this inequality
holds for all $s\in [0,1]$. Since $\bcup_\nu I_\nu(p)=[0,1]$, we conclude that
\begin{equation}\label{E:larger}
	\big\vert{\theta_{\ga_{m+s}^p}(t)-\vphi^p_{-\sig_j(r)}}\big\vert>2\de_j\text{\
		\ for all $s\in [0,1]$,~$t\in J_{j,r}(p)$.}
		\end{equation}

Now let $I\subs J_{j,r}(p)$ be an interval as in (ii) of \lref{L:stretching}.
Let $\nu\in [2n]$ be such that $I\subs I_{\nu}(p)\cup I_{\nu+1}(p)$. Then the
restriction of $\ga^p_m$ to one of $I\cap I_\nu(p)$ or $I\cap I_{\nu+1}(p)$ has
length equal to at least half the length of $\ga^p_m|_{I}$. Suppose without loss
of generality that the former occurs. Let $R>1$ and $C$ be as in the proof of
\lref{L:stretching}. Then \begin{equation*}
	\begin{aligned}
		\abs{\big\lan{\ga^p_{m}(t_{\nu(p)})-\ga^p_{m}(t_{\nu-1}(p))\,,\,i\exp(i\vphi^p)}\big\ran}&>
		\frac{1}{2}C^{2(m+1-j)}-2R\text{, while} \\
		\abs{\big\lan{\ga^p_{m}(t_{\nu(p)})-\ga^p_{m}(t_{\nu-1}(p))\,,\,\exp(i\vphi^p)}\big\ran}
		&< 2R.  \end{aligned} \end{equation*} Recall that by the definition of
flattening, $\ga^p_{m+s}(t_{\nu}(p))=\ga^p_{m}(t_{\nu}(p))$ for all $s\in
[0,1]$, and similarly at $t_{\nu-1}(p)$. Moreover, $\ga^p_{m+1}|_{I_\nu(p)}$ is
of the form $clc$. Its subarc which is a line segment  must thus have slope greater
in absolute value than \begin{equation*}
	\frac{1}{2R}\Big(\frac{1}{2}C^{2(m+1-j)}-2R-2\pi\ka_0^{-1}\Big)>
	\frac{1}{8R}C^{2(m+1-j)}> C^{2(m-j)+1}=\cot(\de_j).  \end{equation*} Let
$I'\subs I\cap I_{\nu}(p)$ be an interval such that $\ga^p_{m+s}|_{I'}$ is a
line segment of length $>8$ for all $s\in [0,1]$, as guaranteed by
\lref{L:lowering}\,(e). Then $\ga^p_{m+s}|_{I'}$ is $\ka_0$-stretchable by
\lref{L:lowering}\,(f). Further, the above estimate implies that
$\big\vert{\theta_{\ga^p_{m+s}}(t)-\vphi^p_{\sig_j(k)}}\big\vert<\de_j$
throughout $I'$ when $s=1$, and this is also true when $s=0$ by
\lref{L:stretching}. By the monotonocity of $\theta_{\ga^p_{m+s}(t)}$ with
respect to $s$ (see Lemma 3.11 of \cite{SalZueh1}), this inequality holds for
all $s\in [0,1]$. Thus, condition (iii) of \dref{D:quasicritical} is satisfied.

To verify condition (ii), let $\nu_0,\nu_1$ be the greatest (resp.~smallest)
index satisfying $t_{\nu_0}(p)<J_{j,k}(p)<t_{\nu_1}(p)$. Since $t_{\nu_i}(p)\nin
\bcup_kJ_k(p)$ by definition, \begin{equation*}
	\abs{\theta_{\ga_{m+s}}(t_{\nu_i}(p))-\vphi^p}=\abs{\theta_{\ga_{m}}(t_{\nu_i}(p))-\vphi^p}<\frac{\pi}{2}-2\eps_m^p\text{\
	\ for $i=0,1$ and all $s\in [0,1]$.} \end{equation*} Therefore, it is
possible to enlarge $J_{j,r}(p)$ to a subinterval of
$\big(t_{\nu_0}(p),t_{\nu_1}(p)\big)$ so that \begin{equation*}
	\big\vert{\theta_{\ga^p_{m+s}}(t)-\vphi^p}\big\vert<\frac{\pi}{2}-2\de_j\ \
	\text{for all $t\in \big[t_{\nu_0}(p),t_{\nu_1}(p)\big]\ssm J_{j,r}(p)$ and
$s\in [0,1]$.} \end{equation*} If this enlargement is carried out for each $k\in
[n_j]$, then condition (ii) of \dref{D:quasicritical} will be satisfied by the
$J_{j,r}(p)$. It is easily verified that the validity of conditions (i) and
(iii) is not affected.

Now $\ga^p_{m+1}$ is of the form \begin{equation*}
	\underbrace{(clc)(clc)\dots(clc)}_{n}\text{\ for all $p\in K$.}
\end{equation*} To prove the lemma, it thus suffices to reduce subarcs of the
form $cclc$ to arcs of the form $clc$. Let $L^p$ denote the length of $\ga^p_{m+1}$; no
generality is lost in assuming that $\ga^p_{m+1}\colon [0,1]\to \C$ is
parametrized proportionally to arc-length for all $p$. Set
\begin{equation}\label{E:Anu}
	A_\nu(p):=\big[t_{\nu}(p)-\frac{\pi}{\ka_0L^p}\,,\,t_{\nu+1}(p)\big]
\end{equation} For each $\nu=1,\dots,2n-1$ in turn, let $\ga^p_{m+\nu+1}$ be
obtained from $\ga^p_{m+\nu}$ by flattening the arc $\ga^p_{m+\nu}|_{A_\nu(p)}$
in the direction of \begin{equation*}
\frac{1}{2}\Big(\sup_{t\in
	A_\nu(p)}\theta_{\ga_{m+\nu}^p}(t)+\inf_{t\in
	A_\nu(p)}\theta_{\ga_{m+\nu}^p}(t)\Big).  \end{equation*} Using
estimates similar to the preceding ones, it is not hard to check that
$\ga^p_s\in \sr V_{(c,\sig_1,\dots,\sig_m)}$ for all $s\in [m+1,m+2n]$.
Moreover, $\ga^p_{m+2n}$ has the desired form for all $p\in K$ by construction.
\end{proof}

The next objective is to prove a version of \lref{L:stretching} and
\lref{L:pulleys} for $\sr V_{(d,\sig_1,\dots,\sig_m)}$. The proof is a
repetition of the arguments used to establish these results, aside from some
preliminary deformations which are needed to guarantee that $\ga_s^p$ will
remain diffuse throughout the homotopy. We begin with a lemma which allows us to
deform a family $ K\to \sr V_{(\sig_1,\dots,\sig_m)} $ to have image contained
in $ \sr V_{(d,\sig_1,\dots,\sig_m)} $.

\begin{lem}\label{L:diffusing} Let $K\to \sr V$,
	$p\mapsto (\ga_0^p,\vphi^p)$ be a continuous map, where 
	$\sr V=\sr V_{(d,\sig_1,\dots,\sig_m)}$ or $ \sr V=\sr
	V_{(\sig_1,\dots,\sig_m)} $. 
 	Then there exists a homotopy $(s,p)\mapsto
	(\ga^p_s,\vphi^p)\in \sr V$ such that 
	$ [\vphi^p_-,\vphi^p_+]\subs \Int (\theta_{\ga_1^p}([0,1])) $ for all $
	p\in K $. 
	\end{lem} 
	Thus, by deforming $\ga^p_0$ they can be made not only diffuse but ``diffuse
	with respect to $\vphi^p$''.

\begin{proof}	
	Let $ \eps_j\colon K\to \R^+ $ be such that $ \ga^p=\ga_0^p $ is $
	(\vphi^p,\eps^p) $-quasicritical of type $ \sig_j $ for each $ j\in [m] $
	and $ p\in K $. Assume first that $ \sr V=\sr V_{(d,\sig_1,\dots,\sig_m)} $
	and that $K$ consists of a single point $p$.
	Since $ \ga^p $ is diffuse, the image of $ \theta_{\ga^p} $ has diameter
	greater than $ \pi $, and it contains $
	[\vphi^p_-+\eps_1^p,\vphi^p_+-\eps^p_1] $ in its interior by condition (iii)
	of \dref{D:quasicritical}. Hence there exist $t,t'\in [0,1]$ such that
	\begin{equation*}
		\theta_{\ga^p}(t')=\pi+\theta_{\ga^p}(t)\text{\ \,and\ \,}
		\theta_{\ga^p}(t)<\vphi^p_-+\eps_1^p< 
		\vphi^p_+-\eps^p_1< \theta_{\ga^p}(t').
		\end{equation*}
	Define a homotopy
	$(s,p)\mapsto \ga_s^p$ ($s\in [0,\frac{1}{2}]$) by grafting straight line
	segments having directions $\ta_{\ga^p}(t)$, $\ta_{\ga^p}(t')$ and length
	greater than 4 at $\ga^p(t)$ and $\ga^p(t')$ (see \cite{SalZueh1}, Definition
	4.13). Note that $(\ga^p_s,\vphi^p)\in \sr V$ for all $s\in
	[0,\frac{1}{2}]$, since $\theta_{\ga^p_s}$ is essentially the same function
	as $\theta_{\ga^p_0}$. Extend the homotopy to all of $[0,1]$ by deforming
	each of these segments to create a ``bump'' 
	(see Figure 10 of \cite{SalZueh1}) so that 
	$ [\vphi^p_-,\vphi^p_+]\subs \Int (\theta_{\ga_1^p}([0,1])) $.
	This is possible because $\sr M(P)$ is connected if $P=(x,1)\in \R\times
	\Ss^1$ with $x> 4$, by Theorem 6.1 of \cite{SalZueh1};
	this also follows from \fref{F:generator} above. Moreover, $
	\ga_s^p\in \sr V_{(d,\sig_1,\dots,\sig_m)} $ for all $ s\in [0,1] $.
	For a general finite simplicial complex $K$, the same idea works if
	partitions of unity are used.  The details will be omitted since they are
	technical and an entirely  similar construction (for deforming segments into
	eight curves, instead of bumps) was already carried out in Lemmas 4.15 and
	4.16 of \cite{SalZueh1}.
	
	Now take $\sr V=\sr V_{(\sig_1,\dots,\sig_m)}$. By Corollary 1.11 of
	\cite{SalZueh1}, it may be assumed that each $\ga^p$ is smooth and that all
	of its derivatives depend continuously upon $p\in K$.  Choose $\ka_0\in
	(\frac{1}{2},1)$ such that $\ka_{\ga^p}([0,1])\subs (-\ka_0,+\ka_0)$ for
	every $p\in K$. Assume first that $K=\se{p}$. Let $J_k(p)$ be intervals
	satisfying \dref{D:quasicritical} for the sign string $\sig_m$ and some
	$\eps>0$, and choose stretchable intervals $I\subs J_k(p)$, $I'\subs
	J_{k'}(p)$ with $\sig_m(k)=\ty{+}$ and $\sig_m(k')=\ty{-}$. By choosing a
	larger $\ka_0\in (\frac{1}{2},1)$ if necessary, it may be assumed that the
	restriction of $\ga^p$ to each of $I,I'$ is $\ka_0$-stretchable with respect
	to $\vphi^p_{\sig_m(k)}$. Define a homotopy $(s,p)\mapsto \ga_s^p$ by
	stretching each of $\ga^p_0|_{I},\ga^p_0|_{I'}$ in the direction of $\pm
	ie^{i\vphi^p}$ by more than $4+2\pi$, linearly with $s\in [0,\frac{1}{2}]$.
	Extend this to $[0,1]$ by choosing straight line segments of length greater
	than $4$ within each of $\ga^p_{\frac{1}{2}}|_{I},\ga^p_{\frac{1}{2}}|_{I'}$
	and deforming them to create bumps as above so as to have 
	$ [\vphi^p_-,\vphi^p_+]\subs \Int (\theta_{\ga_1^p}([0,1])) $. For a general
	finite simplicial complex $K$, use partitions of unity,
	\lref{L:continuousJs} and \lref{L:disjoint}\?(a).
\end{proof}

\begin{lem}\label{L:compressing} Let $K\to \sr V_{(d,\sig_1,\dots,\sig_m)}$,
	$p\mapsto (\ga^p,\vphi^p)$ be a continuous map and assume that
	$ [\vphi^p_-,\vphi^p_+]\subs \Int (\theta_{\ga^p}([0,1])) $ for all $ p\in
	K$.  Then given $\de_0>0$, there exists a homotopy
	$(s,p)\mapsto (\ga^p_s,\vphi^p)\in \sr V_{(d,\sig_1,\dots,\sig_m)}$ such
	that $\ga_0^p=\ga^p$ and $ [\vphi^p_-,\vphi^p_+]\subs \Int
	(\theta_{\ga_1^p}([0,1]))\subs [\vphi^p_--\de_0,\vphi^p_++\de_0] $ for all $
	p\in K$. 
	Moreover, the homotopy is obtained by stretching
	subarcs of $\ga^p$ in the direction of $\pm ie^{i\vphi^p}$.  \end{lem}
\begin{proof} By Corollary 1.11 of \cite{SalZueh1}, no generality is lost in
	assuming that $\ga^p$ is smooth for every $p\in K$, and that its derivatives
	depend continuously on $p$. In particular, there exists $\ka_0\in (0,1)$
	such that $\ka_{\ga^p}([0,1])\subs (-\ka_0,+\ka_0)$ for all $p\in K$. Fix
	$p$ and let \begin{equation*}
	W_p=\set{t\in [0,1]}{\big\vert
			\theta_{\ga^p}(t)-\vphi^p\big\vert>\tfrac{\pi}{2}},\ \ C_p=\set{t\in
			[0,1]}{\big\vert \theta_{\ga^p}(t)-\vphi^p\big\vert\geq
			\tfrac{\pi}{2}+\tfrac{\de_0}{2}}.  \end{equation*} By
		\rref{L:lowering}\?(a), the closure of any component of $W_p$ is a
		$\ka_0$-stretchable interval for $\ga^p$. Moreover, $C_p$ is compact,
		hence it intersects only finitely many of the components of $W_p$.
		Choose disjoint intervals $[c_k,d_k]$ ($k\in [n]$,~ $n=n(p)\in \N)$,
		such that: \begin{enumerate} \item [\sbu] $C_p\subs
				\bcup_{k=1}^n[c_k,d_k]$; \item [\sbu]  $\ga^p|_{[c_k,d_k]}$ is
				$\ka_0$-stretchable with respect to $\vphi^p_{\pm}$ for every
			$k\in [n]$; \item [\sbu]
				$\big\vert{\theta_{\ga^p}-\vphi^p}\big\vert>\frac{\pi}{2}$
				throughout $[c_k,d_k]$; \item [\sbu] $\big\vert
					\theta_{\ga^p}(c_k)-\vphi^p\big\vert<\frac{\pi}{2}+\frac{\de_0}{2}$
					and $\big\vert
					\theta_{\ga^p}(d_k)-\vphi^p\big\vert<\frac{\pi}{2}+\frac{\de_0}{2}$.
			\end{enumerate} Let $U_p\subs K$ be a neighborhood of $p$ such that
			these conditions still hold if $p$ is replaced by any $q\in
			\ol{U}_p$. Cover $K$ by finitely many such open sets $U_i$ ($i\in
			[l]$), with associated stretchable intervals $[c_k^i,d_k^i]\subs
			[0,1]$, $k\in [n(i)]$. By the argument used in the proof of
			\lref{L:disjoint}\?(a), it may be assumed that if $i<i'$ and
			$\ol{U}_i\cap \ol{U}_{i'}\neq \emptyset$, then for each $k\in
			[n(i)]$ and $k'\in [n(i')]$, either $[c_k^i,d_k^i]\subs
			[c_{k'}^{i'},d_{k'}^{i'}]$ or these two intervals are disjoint. Let
			$(\rho_i)_{i\in [l]}$, $\rho_i\colon K\to [0,1]$, be a partition of
			unity subordinate to $(U_i)_{i\in [l]}$. Let $m_{\pm}(i)$ denote the
			cardinality of \begin{equation*}
			S_{\pm}(i)=\set{k\in
				[n(i)]}{\pm \sign\big(\theta_{\ga^p}(t)-\vphi^p\big)>0\text{\
					for all $t\in [c_{k}^i,d_k^i]$}}.  \end{equation*} Observe
				that $m_+(i),\,m_-(i)\geq 1$ by hypothesis. Let
				$M>0$ and for each $i=1,\dots,l$ successively, let $\ga_s^p$
				($s\in [\tfrac{i-1}{l},\tfrac{i}{l}]$) be obtained by stretching
				\begin{equation}\label{E:diffamount}
					\ga_{\tfrac{i-1}{l}}^p|_{[c_k^i,d_k^i]}	\text{ by
					}\begin{cases} m_-(i)\rho_i(p)M & \text{ if $k\in S_+(i)$}
						\\ m_+(i)\rho_i(p)M & \text{ if $k\in S_-(i)$}
					\end{cases} \text{ for each $k\in [n(i)]$}.  \end{equation}
				The factors $m_{\pm}(i)$ are incorporated here to guarantee that
				$\ga_s(1)=q$ for all $s\in [0,1]$. By \lref{L:regularity}\?(b),
				(c) and (g), for each $p\in K$, the four conditions listed above remain
				valid for $\ga^p_s$ ($s\in [0,1]$), so that this deformation is
				well-defined. Further, by \lref{L:regularity}\?(f), if $M$ is
				large enough, then  the resulting curves $\ga_1^p$ will satisfy 
				the required property for all $p\in K$.  \end{proof}

\begin{lem}\label{L:stretchingd} Let $g\colon K\to \sr
	V_{(d,\sig_1,\dots,\sig_m)}$, $g(p)\mapsto (\ga^p,\vphi^p)$,  be a
	continuous map. Then there exists
	a homotopy $g\colon [0,5]\times K\to \sr V_{(d,\sig_1,\dots,\sig_m)}$,
	$g_{s}(p)=(\ga^p_s,\vphi^p)$, such that $\ga_0^p=\ga^p$ and $\ga_{5}^p$ is
	of the form \begin{equation*}
	\text{c}\underbrace{\text{lc\dots
		lc}}_{n} \end{equation*} and each of its straight arcs has length
	greater than 8,  for all $p\in K$.  \end{lem}

\begin{proof} Let the notation be as in the first paragraph of
	\lref{L:stretching} and let $\theta^p:=\theta_{\ga^p}$ and
	$\theta^p_s:=\theta_{\ga^p_s}$ (where $\ga_s^p$ is to be defined below).
	Since $0\in R(Q)$ by definition (see \dref{D:Vs}), $\cos
	\vphi^p=\gen{1,e^{i\vphi^p}}>0$ for all $p\in K$. By \lref{L:diffusing}, it
	may be assumed that 
	$ [\vphi^p_-,\vphi^p_+]\subs \Int (\theta_{\ga_1^p}([0,1])) $
	for all $ p\in K $. Given $p$, choose $u_j\in
	[0,1]$ ($j=1,2$) such that 
	\begin{equation*}
		\theta^p(u_1)<\vphi^p-\tfrac{\pi}{2}<\vphi^p+\tfrac{\pi}{2}<\theta^p(u_2)
	\end{equation*} and
	the origin $0\in \C$ lies in the interior of the triangle whose vertices are
	$1$, $e^{i\theta^p(u_{1})}$ and $e^{i\theta^p(u_{2})}$. These conditions are
	still satisfied throughout a neighborhood $U_p$ of $p$. Let $(U_i)_{i\in
	[l_0]}$ be a subcover of the resulting cover of $K$ and $u_{i,j}\in [0,1]$
	be the corresponding numbers.  Then we can write
	\begin{equation*}
		0=a_{i,0}(p)+a_{i,1}(p)e^{i\theta^p(u_{i,1})}+a_{i,2}(p)e^{i\theta^p(u_{i,2})}\
		\ \text{for some $a_{i,j}(p)>0$ and all $p\in U_i$.} \end{equation*} Moreover,
	$a_{i,j}\colon U_i\to \R^+$ can be chosen to depend continuously on $p$ and
	as large as desired for each $j=0,1,2$. Let $\rho_i\colon K\to [0,1]$ be a
	partition of unity subordinate to $(U_i)_{i\in [l_0]}$. Set $\ga_0^p:=\ga^p$
	and define a homotopy $[0,1]\times K\to \sr V_{(d,\sig_1,\dots,\sig_m)}$,
	$(s,p)\mapsto (\ga^p_s,\vphi^p)$, by grafting straight segments linearly
	with $s$ onto $\ga^p$ at $t=0,\,u_{i,1}(p)$ and $u_{i,2}(p)$ of lengths
	$L_{i,j}(p)=\rho_{i}(p)a_{i,j}(p)$ ($j=0,1,2$, respectively) for all $i\in
	[l_0]$ and $p\in K$. As before, let $R>0$ be such that the image of $\ga_0^p$
	is contained in $B_R(0)$ for all $p\in K$. By taking the $a_{i,j}$ to be
	sufficiently large, it can be guaranteed that for each $p\in K$ there exists
	$i\in [l_0]$ such that \begin{equation*}
		\big\lan{L_{i,j}(p)e^{i\theta_1^p(u_{i,j})},e^{i\vphi^p}}\big\ran<-2(R+2\pi)\text{\
		for $j=1,2$.} \end{equation*} In words, $\ga^p_1$ ``retrocedes'' by at
	least $2(R+2\pi)$ at $t=u_{i,1}$ and $t=u_{i,2}$, with respect to the axis
	$e^{i\vphi^p}$. Thus if $k_j\in [n_m]$ is such that $u_{i,j}\in J_{k_j}(p)$
	($j=1,2$ and $J_{k}(p)$ as at the beginning of the proof of \lref{L:pulleys}), then
	\begin{equation}\label{E:preliminary} \big\lan
		\ga^p_1\big(t_{2k_j}(p)\big)-\ga^p_1\big(t_{2k_j-2}(p)\big)\,,\,e^{i\vphi^p}\big\ran<-4\pi.
	\end{equation} Here $t_\nu(p)$ is as in \eqref{E:tnu}; note that
	$t_{2k_j-2}(p)<J_{k_j}(p)<t_{2k_j}(p)$. The crucial observation here is that
	\eqref{E:preliminary} implies the existence of $t'\in
	[t_{2k_j-2}(p),t_{2k_j}(p)]$ such that
	$(-1)^j\big(\theta_1^p(t')-\vphi^p\big)>\frac{\pi}{2}$.
	
	Let $\de_0$ be given by as in \lref{L:stretching}\?(i) 
	and apply \lref{L:compressing} to
	$\ga^p_1$, extending the homotopy to $[0,2]\times K$. (This deformation is
	necessary to be able to apply \rref{R:smallerandsmaller} as in the proof of
	\lref{L:stretching}.) Because the subarcs of $\ga^p_1$ which are stretched
	in this homotopy all lie in the interior of some $J_k(p)$, and they are
	stretched in the direction of $i\pm e^{i\vphi^p}$,  the coordinate
	$\gen{\ga^p_s(t),e^{i\vphi^p}}$ is the same for all $s\in [1,2]$ provided
	that $t\nin \bcup_k J_k(p)$. Hence, \eqref{E:preliminary} is valid with $
	\ga_s $ in place of $ \ga_1 $ $ (s\in [1,2]) $. 
	Now take $R'>0$ such that the image of $\ga^p_2$ is contained in
	the open disk $B_{R'}(0)$ for all $p\in K$, and take $C$ as in
	\eqref{E:C}, but replacing $R$ by $R'$. Finally, extend the homotopy to
	$[0,5]\times K$ by repeating the proofs of \lref{L:stretching} and
	\lref{L:pulleys}, with $R'$ in place of $R$. We claim that
	\eqref{E:preliminary} is sufficient to guarantee that $\ga^p_s$ remains
	diffuse when the constructions in \lref{L:stretching} and \lref{L:pulleys}
	are carried out for $s\in [2,5]$. There are three constructions to consider,
	which will be assumed to take place for $s\in [2,3]$, $[3,4]$ and $[4,5]$,
	respectively. The first one, in the proof of \lref{L:stretching}, involves
	stretching subarcs of $\ga^p_2$ in the direction of $\pm e^{i\vphi^p}$; as
	above, this does not affect the validity of \eqref{E:preliminary} since
	$t_{\nu}(p)\nin \bcup_k J_k(p)$ for all even $\nu$. The second, at the
	beginning of the proof of \lref{L:pulleys}, involves flattening each of the
	subarcs $\ga^p_3|_{[t_{\nu-1}(p),t_{\nu}(p)]}$; clearly, this also does not
	affect \eqref{E:preliminary}, because by the definition of flattening,
	$\ga^p_s(t)$ remains constant at the endpoints $t=t_{\nu-1}(p)$ and
	$t_\nu(p)$, as well as outside of $[t_{\nu-1}(p),t_{\nu}(p)]$. The last
	step, near the end of the proof of \lref{L:pulleys}, is to flatten the
	restriction of $\ga^p_4$ to the intervals \eqref{E:Anu}. This may affect
	\eqref{E:preliminary}, but it can still be guaranteed that
	\begin{equation*}
	\big\lan
		\ga^p_s\big(t_{2k_j}(p)\big)-\ga^p_s\big(t_{2k_j-2}(p)\big)\,,\,e^{i\vphi^p}\big\ran<0\text{
		for all $s\in [4,5]$, $j=1,2$,} \end{equation*} because the restriction
	of $\ga^p_4$ to $A_{\nu}(p)\ssm [t_{\nu}(p),t_{\nu+1}(p)]$ has length
	$\frac{\pi}{\ka_0}<2\pi$. Thus, for each $p\in K$ and $s\in [0,5]$, there
	exist $v_1,\,v_2\in [0,1]$ satisfying $
	\theta_s^p(v_1)<\vphi^p_-<\vphi^p_+<\theta_s^p(v_2) $, so that $ \ga_s $ is
	diffuse for all $ s $.  \end{proof}

Given any family $(\ga^p,\vphi^p)\in \sr V_{(c,\sig_1,\dots,\sig_m)}$, $ \sr
V_{(\sig_1,\dots,\sig_m)} $ or $\sr
V_{(d,\sig_1,\dots,\sig_m)}$ indexed by a finite simplicial complex, we have
shown that $\ga^p$ can be continuously deformed to look like a curve $\eta^p$ as
in Figure \ref{F:pulleys}. To finish the proof of \pref{P:good}, it thus
suffices to show that any such family is contractible. This is true because any
$\eta$ as in the figure is essentially determined by $p(\eta,\vphi)=(x,\vphi)$,
where $x=(x_1,\dots,x_n)$ is obtained as indicated there and $n=\abs{\sig_m}$.
To make this more precise, we begin by describing a construction which implies
that each fiber of $p$ is contractible. It will then be shown that $p$ is a
quasifibration.

\begin{cons}\label{C:locconvex} Let $\ga_0,\ga_1\colon [0,1]\to \C$ be two
	regular curves parametrized proportionally to arc-length and
	$\theta_{\ga_j}\colon [0,1]\to \R$ be continuous functions satisfying
	$\exp\big(i\theta_{\ga_j}\big)=\ta_{\ga_j}$ $(j=0,1)$. Let
	$\vtheta_0,\vtheta_1\in \R$ and $\ka_1\in (0,1)$. Suppose that
	$\ga=\ga_0,\ga_1$ satisfies: \begin{enumerate} \item [(i)]
			$\theta_{\ga}(0)=\vtheta_0$ and $\theta_{\ga}(1)=\vtheta_1$; \item
				[(ii)] $\ka_{\ga}\colon [0,1]\to [0,\ka_1]$ is a step function.
		\end{enumerate} Recall that $\dot\theta_\eta=\abs{\dot\eta}\ka_\eta$ for
		any piecewise $C^2$ curve (except at finitely many points). Condition
		(ii) thus implies that $\theta_\ga$ is an increasing, piecewise linear
		function. We shall describe a homotopy $s\mapsto \ga_s$ ($s\in [0,1]$)
		joining $\ga_0$ to $\ga_1$ through regular curves satisfying (i) and
		(ii). The idea is to parametrize both $\ga_j$ by the argument $\theta\in
		[\vtheta_0,\vtheta_1]$ and use convex combinations; this works only if
		both $\theta_{\ga_j}$ are strictly increasing, but an easy adaptation
		also covers the general case. See Figure \ref{F:convex}.
	
\begin{figure}[ht] \begin{center} \includegraphics[scale=.42]{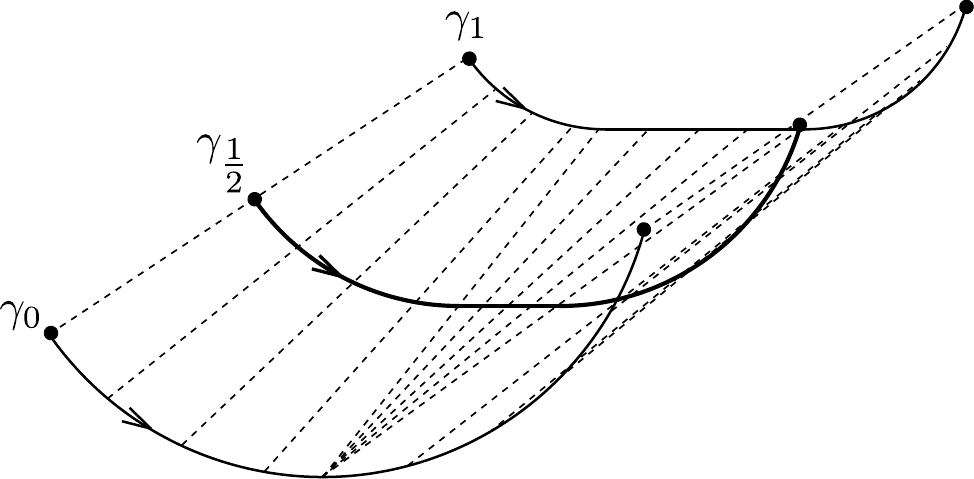}
		\caption{An illustration of \cref{C:locconvex}.} \label{F:convex}
	\end{center} \end{figure}
	
	Let $\se{\al_1<\dots<\al_n}\subs [\vtheta_0,\vtheta_1]$ be the union of the
	set of critical values of $\theta_{\ga_0}$ and $\theta_{\ga_1}$. For each
	$k\in [n]$, let $[a^k_j,b^k_j]\subs [0,1]$ denote the interval
	$\theta_{\ga_j}^{-1}(\se{\al_k})$. Define a reparametrization $\eta_j\colon
	[\vtheta_0,\vtheta_1+n]\to \C$ of $\ga_j$ as follows: The restriction of
	$\eta_j$ to an interval of the form \begin{equation*}\label{E:type1}
		\big[\al_{k-1}+(k-1)\,,\, \al_k+(k-1)\big]\quad (k\in
		[n+1],~\al_0:=\vtheta_0,~ \al_{n+1}:=\vtheta_1) \end{equation*} is the
	reparametrization of $\ga_j|_{[b_j^{k-1},a_j^{k}]}$ by the argument
	$\theta\in [\al_{k-1},\al_k]$, where $b^j_0:=0$ and $a_j^{n+1}:=1$. The
	restriction of $\eta_j$ to an interval of the form
	\begin{equation*}\label{E:type2} [\al_k+(k-1)\,,\, \al_k+k]\quad (k\in [n])
	\end{equation*} is the reparametrization of $\ga_j|_{[a_j^k,b_j^k]}$
	proportional to arc-length. Let $\eta_s\colon [\vtheta_0,\vtheta_1+n]\to \C$
	be given by \begin{equation*}
		\eta_s(t)=(1-s)\eta_0(t)+s\eta_1(t)\quad (s\in [0,1]).  \end{equation*} 

A straightforward computation shows that the radius of curvature
$\rho_s=\frac{1}{\ka_{\ga_s}}$  satisfies \begin{equation*}
	\rho_s=(1-s)\rho_0+s\rho_1\in \big[\tfrac{1}{\ka_1},+\infty\big)\quad (s\in
		[0,1]) \end{equation*} in the interior of intervals of the first type.
	The restriction of $\eta_s$ to an interval of the second type is a
	parametrization of a (possibly degenerate) line segment parallel to
	$e^{i\al_k}$. Thus $\eta_s$ satisfies (i) and (ii). The desired homotopy
	$s\mapsto \ga_s$ is obtained by reparametrizing $\eta_s$  proportionally to
	arc-length. Furthermore: \begin{enumerate} \item [(iii)] If
			$\ga_0(0)=p=\ga_1(0)$, then $\ga_s(0)=p$ for all $s\in [0,1]$;
		similarly at $t=1$.  \item [(iv)]  Let
			$I_s=\theta_{\ga_s}^{-1}(\se{\vtheta_0})$. If $\ga_j|_{I_j}$ is a
			line segment of length $>L$ $(j=0,1)$ and slope $ \varrho $, 
			then $\ga_s|_{I_s}$ is also a
			line segment of length $>L$ and slope $ \varrho $ for all $s\in [0,1]$; similarly for
			$\vartheta_1$.\qed \end{enumerate} \end{cons}

\begin{defn}\label{D:Hd} Let $\sig_1\prec\dots\prec\sig_m$ be sign strings,
	$n=\abs{\sig_m}$ and $\de_j>0$ ($j\in [n]$) satisfy $\de_{j+1}>2\de_{j}$ for
	all $j\in [m-1]$. Define $H_d\subs \R^n$ to consist of all
	$x=(x_1,\dots,x_n)\in \R^n$ such that: \begin{enumerate} \item [(i)] There
			exist $k_1,k_2\in [n]$ such that $\sig_m(k_2)=-\sig_m(k_1)$ and
		$\sig_m(k_i)x_{k_i}>0$ ($i=1,2$).		\item [(ii)] For each $j\in
			[m]$, if $k_1<\dots<k_l$ are all the indices in $[n]$ such that
			$\abs{x_k}< \de_j$ (resp.~$\abs{x_k}\leq 2\de_j$), then $\sig_j$ is
			the reduced string of $\tau\colon [l]\to \se{\pm}$,
			$\tau(i)=\sig_m(k_i)$.

	\end{enumerate} \end{defn} This space is weakly contractible for any choice
	of $\sig_j$, $\de_j$ because it is weakly homotopy equivalent to the space
	\begin{equation*}
	X_{(d,\sig_1,\sig_1,\dots,\sig_m,\sig_m)}
	\end{equation*} described in \dref{D:dnested}; see \pref{P:dnested} and
	\rref{R:strict}. (Here each $ \sig_j $ appears twice because it is involved 
	in two inequalities in (ii), viz., one for $ \de_j $ and the other for  
	$ 2\de_j $.)

\begin{defn}\label{D:Ed} Let $\ka_0\in (\frac{1}{2},1)$, $\sig_1\prec \dots\prec
	\sig_m$ be sign strings, $n=\abs{\sig_m}$ and $\de_j>0$ ($j\in [m]$) satisfy
	$\de_{j+1}>2\de_{j}$ for all $j\in [m-1]$. Define $E_d$ to be the subspace
	of $\sr M(Q)\times R(Q)$ consisting of all $(\ga,\vphi)$ for which there
	exist $0=t_0<\dots<t_{2n+1}=1$ and $(x_1,\dots,x_n)\in H_d$ such that:
\begin{enumerate} \item [(i)] $\ga|_{[t_{2(k-1)},t_{2k-1}]}$ ($k\in [n+1]$) is
		an arc of circle of radius $\geq \frac{1}{\ka_0}$ and amplitude $<\pi$;
	\item [(ii)] $\ga|_{[t_{2k-1},t_{2k}]}$ ($k\in [n]$) is a straight
		line segment of length greater than 8 and \begin{equation*}
			\theta_{\ga}\big([t_{2k-1},t_{2k}]\big)=\se{\vphi_{\sig_m(k)}+x_k}.
	\end{equation*} \end{enumerate} \end{defn} The arcs in condition (i) are
	allowed to be degenerate. Observe that if $(\ga,\vphi)\in E_d$, then $\ga$
	is diffuse and $(\vphi,\de_j)$-quasicritical of type $\sig_j$ for each $j\in
	[m]$ (for $\sig_j$ and $\de_j$ as above). Here $R(Q)$ is the open interval
	described in \dref{D:Vs}.

\begin{lem}\label{L:Ed} The space $E_d$ defined above is weakly contractible.
\end{lem} \begin{proof} Let $p\colon E_d\to H_d\times R(Q)$ be given by
	$p(\ga,\vphi)=(x,\vphi)$, where $x=(x_1,\dots,x_n)$ is as in condition (ii)
	of \dref{D:Ed}. Fix $(x,\vphi)$ and $(\ga_0,\vphi)\in p^{-1}(x,\vphi)$; let
	$t_0'=0<\dots<t_{2n+1}'=1$ be as in \dref{D:Ed} for $(\ga_0,\vphi)$.
	Given  $\ga=\ga_1\in p^{-1}(x,\vphi)$ and $t_0<\dots<t_{2n+1}$ as above,
	apply \cref{C:locconvex} to the restrictions $\ga|_{[t_{\nu-1},t_\nu]}$ and
	$\ga|_{[t'_{\nu-1},t'_\nu]}$ for each $\nu\in [2n+1]$ to obtain a homotopy
	$s\mapsto \ga_s$ joining $\ga_0$ to $\ga_1$. The validity of (iii) and (iv)
	of \cref{C:locconvex} guarantees that $(\ga_s,\vphi)\in E_d$ for all $s\in
	[0,1]$.  Therefore, the fiber $p^{-1}(x,\vphi)$ is either contractible or
	empty, for any $(x,\vphi)\in H_d\times R(Q)$.
	
	For $x=(x_1,\dots,x_n)\in H_d$, let \begin{alignat*}{9}
		\eps_0(x)&=\min\set{\abs{x_k}}{\sig_m(k)x_k>0,~k\in [n]}, \\
		\eps_j(x)&=\min\set{\de_j-\abs{x_k}}{\abs{x_k}<\de_j,~k\in [n]}\quad
		(j\in [m]), \\ \eps(x)&=\min\se{\eps_0(x),\dots,\eps_m(x)}.
	\end{alignat*} Then the open ball $B_\eps(x)$ is convex and $B_\eps(x)\subs
	H_d$ for any $\eps\in (0,\eps(x))$. We claim that $p$ has a section over
	$B_\eps(x)\times R(Q)\subs H_d\times R(Q)$ for any $x\in H_d$ and $\eps\in
	(0,\eps(x))$. In particular, $p$ is surjective. Together with
	contractibility of the fibers and \pref{P:DolTho}, this will imply that $p$
	is a quasifibration, and hence that $E_d$ is weakly contractible.
	
	Let $x\in H_d$ and $\eps\in (0,\eps(x))$. For each $y=(y_1,\dots,y_n)\in
	B_\eps(x)$, consider the (unique) curve $\eta^y\colon [0,1]\to \C$ of the
	form  
	\begin{equation*}
	c\underbrace{lc\dots lc}_{n} \end{equation*}
	such that each arc of circle has radius $ \frac{1}{\ka_0} $,
	$\Phi_{\eta^y}(0)=(0,1)\in \C\times \Ss^1$, $\ta_{\eta^y}(1)=z$ and
	$\theta_{\eta^y}= \vphi_{\sig_m(k)}+y_k$ over the $k$-th line segment, which
	we set to be of length 10 for all $k\in [n]$. Then $(\eta^y,\vphi)$
	satisfies all of the conditions required of elements of $E_d$, except that
	$\eta^y(1)$ may not agree with $q\in \C$ as it should. 
	
	To correct this, choose $k_1,\,k_2\in [n]$ such that
	\begin{equation*}
	\sig_m(k_1)=\ty{+},\
		x_{k_1}>0,~\sig_m(k_2)=\ty{-},\ x_{k_2}<0; \end{equation*} such indices
	exist by condition (i) in the definition of $H_d$. Moreover, by the choice
	of $\eps$, $y_{k_1}>0$ and $y_{k_2}<0$ for any $y\in B_\eps(x)$. Let $t
	\colon B_\eps(x)\to [0,1]$ be a continuous function such that
	$\ta_{\eta^y}(t(y))=e^{i\vphi}$. Then a section $(y,\vphi)\mapsto
	(\ga^y,\vphi)$ of $p$ over
	$B_\eps(x)\times R(Q)$ can be obtained by increasing the length of the $k$-th line
	segment to $l_{k}\geq 10$ for $k=k_1,k_2$ and grafting a straight line
	segment of length $l_0\geq 0$ at $t(y)$. More precisely, 
	the origin $0\in \C$ lies in the interior of the triangle whose vertices are
	$e^{i\vphi}$, $ie^{i(\vphi+y_{k_1})}$ and $-ie^{i(\vphi+y_{k_2})}$.
	Therefore,  any complex number can be written as
	\begin{equation*}
		a_0e^{i\vphi}+a_1ie^{i(\vphi+y_{k_1})}-a_2ie^{i(\vphi+y_{k_2})}\ \
		\text{for some $a_0,a_1,a_2>10$.} \end{equation*} Consequently the
	lengths $l_0,l_{k_1},l_{k_2}$ can be (continuously) chosen to achieve that
	$\ga^y(1)=q$.  \end{proof}

Next we establish a version of \lref{L:Ed} for condensed curves, beginning with
the following.

\begin{lem}\label{L:condensedaxes} Suppose that $(\ga,\vphi)\in \sr
	V_{(c,\sig)}\subs \sr N(Q)$ for some sign string $\sig$. Then there exists a
	critical curve $\eta\in \sr M(Q)$ of type $\sig$ for which
	$\bar\vphi^\eta=\vphi$ \tup(with $\bar\vphi^{\eta}$ as defined in
	\eqref{E:average}\tup).  \end{lem} \begin{proof} Let $n=\abs{\sig}$,
	$J_1<\dots<J_n$ and $I_k\subs J_k$ be as in \dref{D:quasicritical}. Deform
	each $\ga|_{I_k}$ to obtain a curve $\eta$ such that for each $k\in [n]$,
	$\theta_{\eta}(t_k)=\vphi_{\sig(k)}$ for at least one $t_k\in I_k$, but
	$\theta_\eta([0,1])\subs [\vphi_-,\vphi_+]$ still holds.
\end{proof}

\begin{defn}\label{D:Hc} Let $\sig_1\prec\dots\prec\sig_m$ be sign strings,
	$n=\abs{\sig_m}$ and $\de_j>0$ ($j\in [m]$) satisfy $\de_{j+1}>2\de_{j}$ for
	all $j\in [m-1]$. Define $H_c\subs \R^n$ to consist of all
	$x=(x_1,\dots,x_n)\in \R^n$ such that: \begin{enumerate} \item [(i)]
			$\sig_m(k)x_k< 0$ for each $k\in [n]$; \item [(ii)] For each $j\in
				[m]$, if $k_1<\dots<k_l$ are all the indices in $[n]$ such that
				$\abs{x_k}< \de_j$ (resp.~$\abs{x_k}\leq 2\de_j$), then $\sig_j$
				is the reduced string of $\tau\colon [l]\to \se{\pm}$,
				$\tau(i)=\sig_m(k_i)$.  \end{enumerate} \end{defn} Again, $H_c$
is weakly contractible for any choice of $\sig_j$, $\de_j$ by \pref{P:nested}
and \rref{R:strict}, since it has the same weak homotopy type as
$ X_{(\sig_1,\sig_1,\dots,\sig_m,\sig_m)}$.

\begin{defn}\label{D:Ec} Let $\sig_1\prec\dots\prec\sig_m$ be sign strings,
	$n=\abs{\sig_m}$ and $\de_j>0$ ($j\in [m]$) satisfy $\de_{j+1}>2\de_{j}$ for
	all $j\in [m-1]$. Let $J(Q)$ denote the open interval consisting of all
	$\vphi\in \R$ such that $\sr M(Q)$ contains critical curves $\eta$ of type
	$\sig_m$ with $\bar\vphi^\eta=\vphi$ (cf.~\cite{SalZueh1}, Corollary 5.7).
	For $S$ a closed subinterval of $J(Q)$, define  $E_c\subs \sr M(Q)\times S$
	as in \dref{D:Ed}, replacing $R(Q)$ by $S$ and $H_d$ by $H_c$.  \end{defn}

\begin{lem}\label{L:Ec} Let $S$ be a closed subinterval of $J(Q)$. Then for all
	sufficiently small $\de_m>0$, the space $E_c$ defined above is weakly
	contractible.  \end{lem} \begin{proof} It was established in the proof of
	Proposition 5.3 of \cite{SalZueh1} that there exists a critical curve
	$\eta\in \sr M(Q)$ of type $\sig_m$ with $\bar\vphi^\eta=\vphi$ if and only
	if $\vphi\in \ol{R(Q)}$ and $q$ lies in the open region to the right of the
	tangent $T_\vphi$ of direction $ie^{i\vphi}$ to a certain circle. The set of
	all such $\vphi$ is the open interval $J(Q)$, and if $S\subs J(Q)$ is a
	closed interval, then there exists a positive lower bound for the distance from $q$
	to $T_\vphi$ for $\vphi\in S$.
	
	The proof of the present lemma is analogous to that of \lref{L:Ed} except
	for the last paragraph. Retaining the notation used there, choose
	$k_1,k_2\in [n]$ such that $\sig_m(k_1)=\ty{-}$, $\sig_m(k_2)=\ty{+}$ and
	$\abs{y_{k_i}}<\de_1$ ($i=1,2$) for all $y\in B_\eps(x)$, where
	$\eps<\eps(x)=\min\se{\eps_1(x),\dots,\eps_m(x)}$ (and $ \eps_0(x) $ is now 
	undefined). By the preceding
	observations,
	if $\de_m>0$ is sufficiently small, then $q$ lies to the right of  the line
	through $\eta^y(1)$ of direction $ie^{i\vphi}$. By further reducing
	$\de_m>0$ if necessary, it can be guaranteed that $q$ lies in the cone with
	vertex at $\eta^y(1)$ and sides parallel to \begin{equation*}
		i\exp\big(i(\vphi+y_{k_1})\big)\text{\ \ and\ \
		}-i\exp\big(i(\vphi+y_{k_2})\big), \end{equation*} but does not lie in
	the triangle with vertices \begin{equation*}
	\eta^y(1),\ \
		\eta^y(1)+10i\exp\big(i(\vphi+y_{k_1})\big)\text{\ \ and\ \
		}\eta^y(1)-10i\exp\big(i(\vphi+y_{k_2})\big) \end{equation*} for any
	$\vphi\in S$,~$y\in B_\eps(x)$. This implies that $q$ can be written as
	\begin{equation*}
		\eta^y(1)+a_1ie^{i(\vphi+y_{k_1})}-a_2ie^{(\vphi+y_{k_2})}\ \ \text{for some
		$a_1,a_2>10$.} \end{equation*} A section $(y,\vphi)\mapsto (\ga^y,\vphi)$
	for $p$ over $B_\eps(x)\times S$ can thus be obtained by increasing the lengths
	$l_{k_1},l_{k_2}$ of the line segments of $\eta^y$ to ensure that
	$\ga^y(1)=q$.  \end{proof}

The proof of \pref{P:good} is obtained by assembling the results of this
section.

\begin{proof}[Proof of \pref{P:good}] It suffices to show that each of $\sr
	V_{(\sig_1,\dots,\sig_m)},~\sr V_{(c,\sig_1,\dots,\sig_m)}$ and $\sr
	V_{(d,\sig_1,\dots,\sig_m)}$ is weakly contractible. By \lref{L:diffusing},
	the case of $\sr V_{(\sig_1,\dots,\sig_m)}$ can be reduced to that of $\sr
	V_{(d,\sig_1,\dots,\sig_m)}$. Let $k\geq 0$ and $g\colon \Ss^k\to \sr V$,
	$g(p)=(\ga^p,\vphi^p)$, be a continuous map, where $\sr V=\sr
	V_{(c,\sig_1,\dots,\sig_m)}$ or $\sr V=\sr V_{(d,\sig_1,\dots,\sig_m)}$. 
	
	In the former case, let $S=\set{\vphi^p\in \R}{p\in \Ss^k}$. By
	\lref{L:condensedaxes}, $S$ is a closed subinterval of $J(Q)$. By
	\lref{L:stretching} and \lref{L:pulleys}, $g$ can be deformed within $\sr
	V_{(c,\sig_1,\dots,\sig_m)}$ to have image contained in $E_c$, with
	$\de_m>0$ as small as desired. Hence $g$ is nullhomotopic by \lref{L:Ec}.
	
	In the latter case, \lref{L:stretchingd} and
	\lref{L:Ed} immediately imply that $g$ is nullhomotopic.  \end{proof}

\begin{cor}\label{C:main} Let $\tau$ be a top sign string for $\sr M(Q)$ and
	$n=\abs{\tau}$. If there exist critical curves of type $-\tau$ in $\sr
	M(Q)$, then $\sr N(Q)\home \E\times \Ss^{n-1}$. Otherwise $\sr N(Q)\home
	\E$, for $\E$ the separable Hilbert space.  \end{cor}
Observe that nothing is being asserted yet about the topology of $\sr M(Q)$.
\begin{proof} Immediate from \cref{C:goodcovers1}, \pref{P:combinatorics} and
	\pref{P:good}.  \end{proof}


\section{Homotopy equivalence of $\sr M(Q)$ and a
sphere}\label{S:generator}

\begin{lem}\label{L:homotopyequiv} Suppose that $\pm \tau$ are both top sign
	strings for $\sr M(Q)$, where $\abs{\tau}=n$. If $f\colon \Ss^{n-1}\to \sr
	M(Q)$ and $g\colon \sr M(Q)\to \Ss^{n-1}$ satisfy $\deg(gf)=1$, then $f$ and
	$g$ are homotopy equivalences. In particular, $\sr M(Q)$ is homeomorphic to
	$\E\times \Ss^{n-1}$ and $f$ represents a generator of $\pi_{n-1}\sr M(Q)$.
\end{lem}

\begin{proof} According to \lref{L:N=M}, \lref{L:Moore} and \cref{C:main}, under
	the present hypothesis $\sr M(Q)$ is either weakly contractible or a
	homology sphere of dimension $n-1$. The fact that $\deg(gf)=1$ implies that
	the latter must hold, and that $f$ and $g$ induce isomorphisms on all
	(co)homology groups. 

When $n=2$,  it follows directly from \lref{L:N=M} that all higher homotopy
groups of $\sr M(Q)$ are trivial, so that $f$ and $g$ are weak homotopy
equivalences. 

When $n>2$, $\sr M(Q)$ and $\Ss^{n-1}$ are simply-connected. Passing to mapping
cylinders and applying the relative version of the Hurewicz theorem, we again
conclude that $f$ and $g$ induce isomorphisms on all homotopy groups.

Thus $ \sr M(Q) $ is weakly homotopy equivalent to $ \Ss^{n-1} \iso
\E\times \Ss^{n-1} $. Since a weak homotopy equivalence between Hilbert manifolds is 
homotopic to a homeomorphism (\cite{SalZueh1}, Lemma 1.7), $ \sr M(Q) $ is
actually homeomorphic to $ \E\times \Ss^{n-1} $.  \end{proof}

Our next objective is to show that (under the hypothesis of the lemma) such $f$
and $g$ always exist. In fact, they can be constructed explicitly. 

Briefly, the map $g$ defined below measures the extent to which curves in $\sr
M(Q)$ fail to be critical of type $\tau$. Its definition is a slight variation
of that of the map $h$ in \dref{D:h}; cf.~also Figure \ref{F:pulleys}. 

\begin{cons}\label{C:g} Let $\sr U_\tau\subs \sr M(Q)$ consist of all curves
	which are $(\bar{\vphi}^\ga,\eps)$-quasicritical of type $\tau$ for some
	$\eps\in(0,\frac{\pi}{4})$, where $\bar\vphi^\ga$ is given by
	\eqref{E:average}. Then $\sr U_\tau$ is an open subset of $\sr M(Q)$
	containing the set $\sr C_\tau$ of all critical curves of type $\tau$ by
	\lref{L:criticalisquasi}.
	Moreover, $\sr C_\tau$ is closed in $\sr M(Q)$; here the hypothesis that
	$\tau$ is a \tsl{top} sign string for $\sr M(Q)$ is essential.
	
	Given $\ga\in \sr U_\tau$ and intervals $J_1<\dots<J_n$ satisfying the
	conditions of \dref{D:quasicritical} for the quadruple
	$(\ga,\bar\vphi^\ga,\eps,\tau)$, define \begin{alignat*}{9}
		\al_k(\ga)&=\begin{cases} \sup_{t\in
			J_k}\theta_\ga(t)-\frac{\pi}{2} & \text{ if $\tau(k)=\ty{+}$};
			\\ \inf_{t\in J_k}\theta_\ga(t)+\frac{\pi}{2} & \text{ if
				$\tau(k)=\ty{-}$}; \\ \end{cases}\ (k\in [n])\text{\ \ and\ \
			}\\ \al(\ga)&=\frac{1}{n}\big[\al_1(\ga)+\dots+\al_n(\ga)\big].
		\end{alignat*}
It follows from \lref{L:Js}\?(a) that the maps $\al_k\colon \sr U_\tau\to \R$
are well-defined (i.e., they do not depend on the choice of $\eps$ and the
$J_k$) and continuous; compare \rref{R:welldefined}. Let
\begin{equation*}
\Sig=\set{(x_1,\dots,x_n)\in
	\R^n}{\textstyle{\sum_kx_k=0}}\home \R^{n-1}	\end{equation*} and define
\begin{equation*}
A\colon \sr U_\tau\to
	\Sig,~A(\ga)=\big(\al_1(\ga)-\al(\ga),\dots,\al_n(\ga)-\al(\ga)\big).
\end{equation*} 
\begin{lem}\label{L:}
	Let $ \ga\in \sr U_\tau $. Then $ A(\ga)=0 $ if and only if $ \ga\in \sr
	C_\tau $.
\end{lem}
\begin{proof}
	It is clear that $A(\ga)=0$ if $\ga\in \sr C_\tau$. Conversely, if $
A(\ga)=0 $ then $ \al_1(\ga)=\dots=\al_n(\ga) $. Since there exist $ k,l\in [n] $
such that 
\begin{equation*}
	\sup_{t\in J_k}\theta_\ga(t)=\sup_{t\in [0,1]}\theta_\ga(t)
	\text{\ \,and\ \,} 
	\inf_{t\in J_l}\theta_\ga(t)=\inf_{t\in [0,1]}\theta_\ga(t),
\end{equation*}
the equality of $ \al_k(\ga) $ and $\al_l(\ga) $ implies that $ \om(\ga)=\pi $,
that is, $ \ga $ is critical (of type $ \tau $). 
\end{proof}

Let $\sr W\subs \sr M(Q)$ be an open set such that 
\begin{equation*}
\sr C_\tau\subs \sr W\subs \ol{\sr W}\subs \sr U_\tau 
\end{equation*} and let
$\la\colon \sr M(Q)\to [0,1]$ be a continuous function such that
$\la^{-1}(1)=\sr M(Q)\ssm \sr W$ and $\la^{-1}(0)$ is a neighborhood of $\sr
C_\tau$. Let $r\colon \Sig \to \Ss^{n-1}$ denote the map which collapses the
complement of $B_1(0)\cap \Sig$ to a single point, which will be identified with
the south pole $-N\in \Ss^{n-1}$, with $0$ mapping to the north pole $N$.
Finally, define \begin{equation}\label{E:g} g\colon \sr M(Q)\to
	\Ss^{n-1},~g(\ga)=\begin{cases}
		r\Big((1-\la(\ga))A(\ga)+\la(\ga)\frac{A(\ga)}{\abs{A(\ga)}}\Big) &
		\text{ if $\ga\in \ol{\sr W}$;} \\ -N & \text{ if $\ga\nin \sr W$.}
	\end{cases} \end{equation} Observe that $g^{-1}(N)=\sr C_\tau$.\qed
	\end{cons}

We shall now construct a generator $[f]$ for $\pi_{n-1}\sr M(Q)$.

\begin{cons}\label{C:f} Let $C$ denote the cube
	$\big[-\frac{\pi}{2},\frac{\pi}{2}\big]^n\subs \R^n$, $\bd C$ its boundary
	and \begin{equation*}
	S=\set{(x_1,\dots,x_n)\in
		C}{\text{$x_{k_1}=\textstyle{\frac{\pi}{2}}$ and
		$x_{k_2}=-\textstyle{\frac{\pi}{2}}$ for some $k_1,\,k_2\in [n]$}}.
	\end{equation*} Note that $S\subs \bd C$ is the complement of the union of
	the open stars of the opposite vertices of $C$ whose coordinates are given
	by $x_k=\frac{\pi}{2}$ and $x_k=-\frac{\pi}{2}$ for each $k\in [n]$,
	respectively. (These vertices are labeled $\ty{+++}$ and $\ty{---}$ in
	Figure \ref{F:spheres}\?(b).) We shall identify $\bd C$ with $\Ss^{n-1}$ and
	$S$ with its equator $\Ss^{n-2}$ when convenient.
	
	\begin{figure}[ht] \begin{center} \includegraphics[scale=.35]{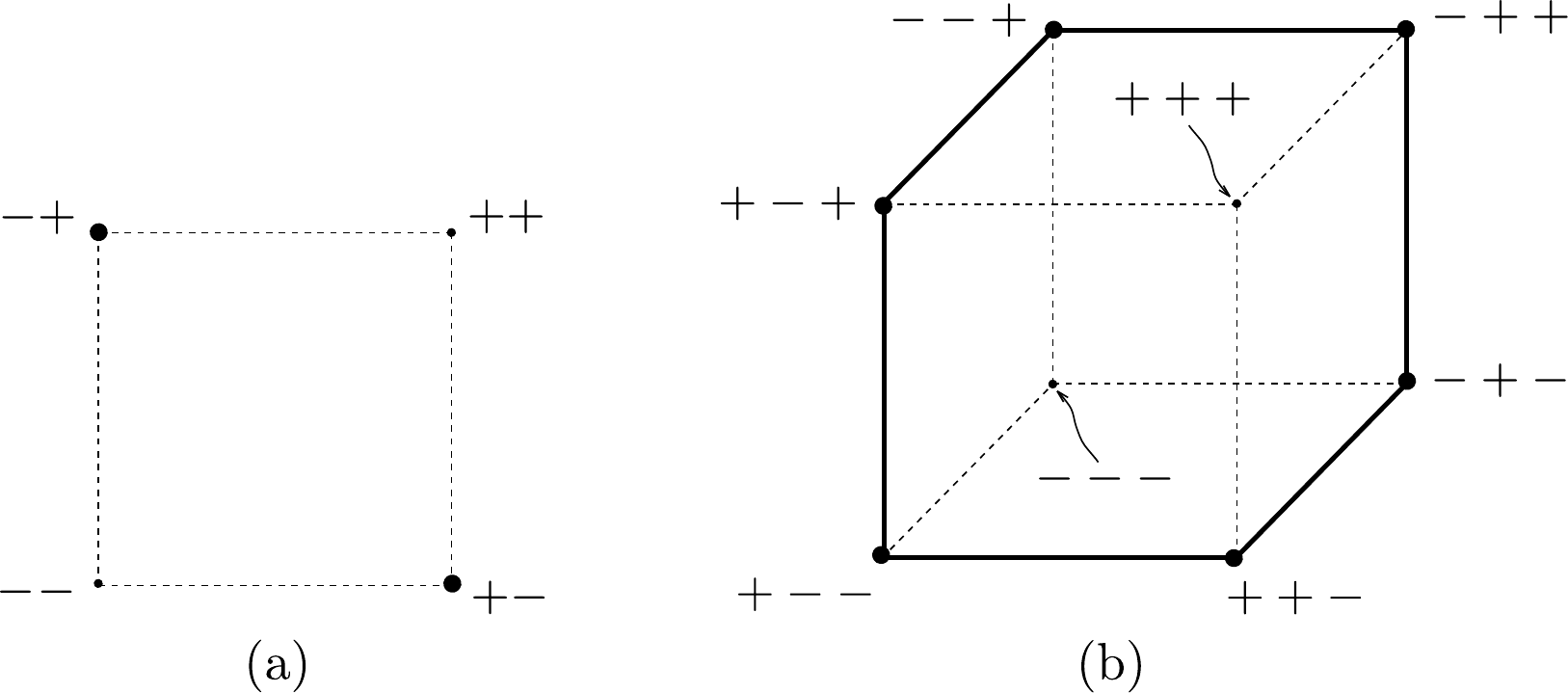}
			\caption{The subset $S\home \Ss^{n-2}$ of $C$ (in thick) for $n=2$
		and 3.} \label{F:spheres} \end{center} \end{figure}
	
	 To simplify the explanation, let us assume first that there exists
	 $\vphi\in \R$ such that it is possible to find critical curves
	 $\ga_1,\ga_2\in \sr M(Q)$ of types $\tau$ and $-\tau$ such that
	 $\bar\vphi^{\ga_1}=\vphi=\bar\vphi^{\ga_2}$. (It is not hard to show that
	 this is always the case if $n$ is even, but this fact will not be used.)
	 This implies that it is possible to find a critical curve $\ga\in \sr M(Q)$
	 of type $\sig$ with $\bar\vphi^\ga=\vphi$ for any $\sig$ with
	 $\abs{\sig}\leq n$. Let $\ka_0,\,\de\in (0,1)$. Let $T=S\times [-\de,\de]$
	 be identified with a tubular neighborhood of $\Ss^{n-2}\equiv S$ in
	 $\Ss^{n-1}\equiv \bd C$, with a point $(x,s)\in T$ lying in the hemisphere
	 $H_{\sign(s)}$ at distance $\abs{s}$ from $\Ss^{n-2}$ and $x\in \Ss^{n-2}$
	 realizing this distance (here $H_{\pm}$ are the two hemispheres bounded by
	 $\Ss^{n-2}$). 
	 
	 For each $(x,s)\in T$, let $\eta^{(x,s)}$ denote the unique curve of the
	 form
	 \begin{equation*}
	 \underbrace{c\dots c}_{n+1} \end{equation*}
	 such that $\Phi_{\eta^{(x,s)}}(0)=(0,1)\in \C\times \Ss^1$,
	 $\ta_{\eta^{(x,s)}}(1)=z$ and \begin{equation}\label{E:f}
		 \theta_{\eta^{(x,s)}}= \vphi+(1+s)x_k \end{equation} at the point where
	 the $k$-th circle is concatenated with the $(k+1)$-th circle, for all $k\in
	 [n]$, where each of the circles has radius $\frac{1}{\ka_0}$. Observe that
	 for all $x\in S$, $\eta^{(x,s)}$ is critical, condensed or diffuse
	 according as $s=0$, $s<0$ or $s>0$, respectively. 

	The curves $\eta^{(x,s)}$ do not in general satisfy $\eta^{(x,s)}=q$, but
	this can be corrected as follows. Because of the hypothesis on $\vphi$, if
	$\ka_0\in (0,1)$ is sufficiently close to 1 and $\de\in (0,1)$ sufficiently
	close to $0$, then \begin{equation*}
		\big\lan\eta^{(x,s)}-q\,,\,e^{i\vphi}\big\ran<0\text{\ \ for all $x\in
		S,\,s\in [-\de,\de]$.} \end{equation*} For fixed $x\in S$, choose
	$t_0,t_1,t_2\in [0,1]$ such that
	$\theta_{\eta^{(x,0)}}(t_i)=\vphi,\,\vphi+\frac{\pi}{2}$ and
	$\vphi-\frac{\pi}{2}$ for $i=0,1,2$, respectively. By grafting line segments
	at $\eta^{(x,0)}(t_i)$ ($i=0,1,2$), a curve $\ga^{(x,0)}$ with
	$\ga^{(x,0)}(1)=q$ as desired is obtained. Clearly, the same procedure will
	work in a neighborhood of $(x,0)$, for the same choices of $t_i$. Using a
	partition of unity and reducing $\de>0$ further if necessary, this yields a
	family $\ga^{(x,s)}\in \sr M(Q)$ ($x\in S$, $s\in [-\de,\de]$). The chosen
	open sets, the corresponding $t_i$ and the lengths of the segments do not
	change the homotopy class of $f$ and are irrelevant for the calculation of
	$\deg(gf)$.
	
	The correspondence $(x,s)\mapsto \ga^{(x,s)}\in \sr M(Q)$ can be extended to
	a map $f\colon \Ss^{n-1}\to \sr M(Q)$ through nullhomotopies of the families
	$\ga^{(x,\de)}$ and $\ga^{(x,-\de)}$ $(x\in S)$ within $\sr U_d$ and $\sr
	U_c$, respectively. The latter two sets are contractible by Theorems 3.3 and
	4.19 of \cite{SalZueh1}. This completes the construction of $f$ under the
	initial assumption on $\vphi$. 
	
	In the general case, let $\ga_{\pm \tau}\in \sr M(Q)$ be arbitrary critical
	curves of type $\pm \tau$, and set $\vphi_{\pm \tau}=\bar{\vphi}^{\ga_{\pm
	\tau}}\in \R$. Let $U_{\pm\tau}$ denote the open star in $S$ of the vertices
	$p=\frac{\pi}{2}(\tau(1),\dots,\tau(n))$ and $-p$, respectively. Since
	$\ol{U}_\tau\cap \ol{U}_{-\tau}=\emptyset$, we can find a continuous
	function $S\to \R$, $x\mapsto \vphi^x$, taking values in the closed interval
	with endpoints $\vphi_{\pm\tau}$, such that $\vphi^x=\vphi_{\pm \tau}$ if
	$x\in U_{\pm \tau}$. By Proposition 5.3 in \cite{SalZueh1}, if
	$\abs{\sig}<n$, then there exist critical curves $\ga$ of type $\sig$ with
	$\bar{\vphi}^\ga=\psi$ for all $\psi$ in this interval. Hence, the preceding
	definition of $\ga^{(x,s)}$ works for every $x\in S$ if $\vphi$ is replaced
	by $\vphi^x$ in \eqref{E:f}.\qed \end{cons}

\begin{lem}\label{L:degree} Suppose that $\pm \tau$ are both top sign strings
	for $\sr M(Q)$, where $\abs{\tau}=n$. Let $g\colon \sr M(Q)\to \Ss^{n-1}$
	and $f\colon \Ss^{n-1}\to \sr M(Q)$ be the maps described in Constructions
	\ref{C:g} and \ref{C:f}. Then $\deg(gf)=\pm 1$.  \end{lem} \begin{proof} Let
	$N$ denote the north pole of $\Ss^{n-1}$ and
	$p=\frac{\pi}{2}(\tau(1),\dots,\tau(n))\in S\subs \bd C\equiv \Ss^{n-1}$.
	Then $(gf)^{-1}(N)=\se{p}$,
	hence the result will follow if $gf$ is a homeomorphism near $p$. By
	Brouwer's invariance of domain, it suffices to show that $gf$ is injective
	on a neighborhood of $p$ in $\bd C$. Finally, by the definition of $f|_T$,
	it actually suffices to show that $gf$ is injective on a neighborhood of $p$
	in $S$. Let $U\subs S$ be an open set containing $p$ such that
	$\la(U)=\se{0}$ and $A(U)\subs B_1(0)$, where $A$ and $\la$ are as in
	\eqref{C:g}. For $x,\,\bar x\in U$, \begin{equation*}
		\al_k(\ga^x)-\al_k(\ga^{\bar x})=x_k-\bar x_k\text{\ $(k\in [n])$\ \
		and\ \ }\al(\ga^x)-\al(\ga^{\bar x})=\frac{1}{n}\sum_{\nu=1}^n(x_{\nu}-\bar
		x_{\nu}).  \end{equation*} Therefore, $A(\ga^x)=A(\ga^{\bar x})$ if and only
	if $(x-\bar x)$ is a multiple of $(1,1,\dots,1)$. In a small neighborhood of
	$p$ in $S$, this occurs if and only if $x=\bar{x}$. Thus $gf|_S$ is
	injective near $p$, and $\deg(gf)=\pm 1$.  \end{proof}

Obviously, it can be achieved that $\deg(gf)=+1$ by composing $g$ with a reflection
if necessary. The next result is a corollary of \lref{L:homotopyequiv} and
\lref{L:degree}.

\begin{cor}\label{C:N=M2} Let $\pr \colon \sr N(Q)\to \sr M(Q)$ be the
	restriction of the canonical projection of $\sr M(Q)\times \R$ onto $\sr M(Q)$. Then $\pr$ is
	a homotopy equivalence and $\sr M(Q)$ is homeomorphic to $\sr N(Q)$.
\end{cor} \begin{proof} By \lref{L:N=M}, the induced map $\pr_{\ast}\colon
	H_\ast(\sr N(Q))\to H_\ast (\sr M(Q))$ is surjective. Since $\sr M(Q)$ and
	$\sr N(Q)$ are either simultaneously contractible or simultaneously homotopy
	equivalent to a sphere, $\pr_\ast$ must actually be an isomorphism. We
	conclude that $\pr$ is a homotopy equivalence using the same argument as in
	the proof of \lref{L:homotopyequiv}.  \end{proof}

The proof of the main theorem (stated in the introduction) is now
straightforward.  
\begin{thm}\label{T:main} 
	Let $Q=(q,z)\in \C\times \Ss^1$, $z\neq -1$. Then
	$\sr M(Q)\home \E\times \Ss^{2k} \text{ or }\E\times \Ss^{2k+1}\ (k\geq 0)$
	for $q$ in the open region intersecting the ray from 0 through $1+z$ and
	bounded by the three circles \begin{equation*} \begin{cases}
			C_{4k+4}(iz-i)\text{ \,and \,}C_{4k+2}(\pm (i+iz)),\text{\ \ or}\\
			C_{4k+4}(i-iz)\text{ \,and \,}C_{4k+6}(\pm(i+iz)),\text{\ \
			respectively} \end{cases} (\text{see \fref{F:shadesofgrey}}).
	\end{equation*} If $q$ does not lie in
	the closure of any of these regions, then $\sr M(Q)\home \E$. If $q$ lies on
	the boundary of one of them, then $\sr M(Q)\home \sr
	M\big((q-\de(1+z),z)\big)$ for all sufficiently small $\de>0$.  
\end{thm}

\begin{proof} Proposition 5.3 of \cite{SalZueh1} describes precisely when $\sr
	M(Q)$ contains critical curves of any given type. If $\sr M(Q)$ does not
	contain any critical curves (or, equivalently, if it does not admit a top
	sign string), then $\sr M(Q)\home \E$ or $\E\times \Ss^0$ according as $\sr
	U_c$ is empty or not, as described in Theorem 6.1 of \cite{SalZueh1}. If it
	does admit a top sign string, then we conclude from \cref{C:main} and
	Proposition 5.3 of \cite{SalZueh1} that the theorem holds if $\sr M(Q)$ is
	replaced by $\sr N(Q)$ in the statement. But $\sr M(Q)\home \sr N(Q)$ by
	\cref{C:N=M2}.  \end{proof}

\subsection*{Acknowledgements} The first author is partially supported by grants
from \tsc{capes}, \tsc{cnpq} and \tsc{faperj} (Brazil). The second author
gratefully acknowledges the hospitality of PUC--Rio, IME--USP and UnB
as well as the financial support of \tsc{cnpq} and \tsc{fapesp}
in the form of post-doctoral fellowships.

\vfill\eject
\providecommand{\bysame}{\leavevmode\hbox to3em{\hrulefill}\thinspace}
\providecommand{\MR}{\relax\ifhmode\unskip\space\fi MR }
\providecommand{\MRhref}[2]{%
  \href{http://www.ams.org/mathscinet-getitem?mr=#1}{#2}
}
\providecommand{\href}[2]{#2}

\bigskip

\vspace{12pt} \noindent{\small \tsc{Departamento de Matem\'atica, Pontif\'icia
	Universidade Cat\'olica do Rio de Janeiro (PUC-Rio)\\
Rua Marqu\^es de S\~ao Vicente 225, G\'avea,
Rio de Janeiro, RJ 22451-040, Brazil}}\\
\noindent{\ttt{saldanha@puc-rio.br}}

\vspace{12pt} \noindent{\small \tsc{Universidade Federal de Santa Maria, \\
Av. Roraima n 1000, Cidade Universit\'aria, Bairro Camobi,
Santa Maria,  RS 97105-900, Brazil}\\
\noindent{\ttt{pzuehlke@protonmail.com}}

\vspace{2pt}
\end{document}